\documentclass[onefignum,onetabnum]{siamart171218}
\pdfoutput=1
\usepackage[top=1.4in,bottom=1.275in,left=1.4in,right=1.4in]{geometry}

\usepackage[utf8]{inputenc}
\usepackage[T1]{fontenc}
\usepackage[english]{babel}\usepackage{csquotes}\usepackage{microtype}\usepackage{hyphenat}\usepackage[expproduct=\cdot,binary-units=true]{siunitx}

\usepackage{amsfonts}
\usepackage{amssymb}
\usepackage{amscd}
\usepackage{mathtools}
\usepackage[mathscr]{eucal}
\usepackage{calligra}
\DeclareMathAlphabet{\mathcalligra}{T1}{calligra}{m}{n}
\DeclareFontShape{T1}{calligra}{m}{n}{<->s*[1.1]callig15}{}
\usepackage{bm}
\usepackage{autonum}\usepackage[nice]{nicefrac}

\usepackage{tikz}
\usepackage{pgfplots}
\usepackage{pgfplotstable}
\pgfplotsset{compat=1.5}

\usepackage{comment}
\usepackage{graphicx}
\usepackage[colorinlistoftodos,prependcaption,textsize=small]{todonotes}

\usepackage{cleveref}
\crefname{equation}{}{}
\usepackage{csvsimple}
\usepackage{booktabs}
\usepackage{environ}
\usepackage{bookmark}

\newtheorem{remark}{Remark}[section]

\let\inf\relax \DeclareMathOperator*\inf{\vphantom{p}inf}
\let\min\relax \DeclareMathOperator*\min{\vphantom{p}min}
\let\max\relax \DeclareMathOperator*\max{\vphantom{p}max}
\let\subset\relax \DeclareMathOperator{\subset}{\subseteq}

\DeclareMathOperator*{\supp}{supp}
\DeclareMathOperator*{\argmin}{arg\,min}

\let\div\relax \DeclareMathOperator{\div}{div}

\let\tilde\widetilde
\newcommand{\sspace}{\hspace{0.25pt}}
\newcommand{\msspace}{\hspace{-0.25pt}}

\newcommand{\weaksol}[1]{#1}
\newcommand{\fesol}[1]{{#1}_{h}}
\newcommand{\approxsol}[1]{{\tilde{#1}}_{h}}
\newcommand{\norm}[2][]{\left\|#2\right\|_{#1}}
\newcommand{\honenorm}[1]{\norm[1]{#1}}

\renewcommand{\a}{a}
\newcommand{\asurr}{\tilde{a}}
\newcommand{\poly}{\mathcal{P}}

\newcommand{\R}{\mathbb{R}} \newcommand{\N}{\mathbb{N}} \newcommand{\Z}{\mathbb{Z}}     \newcommand{\mesh}{\mcT} \newcommand{\dd}{\,\mathrm{d}} \newcommand{\bdry}{\partial}    \newcommand{\ext}{{\mathrm{ext}}}

\newcommand{\sfc}{\mathsf{c}}
\newcommand{\sfd}{\mathsf{d}}

\newcommand{\sff}{\mathsf{f}}

\newcommand{\sfu}{\mathsf{u}}

\newcommand{\sfx}{\mathsf{x}}

\newcommand{\sfA}{\mathsf{A}}
\newcommand{\sfB}{\mathsf{B}}
\newcommand{\sfC}{\mathsf{C}}
\newcommand{\sfD}{\mathsf{D}}

\newcommand{\sfM}{\mathsf{M}}
\newcommand{\sfN}{\mathsf{N}}

\newcommand{\mcI}{\mathcal{I}}

\newcommand{\mcP}{\mathcal{P}}

\newcommand{\mcS}{\mathcal{S}}
\newcommand{\mcT}{\mathcal{T}}

\newcommand{\mcZ}{\mathcal{Z}}

\newcommand{\scD}{\mathscr{D}}

\newcommand{\bbL}{\mathbb{L}}

\newcommand{\bbX}{\mathbb{X}}

\newcommand{\bbZ}{\mathbb{Z}}

\newcolumntype{?}{!{\vrule width 1.2pt}}

\makeatletter
\newsavebox{\measure@tikzpicture}
\NewEnviron{scaletikzpicturetowidth}[1]{	\def\tikz@width{#1}	\def\tikzscale{1}\begin{lrbox}{\measure@tikzpicture}		\BODY
	\end{lrbox}	\pgfmathparse{#1/\wd\measure@tikzpicture}	\edef\tikzscale{\pgfmathresult}	\BODY
}
\makeatother

\definecolor{color1}{rgb}{0, 0.4470, 0.7410}
\definecolor{color2}{rgb}{0.8500, 0.3250, 0.0980}
\definecolor{color3}{rgb}{0.9290, 0.6940, 0.1250}
\definecolor{color4}{rgb}{0.4940, 0.1840, 0.5560}
\definecolor{color5}{rgb}{0.4660, 0.6740, 0.1880}
\definecolor{color6}{rgb}{0.3010, 0.7450, 0.9330}
\definecolor{color7}{rgb}{0.6350, 0.0780, 0.1840}

\title{The surrogate matrix methodology: {a priori} error estimation}

\author{Daniel~Drzisga\thanks{Lehrstuhl für Numerische Mathematik, Fakultät für Mathematik (M2), Technische Universität München, Garching bei München (\email{drzisga@ma.tum.de}, \email{keith@ma.tum.de}, \email{wohlmuth@ma.tum.de})}
\and Brendan~Keith\footnotemark[1]
\and Barbara~Wohlmuth\footnotemark[1]}

\begin{document}

\maketitle

\begin{abstract}
We give the first mathematically rigorous analysis of an emerging approach to finite element analysis (see, e.g., Bauer et al. [Appl. Numer. Math., 2017]), which we hereby refer to as \emph{the surrogate matrix methodology}.
This methodology is based on the piece-wise smooth approximation of the matrices involved in a standard finite element discretization.
In particular, it relies on the projection of smooth so-called stencil functions onto high-order polynomial subspaces.
The performance advantage of the surrogate matrix methodology is seen in constructions where each stencil function uniquely determines the values of a significant collection of matrix entries.
Such constructions are shown to be widely achievable through the use of locally-structured meshes.
Therefore, this methodology can be applied to a wide variety of physically meaningful problems, including nonlinear problems and problems with curvilinear geometries. Rigorous \textit{a priori} error analysis certifies the convergence of a novel surrogate method for the variable coefficient Poisson equation.
The flexibility of the methodology is also demonstrated through the construction of novel methods for linear elasticity and nonlinear diffusion problems.
In numerous numerical experiments, we demonstrate the efficacy of these new methods in a matrix-free environment with geometric multigrid solvers.
In our experiments, up to a twenty-fold decrease in computation time is witnessed over the classical method with an otherwise identical implementation.
\end{abstract}



\section{Introduction} \label{sec:introduction}

In the field of computational science, major funding initiatives in North America, Europe, and Asia have thrust high performance computing (HPC) to ascendancy.
In anticipation of future exascale computers, much work in this discipline involves the deep and careful reconstruction of long-established computing practices.
An important characteristic of numerical algorithms aimed for these computers, is the floating-point operation (FLOP) per byte ratio.
In order to achieve optimal performance and power efficiency on future machines, the time spent on FLOPs relative to memory transfer needs to be substantial.

Most traditional finite element softwares assemble global stiffness matrices by looping over elements and adding the corresponding local contributions to the global matrix.
Storing the resulting sparse matrices requires significantly more memory than just storing the degrees of freedom.
However, memory consumption is certainly not the only obstacle at the computational frontier.
Indeed, at such scales, the memory traffic and latency involved in loading indices and entries for matrix vector products (MVPs) also presents critical challenges.

Since iterative solvers only require MVPs, it is not necessary to store all of the nonzeros of the global matrix in memory.
Instead, it is sufficient to compute the nonzero entries on-the-fly, i.e., matrix-free.
Different tactics exist to implement matrix-free methods, but the predominant candidate for low-order finite elements is the
element-by-element approach \cite{Arbenz:2008:IJMME,Bielak:2005:CMES,Carey:1986:CANM,Flaig:2012:LSSC,Rietbergen:1996:IJMME}, wherein local stiffness matrices are multiplied by local vectors and later added to the global solution vector.
These local stiffness matrices may either be stored in memory\textemdash{}which actually requires {more} memory than storing the global matrix\textemdash{}or computed on-the-fly. When using high-order finite elements, the weak forms can be integrated on-the-fly using standard or reduced quadrature formulas \cite{Brown:2010:JSC,Kronbichler:2012:CAF,Ljungkvist:2017:HPC17,Ljungkvist:2017:TechRep,May:2015:CMAME}.
This is a well-suited tactic for future machines because of its large arithmetic intensity \cite{loffeld2017arithmetic}.

Significant performance gains are often attributed to a problem, scale, and architecture-specific balance between FLOPs and memory traffic.
As a matter of course, exploiting symmetries in a problem or discretization can significantly improve the time to solution.
This is often the cause of enormous speed-ups in computations on structured meshes.
Likewise, high performance of matrix-free methods can be most easily achieved in homogeneous problems with simple geometries.
Nevertheless, most geometries coming from significant real-world problems cannot be adequately approximated by fully-structured meshes.
A possible trade-off is to use locally-structured meshes like hierarchical hybrid grids (HHG) where initially unstructured coarse grids are locally refined in a uniform way.
This local structure allows the application of stencil-based finite element procedures which operate similar to finite difference methods.
By using these grids, efficient stencil-based methods have been successfully applied to a wide range of problems \cite{bergen2005hierarchical,bergen2004hierarchical,bergen2007hierarchical,Engwer:2017:CCPE,gmeiner2015towards}.
A related approach, suitable for low-order finite element discretizations of scalar elliptic partial differential equations (PDEs) with variable coefficients, based on scaling of reference stencils is discussed in \cite{bauer2017stencil}.

In this paper, we revisit the classical lowest-order Bubnov--Galerkin finite element method and analyze a modification of it which is strongly amenable to stencil-based matrix-free computation.
In our approach, a macro-mesh, which is not required to have any global structure, is used to triangulate the model geometry.
This macro-mesh is then uniformly refined a large number of times, resulting in a fine-scale locally-structured mesh.
For each macro-element, a \emph{local approximation} of the fine-scale global matrix delivers a fine-scale \emph{surrogate matrix} which maintains the convergence properties of the fine-scale discrete solution, up to the original order of the approximation.
A related investigation \cite{bauer2017two} illustrated the promise of this methodology and provided numerical evidence for the convergence rates which are proven here rigorously.
This work considered only Poisson's equation.
Later on, Stokes flow (with variable viscosity) was considered in two follow-up articles \cite{bauer2018large,bauer2018new}.
Each of these initial studies focused on the massively parallel high performance computing aspects of their respective methods.
These studies used the HHG software framework \cite{bergen2005hierarchical,bergen2004hierarchical,bergen2007hierarchical} in their experiments.
Here, the finite-element software framework HyTeG \cite{Kohl2018HyTeGfinite} is used.

In this paper, we recast the central features of the original work as a methodology complete with a mathematical framework suitable for rigorous analysis.
The principal novelty is the mathematical foundation developed here, which can be used to analyze further incarnations of the methodology.
In total, we consider three specific mathematical models; namely, the variable coefficient Poisson equation, linear elastostatics, and $p$-Laplacian diffusion.
Although our presentation demonstrates that the surrogate matrix methodology applies to each of these models equally well, we only employ a complete \textit{a priori} analysis of the simplest model, the variable coefficient Poisson equation.
In our numerical experiments, we carefully verify the proven \textit{a priori} convergence rates with the variable coefficient Poisson equation.
We also include proof-of-concept demonstrations from numerical experiments with the linear elastostatics and $p$-Laplacian diffusion problems.

\section{Notation and outline} \label{ssub:overview}

Let $V$ be a reflexive Banach space over $\R$, the field of real numbers, and let $V_h\subsetneq V$ be a finite-dimensional subspace.
Consider a continuous and weakly coercive bilinear form $a:V\times V \to \R$ and a bounded linear functional $F\in V^\ast$, the topological dual of $V$.

In this paper, we are concerned with the solutions $u$, $u_h$, and $\tilde{u}_h$ of the following three abstract variational problems.
\begin{subequations}
\label{eq:VariationalFormulations}
\begin{alignat}{3}
	&
	\text{Find } u \in V \text{ satisfying }
	&
	\quad
	a(u,v)
	&=
	F(v)
		\quad
	&&\text{for all } v\in V
	\,.
\label{eq:ContinousVF}
	\\
	&
	\text{Find } u_h \in V_h \text{ satisfying }
	&
	\quad
	a(u_h,v_h)
	&=
	F(v_h)
		\quad
	&&\text{for all } v_h\in V_h
	\,.
\label{eq:DiscreteVF}
	\\
	&
	\text{Find } \tilde{u}_h \in V_h \text{ satisfying }
	&
	\quad
	\tilde{a}(\tilde{u}_h,v_h)
	&=
	F(v_h)
		\quad
	&&\text{for all } v_h\in V_h
	\,.
\label{eq:SurrogateVF}
\end{alignat}
\end{subequations}
In~\cref{eq:SurrogateVF}, a \emph{surrogate} bilinear form $\tilde{a}:V_h\times V_h \to \R$ has been introduced.
In order to properly define $\tilde{a}(\cdot,\cdot)$, some additional assumptions on $a(\cdot,\cdot)$ are still required; see \Cref{sec:surrogate_stiffness_matrix}.

The discrete variational problems~\cref{eq:DiscreteVF,eq:SurrogateVF} induce matrix equations for coefficients $\sfu,\tilde{\sfu}$ in some $\R^N$,
\begin{equation}
	\sfA\sfu = \sff\qquad \text{ and } \qquad\tilde{\sfA}\tilde{\sfu} = \sff,
\label{eq:DiscreteVariationalProblems}
\end{equation}
respectively.
In the first case, fix a basis for $V_h$, say $\{\phi_i\}_{i=1}^N$.
For this basis, each $(i,j)$-component of the stiffness matrix $\sfA$ is simply $\sfA_{ij} = a(\phi_j,\phi_i)$.
In the following section, we present a methodology to construct a \emph{surrogate} stiffness matrix $\tilde{\sfA}\approx \sfA$ which can be used in place of the \emph{true} stiffness matrix $\sfA$.
This methodology stands apart from technical details, such as differences in quadrature formulas.
\Cref{sec:examples} provides a short (non-comprehensive) list of examples fitting into our framework.
In \Cref{sec:the_zero_row_sum_property}, we discuss the incorporation of non-homogeneous boundary conditions and what we hereon refer to as \emph{the zero row sum property}.
In \Cref{sec:stability_of_the_surrogate_operator}, sufficient conditions for the discrete stability of surrogate bilinear forms $\tilde{a}(\cdot,\cdot)$ are briefly discussed.
Next, in \Cref{sec:a_priori_error_estimation}, we perform a rigorous \emph{a priori} error analysis of our approach applied to the variable coefficient Poisson equation.
A brief description of our implementation is given in \Cref{sec:implementation}.
Then, in \Cref{sec:numerical_experiments}, we document several numerical experiments.
Here, a thorough verification of each error estimate in \Cref{sec:a_priori_error_estimation} is given.
This is complemented by performance measurements for the additional examples.

Throughout this article, we assume that $\Omega\subset\R^n$ is a bounded Lipschitz domain.
For matrices $\sfM \in \R^{l\times m}$, define the $\ell^\infty$-, and ${\max}$-norms, $\|\sfM\|_\infty = \max_{i}\sum_{j}|\sfM_{ij}|$ and $\|\sfM\|_{\max} = \max_{i,j}|\sfM_{ij}|$.
Likewise, for any function $v:\Omega\to \R$, we will use the similar notation, $\|v\|_0$, $\|v\|_1$, and $\|v\|_2$, for the canonical $L^2(\Omega)$-, $H^1(\Omega)$-, and $H^2(\Omega)$-norms, respectively.
When dealing with a subset $T\subset\Omega$, denote the related $L^2(T)$-, $H^1(T)$-, and $H^2(T)$-norms by $\|v\|_{0,T}$, $\|v\|_{1,T}$, and $\|v\|_{2,T}$, respectively.
For any simplex $T$ and integer $0\leq q <\infty$, we denote the space of polynomials of degree at most $q$ as $\mcP_q(T)$. 
All remaining notation will be defined as it arises.

\section{Surrogate stiffness matrices} \label{sec:surrogate_stiffness_matrix}

In this section, we present the constitutive elements of the surrogate matrix methodology.
Our approach here is to gradually introduce the necessary concepts, all the while maintaining a clear sense of generality.
In order to arrive at a tractable framework for our problems of interest, we gradually refine the presentation from general $n$-dimensional spaces to only $n=1,2$ or $3$ and from general Banach spaces to only $W^{1,p}(\Omega)$ (or products thereof), where $1<p<\infty$.
The intention of proceeding in this way is to indicate that the methodology can be applied to an extremely broad set of problems and, specifically, to most problems where finite element methods are traditionally applied.

\subsection{Preliminary assumptions} \label{sub:assumptions_and_preliminaries}

Given a bounded domain $\Omega\subset\R^n$, assume that the true bilinear form can be expressed as
\begin{subequations}
\begin{equation}
	a(u,v) = \int_\Omega G(x,u(x),v(x)) \dd x
	\qquad \text{for all } u, v\in V.
\label{eq:BilinearFormStructure}
\end{equation}
Additionally, upon defining $\supp(u) = \{ y \in \Omega : u(y) \neq 0 \}$ for smooth functions, make the following sparsity assumption:
\begin{equation}
	G(x,u(y),v(y)) = 0
	\qquad
	\text{whenever }
			y\notin\supp(u)\cap\supp(v)
		.
\label{eq:SupportAssumption}
\end{equation}
\end{subequations}

These assumptions permit us to consider the discretization of most classical differential operators.
Indeed, in the assumptions above, the integrand $G(x,u,v)$ may induce distributional derivatives on its second and third arguments.
Meanwhile, the first argument can be identified with the spatial argument of any associated variable coefficients.
For example, in the weak form of a Poisson-type equation, $-\div(K\nabla u) = f$, with a variable, symmetric positive-definite diffusion tensor $K(x)$ (cf.~\Cref{sub:poisson_s_equation}), we have the bilinear form
\begin{equation}
	a_1(u,v)
	=
	\int_\Omega \nabla u(x)^\top K(x)\sspace \nabla v(x) \dd x
		\qquad \text{for all } u, v\in V=H^1(\Omega)
	\,.
\label{eq:BilinearFormPoisson}
\end{equation}
Here, taking any point $x\in\Omega$, the integrand in~\cref{eq:BilinearFormStructure} reduces to $G(x,u,v) = {\nabla u^\top K(x)\sspace \nabla v}$.
Evidently, this $G$ satisfies the sparsity assumption~\cref{eq:SupportAssumption}.

\subsection{Stencil functions} \label{sub:stencil_functions}

Let $\phi\in V$ be a test function with compact support in $\Omega$ and, for any fixed $y\in \R^n$, define $\phi_y(x) = \phi(x-y)$.
Now, consider any fixed set of ordered points $\bbX = \{x_i\}$ in $\Omega$ and recall~\cref{eq:BilinearFormStructure}.
Assuming that both $\phi_{x_i},\phi_{x_j}\in V$, observe (via a simple change of variables) that
\begin{equation}
	\begin{aligned}
						a(\phi_{x_j},\phi_{x_i})
				=
		\int_\Omega G(y,\phi_{x_j}(y),\phi_{x_i}(y)) \dd y
												&=
		\int_{\Omega_{\delta}} \!\!G(x_i+y,\phi_{\delta}(y),\phi(y)) \dd y
		\,,
	\end{aligned}
\label{eq:ChangeOfVars}
\end{equation}
where $\delta = x_j-x_i$ and $\Omega_{\delta} = \supp(\phi)\cap\supp(\phi_{\delta})$.
In the second equality, passing from an integral over $\Omega$ to an integral over the subset $\Omega_{\delta}\subset \Omega$ follows immediately from the sparsity assumption~\cref{eq:SupportAssumption}. 
For each fixed $x_i$, the affine structure of the identity above may be illuminated by collecting each contributing translation $\delta$ into the set $\scD(x_i) = \{x_j-x_i \,:\, x_j\in\bbX,\, a(\phi_{x_j},\phi_{x_i}) \neq 0\}$ and defining a \emph{stencil function}
\begin{equation}
	\Phi_i^\delta(x)
		=
	\int_{\Omega_\delta} G(x+y,\phi_{\delta}(y),\phi(y)) \dd y
		\qquad
	\text{for each } \delta \in \scD(x_i).
	\label{eq:StencilFunction}
\end{equation}

We have just reduced the computation of any $a(\phi_{x_j},\phi_{x_i})$ to the evaluation of scalar-valued functions enumerated by affine coordinates $(x_i,x_j-x_i)$.
Indeed,
\begin{equation}
	a(\phi_{x_j},\phi_{x_i})
	=
	\begin{cases}
	\Phi_i^\delta(x_i)
		\,,\quad & \text{if } \delta = x_j-x_i \in\scD(x_i),\\
	0\,, & \text{otherwise.}
	\end{cases}
	\label{eq:StencilFunction2StiffnessMatrixGENERAL}
\end{equation}
In the present scenario, there may be a different set of translations $\scD(x_i)$ for every point $x_i$.
However, if each point is drawn from a point lattice, most of the sets $\scD(x_i)$ are identical.
This observation is the subject of the following subsection and a foundational principle in our approach.

\begin{remark}
\label{rem:Symmetry}
	In the scenario that the bilinear form is symmetric, $a(u,v) = a(v,u)$, it is natural to assume that the stencil functions~\cref{eq:StencilFunction} will inherit a similar symmetry.
													Indeed, under the equivalent symmetry condition $G(x,u,v) = G(x,v,u)$ almost everywhere, one may easily verify that if $\delta = x_j - x_i$, then
																\begin{equation}
		\Phi^\delta_i(x_i)
		=
		\int_{\Omega_{\delta}} \!\!G(x_i+y,\phi(y),\phi_{\delta}(y)) \dd y
		=
				\int_\Omega G(x_j+y,\phi_{-\delta}(y),\phi(y)) \dd y
		=
		\Phi^{-\delta}_j(x_j)
									\label{eq:SymmetricStencilFunction2StiffnessMatrixGENERAL}
	\end{equation}
	or, equivalently, $\Phi^\delta_i(x_i) = \Phi^{-\delta}_j(x_i + \delta)$.
													\end{remark}

\subsection{Local stencil functions and locally-structured meshes} \label{sub:stencil_functions_on_uniformly_structured_meshes}

An affine point lattice $\bbL$, from here on referred to only as a \emph{lattice}, is a regularly spaced array of points in $\R^n$ where every point $x_i\in\bbL$ belongs to a neighborhood containing no other points in $\bbL$.
In this paper, each (possibly finite) lattice is determined by a finite linearly independent set of translations in $\R^n$; i.e., $\bbL \subset \{ \delta_0 + a_1 \delta_1 + \cdots + a_l \delta_l \,:\, a_1,\ldots,a_l\in\bbZ\}$.

Assuming that the test function $\phi\in V$ is sufficiently localized and each point $x_i$ is drawn from a lattice $\bbL\subset\Omega$, then each $\scD(x_i)$ is a subset of a small number of admissible translations $\scD(\bbL) = \bigcup\{\scD(x_i): x_i\in\bbL\}$, determined solely by the lattice structure.
In such a scenario, every stencil function is closely related; i.e., $\Phi_i^\delta(x) = \Phi_j^\delta(x)$, whenever both are defined.
Therefore, it is prudent to drop the subscript and define only one common stencil function $\Phi^\delta(x)$ for each $\delta \in \scD(\bbL)$.
Clearly,
\begin{equation}
	a(\phi_{x_j},\phi_{x_i})
	=
	\begin{cases}
	\Phi^\delta(x_i)\,,\quad & \text{if } \delta = x_j-x_i \in\scD(\bbL),\\
	0\,, & \text{otherwise.}
	\end{cases}
	\label{eq:StencilFunction2StiffnessMatrixSTRUCTURED}
\end{equation}

We are interested in exploiting \cref{eq:StencilFunction2StiffnessMatrixSTRUCTURED} for solving a wide variety of PDEs with curvilinear geometries.
Toward this end, the following examples help motivate our construction further.

\begin{remark}
\label{rem:ReflexiveSobolevSpaces}
From now on, it is useful to let $V$ be a closed subset of $W^{1,p}(\Omega)$, for some $1<p<\infty$.
This will allow us to use a basis for $V_h \subset V$ consisting of finite element vertex functions \cite{ern2013theory}.
Other problems where $V\subset \big[W^{1,p}(\Omega)\big]^k$, $k\in\N$, can be handled similarly (see, e.g., \Cref{sub:linearized_elasticity}), by employing a basis consisting of the same vertex functions in each component.
\end{remark}

\subsubsection{The one-dimensional setting} \label{ssub:the_one_dimensional_setting}

Let $V=H^1_0(\Omega)$, where $\Omega=(0,1)\subset\R$, and fix a small translation $dx = 1/(N+1)$.
Consider the scenario where each point, ${x_i = x_{i-1} + dx}$, evenly divides $\Omega$ and $\phi$ is the piecewise-linear hat function defined $\phi(x) = \max(1-\frac{|x|}{dx},0)$.
Let $V_h = \{v \in H^1_0(\Omega) : v|_t \in \mcP_1(t),\,t = (x_i, x_{i+1}), \text{ for each } 1\leq i\leq N \}$ and identify each shifted hat function with the standard basis, $\phi_{x_i} = \phi_i\in V_h$.
Here, we may define $\bbL = \{x_i\}$.
In this case, for each $i\geq 1$, the value $a(\phi_i,\phi_j)$ can either be computed directly from \eqref{eq:BilinearFormStructure}, in the standard way, or evaluated using~\cref{eq:StencilFunction2StiffnessMatrixSTRUCTURED}, assuming that each $\Phi^\delta$ is available at the onset of computation.
Note that $\scD(x_1) = \{0,dx\}$ and $\scD(x_N) = \{-dx,0\}$, but $\scD(x_i) = \{-dx,0,dx\}$ for each $2\leq i \leq N-1$.
Therefore, $\scD(\bbL) = \{-dx,0,dx\}$.
Ultimately, defining each structured stencil function $\Phi^\delta(x) = \Phi((x_i,\delta);x)$ from an arbitrary candidate point $x_i$, one may verify that
\begin{equation}
	a(\phi_j,\phi_i)
	=
	\begin{cases}
	\Phi^\delta(x_i)\,,\quad & \text{if } \delta = x_j-x_i \in \{dx,0,dx\},\\
	0\,, & \text{otherwise,}
	\end{cases}
\end{equation}
which is clearly the same format as~\cref{eq:StencilFunction2StiffnessMatrixSTRUCTURED}.

Recall \cref{rem:Symmetry}.
If $a(\cdot,\cdot)$ is symmetric, then, by~\cref{eq:SymmetricStencilFunction2StiffnessMatrixGENERAL}, $\Phi^{dx}(x_i) = \Phi^{-dx}(x_i+dx)$ and the number of required stencil functions can be reduced to two.
In some situations\textemdash{}e.g., when the zero row sum property can be employed (see \Cref{sec:the_zero_row_sum_property})\textemdash{}only a single stencil function is actually required.

\subsubsection{Locally-structured meshes with triangles} \label{ssub:hierarchical_hybrid_grids_with_triangles}
Let $m\in\N$.
Beginning with a scaled Cartesian lattice $\bbL_m = 2^{-m}\Z^n$, it is useful to define its intersection with the closure of the right-angled reference simplex $\hat{T} = \{\hat{x}\in\R^n : \|\hat{x}\|_1 < 1,\, \hat{x}\cdot e_i > 0,\,\forall i=1,\ldots,n\}$.
This simplicial lattice, $\hat{T}_m = \bbL_m \cap \overline{\hat{T}}$, can easily be transformed into a similar simplicial lattice $T_m$ for any arbitrary simplex $T\subset\Omega$ via an affine transformation (see, e.g., \Cref{fig:Subdivision}).
Indeed, fixing the unique $A\in\R^{n\times n}$ and $b\in\R^{n}$ such that $T = \{A\hat{x}+b : \hat{x}\in\hat{T}\}$, the corresponding lattice is clearly $T_m = \{ A\hat{x}_i+b : \hat{x}_i\in \hat{T}_m\}$.
Note that there are no interior points, $\mathring{T}_m = T_m\cap T = \emptyset$, if $m<2$.
We also define the set of boundary points $\bdry T_m = T_m \setminus \mathring{T}_m$, which is always non-empty.

Let $\Omega\subset\R^n$, where $n=2,3$.
The utility of this transformation is evident upon considering a \emph{macro}-triangulation of $\Omega$, say $\mesh_H$, where each \emph{macro}-element $T\in\mesh_H$ is endowed with a simplicial lattice $T_m$, as defined above.
Here, $H = \max_{T\in\mesh_H} H_T$, where each $H_T = \mathrm{diam}(T)$ denotes the diameter of $T$.
Notice that, for any fixed level $m\geq0$, all interface points $x_i\in \bdry T_m\cap\Omega$ are coincident with an interface lattice point on some neighbouring simplex.
We may now define the set of all vertices on level $m$: $$\bbX_m = \bigcup \{T_m:T\in\mesh_H\}.$$

For every $\bbX_m$, there is a corresponding finite element mesh such that each vertex function within a fixed macro-element $T\in\mesh_H$ is self-similar.
In the cases $n=2$ or $3$, we are left to define each triangle or tetrahedron whose vertices coincide with points in $\bbX_m$.
Let $[y_0,y_1,\ldots,y_k]\subset\R^n$ denote the convex combination of the points $y_0,y_1,\ldots,y_k\in\Omega$.
When $n=2$, the natural construction begins by considering the following uniform subdivision of a triangle $T = [y_0,y_1,y_2] \in\mesh_H$ into a set of four equal-volume triangles:
\begin{equation}
	\mcS(T)
	=
	\big\{[y_0,\tfrac{y_0+y_1}{2},\tfrac{y_0+y_2}{2}],[\tfrac{y_0+y_1}{2},y_1,\tfrac{y_1+y_2}{2}],[\tfrac{y_0+y_2}{2},\tfrac{y_1+y_2}{2},y_2],[\tfrac{y_0+y_1}{2},\tfrac{y_1+y_2}{2},\tfrac{y_0+y_2}{2}]\big\}
	.
\end{equation}
For an illustration of this $n=2$ case, see \cref{fig:Subdivision}.
For $n=3$, see the construction in \cite{bey1995tetrahedral}.
Further subdivisions can then be defined recursively, \textit{viz.},
\begin{equation}
	\mcS^{m+1}(T)
	= 
	\bigcup\{\mcS(t) : t \in \mcS^{m}(T)\}
		\qquad \text{ for all } m\geq 1\,.
\end{equation}

\begin{figure}\centering
\includegraphics[width=\textwidth]{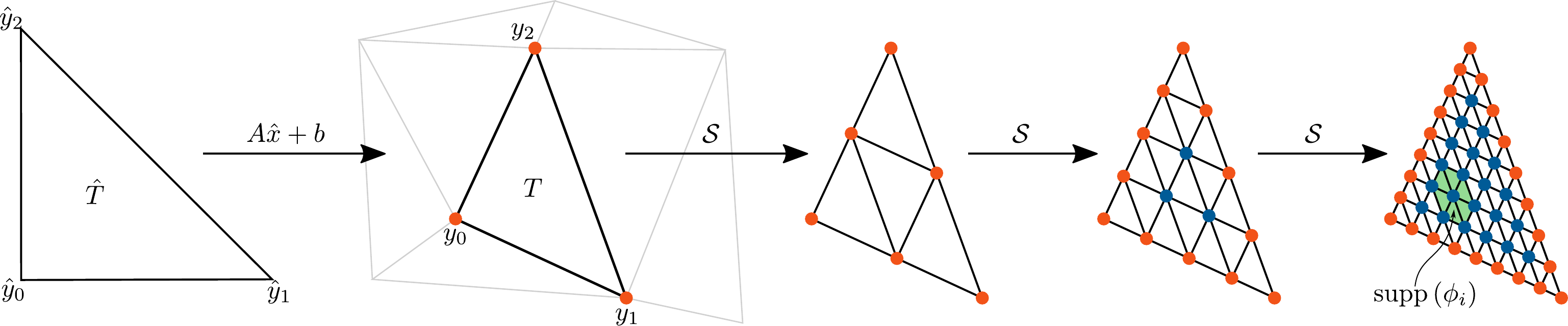}
\caption{\label{fig:Subdivision}Illustration of three refinement steps of a single macro-element $T$ for $n=2$ with the corresponding vertex lattices $T_0$, $T_1$, $T_2$, and $T_3$. The interior lattice points $\mathring{T}_m$ are colored blue and the boundary lattice points $\partial T_m = T_m \setminus \mathring{T}_m$ are colored in orange. Additionally, the support of an exemplary vertex function $\phi_i$ is shaded in green.}
\end{figure}

The set of all vertices in a given subdivision $\mcS^{m}(T)$ forms an evenly spaced set of points inside $T$.
The set of vertices in $\mcS^{1}(T)$ clearly coincides with $T_1$ and one can easily verify from the recursive definition that the set of vertices in $\mcS^{m}(T)$ also coincides with $T_m$.
We may finally define the sequence of \emph{locally-structured meshes}:
\begin{equation}
	\mcS^m(\mesh_H)
	=
	\bigcup \{\mcS^m(T) : T\in\mesh_H\},
	\qquad \text{ for all } m\geq 1\,.
\label{eq:FineScaleMesh}
\end{equation}
Notice that for each $m$, $\bbX_m$ coincides with the set of all vertices in $\mcS^m(\mesh_H)$.
Therefore, each point $x_i$ in $\bbX_m$ can be identified with a vertex function $\phi_i$ supported by only the neighboring elements appearing in $\mcS^m(\mesh_H)$; see \cref{fig:Subdivision}.

Assume that the fine mesh level, $m$, is chosen large enough that one may find several points in $\bbX_m$ which lie in the interior of some macro-element $T$.
If $x_i\in \mathring{T}_m$ is any such point, then the translation set $\scD(x_i)=\scD(\mathring{T}_m)$ will only contain translations aligned with edges appearing in the original subdivision $\mcS^1(T)$.
Identifying $\phi_{x_i}=\phi_i$ we see that for every $x_i\in \mathring{T}_m$ and $x_j\in T_m$,
\begin{equation}
	a(\phi_{j},\phi_{i})
	=
	\begin{cases}
	\Phi^\delta_{T}(x_i)\,,\quad & \text{if } \delta = (x_j-x_i) \in\scD(\mathring{T}_m),\\
		0\,, & \text{otherwise,}
	\end{cases}
	\label{eq:StencilFunction2StiffnessMatrixHHG}
\end{equation}
where $\Phi^\delta_{T}:\mathrm{conv}(\mathring{T}_m)\to\R$ is a local stencil function for the current level $m$ and macro-element $T$ and $\mathrm{conv}(\mathring{T}_m)$ is the convex hull of $\mathring{T}_m$.
In general, notice that $\Phi^\delta_{T} \neq \Phi^{2\delta}_{T}$ are two different stencil functions, corresponding to the same direction but different mesh levels.

Consider the bilinear form~\cref{eq:BilinearFormPoisson} with a variable diffusion coefficient.
A visualization of several corresponding local stencil functions, coming from locally-structured meshes used in our numerical experiments, is given in \Cref{fig:stencilweight_surfaceplot}.
It is clear from this figure that each $\Phi^\delta_{T}$ has the potential to be a smooth function.
We now come to the final essential component of our surrogate methodology; the approximation of $\Phi^\delta_{T}$.

\begin{figure}	\centering
	\begin{minipage}{0.48\linewidth}
		\centering
		\includegraphics[width=\linewidth]{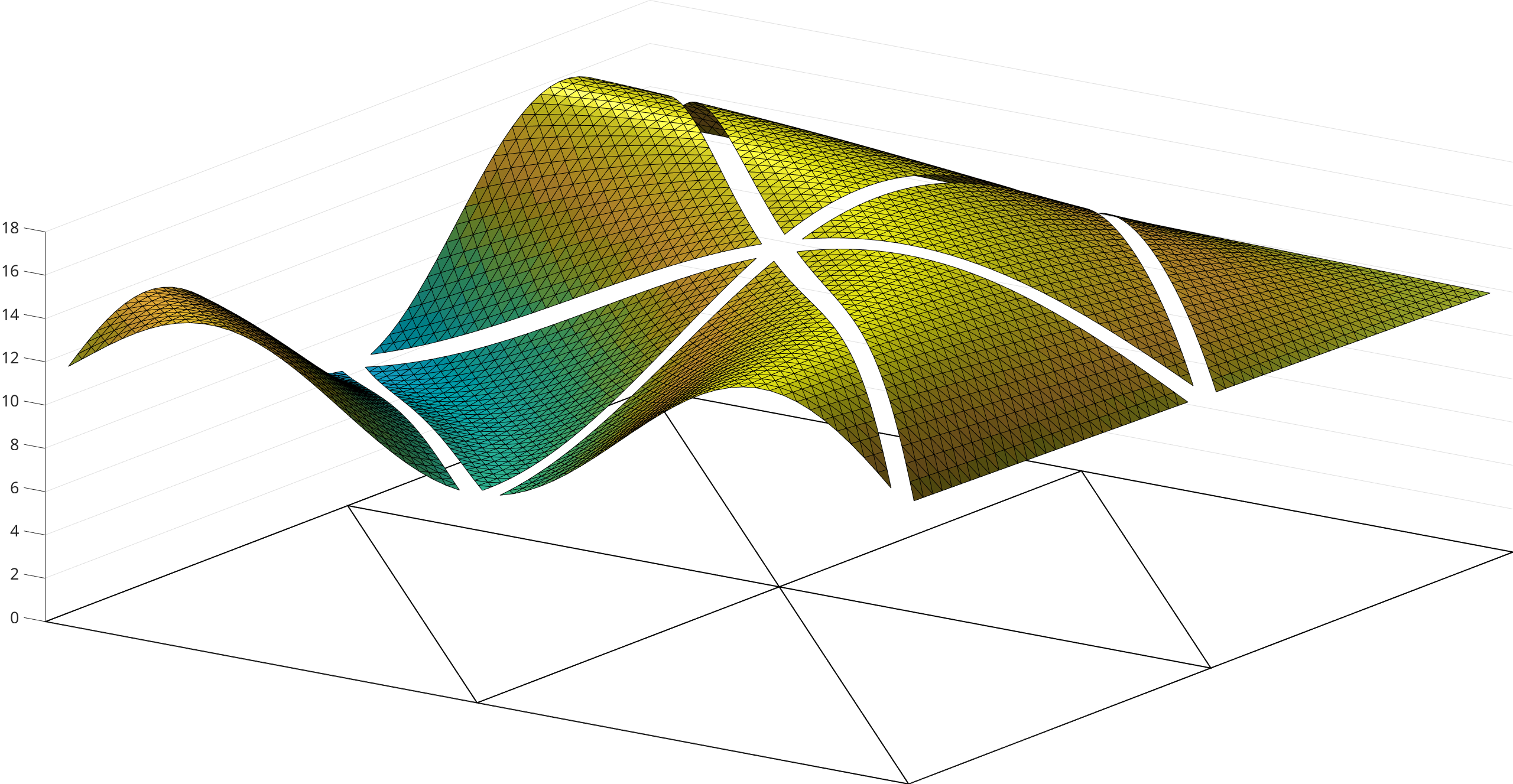}
	\end{minipage}\hfill
	\begin{minipage}{0.48\linewidth}
		\centering
		\includegraphics[width=\linewidth]{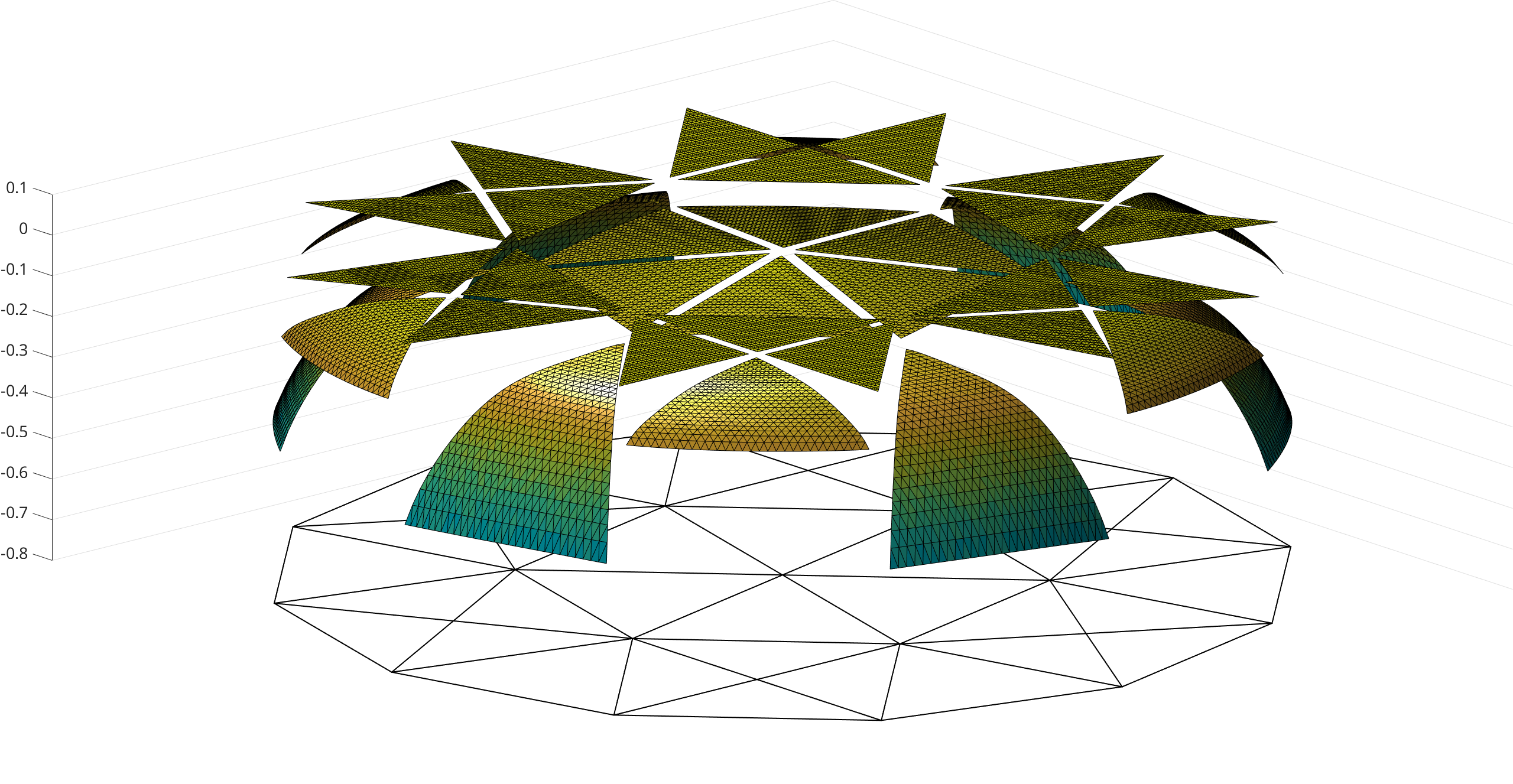}
	\end{minipage}
	\caption{\label{fig:stencilweight_surfaceplot}Left: Surface plots of the local stencil functions $\Phi^\delta_{T}(x)$, for the degenerate direction $\delta=0$ and level $m = 5$, from the numerical experiments recounted in \Cref{sec:scalar_coefficient_benchmark}.
	Here, the function is plotted over each subset $\mathrm{conv}(\mathring{T}_m)\subset T\in\mesh_H$.
	In this case, it clearly appears that each stencil function can be related to the restriction of a globally continuous function $\Phi^\delta_{T}(x)$. This is a result of the structure of the macro-mesh only.
	Right: Surface plots of the stencil functions $\Phi^\delta_{T}(x)$ after the first time step from the experiment recounted in \Cref{sec:plaplacian_diffusion_example}, for the eastern direction $\delta$ relative to each macro-element and level $m = 5$. Moreover, although they are clearly related, it is evident that the corresponding stencil functions lack any global smoothness property.}
\end{figure}

\begin{remark}
\label{rem:Geometry}
	The lattice structure of locally-structured meshes is destroyed under smooth, non-affine transformations $\hat{T} \to \tilde{T}\subset\Omega$.
	This offers a possible impediment to our construction in the case of non-polygonal domains $\overline{\Omega} \neq \overline{\Omega}_H = \bigcup_{T\in\mesh_H} \overline{T}$.
	In fact, if a globally continuous transformation $\varphi:\overline{\Omega}_H \to \overline{\Omega}$ is available, with $\varphi|_T$ a smooth bijection for every $T\in\mesh_H$, then an equivalent method can be found using local pull-backs of $\varphi$.
							For example, the bilinear form in~\cref{eq:BilinearFormPoisson}, $a_1:H^1(\Omega)\times H^1(\Omega) \to \R$, simply transforms to $a_{1,H}:H^1(\Omega_H)\times H^1(\Omega_H) \to \R$, where
	\begin{equation}
		a_{1,H}(u,v)
		=
		\int_{\Omega_H} \nabla u(x)^\top K_{H}(x) \nabla v(x) \dd x
		\,,
						\qquad
		K_{H} = \frac{D\varphi^{-1}\sspace (K\circ\varphi)\sspace D\varphi^{-\top}}{|\det{\left(D\varphi^{-1}\right)}|}
		\,.
	\label{eq:BilinearFormPoissonTensor}
	\end{equation}
	Therefore, from now on, we operate under the assumption that the macro-mesh $\mesh_H$ is geometrically conforming, $\overline{\Omega} = \overline{\Omega}_H$.
\end{remark}

\subsection{The inherited regularity of stencil functions} \label{sub:the_inherited_regularity_of_stencil_functions}

Our approach is to locally project each stencil function $\Phi^\delta_{T}$ in~\cref{eq:StencilFunction2StiffnessMatrixHHG} onto a high-dimensional space of polynomials and later use this projection to compute approximate values of the stiffness matrix $\sfA$. Let $\mesh_H$ be a shape-regular simplicial mesh and consider a locally-structured mesh of level $m$ subordinate to this mesh, $\mcS^m(\mesh_H)$.
Before moving on, observe that definition~\cref{eq:StencilFunction2StiffnessMatrixHHG} in fact holds for any $x_i,x_j\in T_m$ if $x_i$ or $x_j \in \mathring{T}_m$.
Therefore, due to the structure of the vertex functions in a locally-structured mesh, the domain of each $\Phi^\delta_{T}$ can actually be extended to a set $\overline{T}_\delta$ lying between $\mathrm{conv}(\mathring{T}_m)$ and $\overline{T}$.
Indeed, identifying the test function in~\cref{eq:StencilFunction} with the $i$-th vertex function, $\phi(x) = \phi_i(x+x_i)$, for an arbitrary vertex $x_i\in\mathring{T}_m$, define
\begin{equation}
	T_\delta
	=
		\{x\in T : x + y \in T, \text{ for all } y\in\Omega_\delta = \supp(\phi)\cap\supp(\phi_\delta)\}
	.
\end{equation}
From now on, we assume $\Phi^\delta_{T}:\overline{T}_\delta\to\R$.
See~\cref{fig:DomainsofTheStencilFunctions} for a depiction of the sets in a triangular mesh and note that $T_{-\delta} = T_\delta + \delta$, for every $\delta \in \scD(T_m)$.

\begin{figure}\centering
\includegraphics[width=0.8\textwidth]{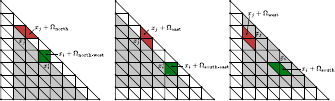}
\caption{\label{fig:DomainsofTheStencilFunctions}Illustrations of the domains $T_\delta$ in gray and $\Omega_{\delta}$ for six exemplary directions $\delta$. Left: Northern and north-western direction. Middle: Eastern and south-eastern direction. Right: Western and southern direction.}
\end{figure}

For any $T\in\mesh_H$, let $\mcP_q(T_\delta)$ denote the space of polynomials of degree at most $q$ on the simplex $\overline{T}_\delta$ and let $\Pi^\delta_T:C^0(\overline{T}_\delta)\to \mcP_q(T_\delta)$ be an $L^\infty$-continuous projection operator, $\Pi^\delta_T\circ \Pi^\delta_T = \Pi^\delta_T$.
For each macro-element $T\in\mesh_H$ and level $m\in\N$, define the \emph{surrogate stencil function} ${\tilde{\Phi}^\delta_{T}\!:{T}_\delta\to\R}$ to be the corresponding polynomial projection of $\Phi^\delta_{T}$.
Namely,
\begin{equation+}
	\tilde{\Phi}^\delta_{T}
	=
	\Pi^\delta_T\, \Phi^\delta_{T}
	.
	\label{eq:StencilInterpolation}
\end{equation+}

In order to correctly argue that a polynomial approximation of $\Phi^\delta_{T}$ is feasible, it is necessary to classify its regularity depending on the problem at hand.
In the following proposition, we show, under the modest assumptions above, that if $G(\cdot|_T,\cdot,\cdot)$ is a polynomial in its first argument, then $\Phi^\delta_{T}$ is also a polynomial of the same degree.

\begin{lemma}
\label{prop:Interpolation}
	Fix a simplex $T\in\mesh_H$.
		 	Assume that the bilinear form $a(\cdot,\cdot)$ in~\cref{eq:ContinousVF} satisfies assumptions~\cref{eq:BilinearFormStructure,eq:SupportAssumption}, where the integrand $G(\cdot|_T,\cdot,\cdot)$ is a polynomial of at most degree $q$ in its first argument.
	Then, for any locally-structured mesh $\mcS^m(\mesh_H)$, as defined above, every local stencil function $\Phi^\delta_{T} \in \mcP_q(T_\delta)$ is a polynomial of the same degree.
\end{lemma}
\begin{proof}
	Recall definition~\cref{eq:StencilFunction}.
	For every $x_i,x_j\in\bbX_m$, every stencil function $\Phi((x_i,x_j-x_i),x)$ is defined with a test function $\phi\in V$.
	Fixing any arbitrary vertex $x_i\in\mathring{T}_m$, identify this test function with the $i$-th vertex function, $\phi(x) = \phi_i(x+x_i)$.
		Let $\alpha = (\alpha_1,\ldots,\alpha_n) \in \N_0$ denote a standard multi-index, $|\alpha| = \sum_{i=1}^n |\alpha_i|$ and $x^\alpha = x_{\scriptscriptstyle 1}^{\alpha_1}\cdots x_{\scriptscriptstyle n}^{\alpha_n}$.
	By assumption, we may express $G(x, \phi_\delta(y),\phi(y)) = \sum_{|\alpha|\leq l} c_\alpha(y) x^\alpha$, where each coefficient function $c(y)$ has support only in $\Omega_\delta = \supp(\phi)\cap\supp(\phi_\delta)$.
		Moreover, if $x+y\in T$, then
	\begin{equation}
		G(x+y, \phi_\delta(y),\phi(y))
		=
		\sum_{|\alpha|\leq l} c_\alpha(y) (x+y)^\alpha
		=
		\sum_{|\alpha|\leq l}\sum_{|\nu|\leq \alpha}\binom{\alpha}{\nu} c_\alpha(y)y^\nu x^{\alpha-\nu}
			\label{eq:PolyExpansion}
	\end{equation}
	is clearly an equal degree polynomial in the variable $x$.
	The proof is completed by noting that the integral in~\cref{eq:StencilFunction} is performed only over the variable $y\in\Omega_\delta$ and so the stencil function acts like a convolution.
	Indeed, under the assumption $x\in T$, only the subset of points $x\in T_\delta = \{x\in T : x+y \in T \text{ for all } y\in \Omega_\delta\}$ guarantee that~\cref{eq:PolyExpansion} holds at every point of integration $y$.
	In this case, by linearity of integration, $\Phi^\delta_{T}$ is a member of $\mcP_q(T_\delta)$, with its coefficients defined by the associated integrals of the $y$-dependent functions in the right-hand side of~\cref{eq:PolyExpansion}.
\end{proof}

\begin{corollary}
\label{cor:Interpolation}
	In the setting of \Cref{prop:Interpolation},
	 		$
		\tilde{\Phi}^\delta_{T}
		=
		\Phi^\delta_{T}
	$.
			\end{corollary}
\begin{proof}
	Since $\Phi^\delta_{T} \in \mcP_q(T_\delta)$, we immediately see that $\Pi^\delta_T\, \Phi^\delta_{T} = \Phi^\delta_{T}$.
\end{proof}

\subsection{Surrogate stiffness matrices} \label{sub:the_surrogate_stiffness_matrix}

The main goal of the entire effort above is to guide us in reducing the vast majority of the finite element assembly process to the evaluation of a small set of functions which can, in fact, be locally approximated by polynomials.
As in \Cref{ssub:hierarchical_hybrid_grids_with_triangles}, let $\mesh_H$ be a shape-regular triangulation of a domain $\Omega$ into disjoint \emph{macro}-simplices $T$.

Our construction of the surrogate matrix $\tilde{\sfA}$ is obviously built to exploit the local lattice structure of locally-structured meshes.
As argued previously, a surrogate stencil function $\tilde{\Phi}^\delta_{T}$ can be used to approximate any matrix entry $\sfA_{ij}$ coming from a locally-structured mesh if at least one of the corresponding vertices $x_i$ or $x_j$ belongs to $\mathring{T}_m$.
This leaves us to define only the nonzero matrix entries coming from the mutual interaction of vertex functions at the boundaries of the macro-elements.
Although these entries could also be approximated by surrogate stencil functions \textemdash{} in this case, these additional functions would be defined on each subsimplex of the macro-mesh $\mesh_H$ \textemdash{} let us assume that they are computed directly.
Because the growth of the macro-mesh boundary and interface interactions grow at an order of magnitude less than the interior interactions, computing these matrix entries directly does not affect to the asymptotic performance of the methodology.
Finally, letting $\bdry\bbX_m= \bigcup_{T\in\mesh_H} \bdry T_m$ denote the union of all macro-mesh boundary vertices, we define the general surrogate stiffness matrix
\begin{equation}
	\tilde{\sfA}_{ij}
	=
	\begin{cases}
	\int_\Omega G(y,\phi_j(y),\phi_i(y)) \dd y\,,\quad & \text{if both } x_i \text{ and } x_j \in \bdry\bbX_m \\
	\tilde{\Phi}^\delta_{T}(x_i)\,,\quad & \text{if } \delta = (x_j-x_i) \in\scD(\mathring{T}_m) \text{ and } x_i \text{ or } x_j \in \mathring{T}_m,\\
	0\,, & \text{otherwise.}
	\end{cases}
	\label{eq:SurrogateStiffnessMatrix}
\end{equation}

\begin{remark}
\label{rem:SurrogateSymmetry}
Due to the presence of the surrogate stencil functions $\tilde{\Phi}^\delta_T$, even if $\sfA$ is symmetric, $\tilde{\sfA}$ will generally not be.
However, recalling~\cref{eq:SymmetricStencilFunction2StiffnessMatrixGENERAL}, observe that if $a(\cdot,\cdot)$ is symmetric, then $\Phi^\delta_{T}(x) = \Phi^{-\delta}_{T}(x+\delta)$.
Therefore, if we use related projection operators
\begin{equation}
	\Big[\Pi^\delta_T \Phi^\delta_{T}\Big](x) = \Big[\Pi^{-\delta}_T \Phi^{-\delta}_{T}\Big](x+\delta)\,,
\label{eq:SymmetricInterpolants}
\end{equation}
for each opposing direction $\delta$ and $-\delta$, then $\tilde{\sfA}$ will be symmetric.
Indeed, if $\delta = x_j-x_i$, then
\begin{equation}
	\tilde{\sfA}_{ij}
	=
	\Pi^\delta_T \Phi^\delta_{T}(x_i)
	=
	\Pi^{-\delta}_T \Phi^{-\delta}_{T}(x_i+\delta)
	=
	\Pi^{-\delta}_T \Phi^{-\delta}_{T}(x_j)
	=
	\tilde{\sfA}_{ji}
	\,.
\end{equation}
\end{remark}

\section{Examples} \label{sec:examples}

In this section, we present three example problems which easily fit into the framework above.

\subsection{The variable coefficient Poisson equation} \label{sub:poisson_s_equation}

Consider the Poisson-type equation ${-\mathrm{div}(K\nabla u) = f}$ in $\Omega$, $u = 0$ on $\bdry\Omega$, with a load $f\in L^2(\Omega)$ and a variable, symmetric positive-definite tensor $K$.
Furthermore, assume that for each index $a,b$, $K_{ab}\in \mcP_q(\Omega)$.
Recall~\cref{eq:BilinearFormPoisson} and note that we have already shown that the weak form of this problem can be cast into the framework above.

Assume that, within some set $T\subset\Omega$, each vertex function $\phi_i$ is a translation of a fixed test function $\phi(x) = \phi_i(x-x_i)$.
Then for each $\phi_i,\phi_j$, the stiffness matrix entry
\begin{equation}
	\sfA_{ij}
		=
	\int_{\Omega} \nabla \phi_i(x)^\top K(x)\sspace \nabla \phi_j(x)\hspace{0.5pt} \dd x
		\label{eq:TrueStiffnessMatrixPoisson}
\end{equation}
can equally well be expressed as the evaluation (at the point $x_i$) of a stencil function, which, by \Cref{prop:Interpolation}, is simply a polynomial of the same degree as the diffusion tensor $K$.
In the case of locally-structured meshes, there is a locally defined stencil function $\Phi^{\delta}_{T}:\Omega\to\R$ for each macro-element $T$ and level $m$.
In this case, each $\Phi^{\delta}_{T}:\overline{T}_\delta\to\R$ is a polynomial (of degree at most $q$) on $T_\delta$.

\subsection{Linearized elasticity} \label{sub:linearized_elasticity}

Let $\vec{\nabla}$ and $\mathrm{Div}$ denote the row-wise distributional gradient and divergence, respectively.
Now define $\epsilon(u) = \frac{1}{2}[\vec{\nabla} u + (\vec{\nabla}u)^\top]$ to be the symmetric gradient operator $\epsilon:H^1(\Omega)^n \to L^2(\Omega)^n$, where $n\geq 2$.
Consider the following standard PDE model for the displacement $u\in H^1_0(\Omega)^n$ of a linearly elastic isotropic material: $-\mathrm{Div}\, \sigma = \vec{f}$, where the stress $\sigma = 2\mu \epsilon(u) + \lambda I \mathrm{div}\,u$ and the load $f\in L^2(\Omega)^n$.

The weak form of this equation is well known in the literature \cite{ciarlet1988three} and the associated bilinear form is simply
\begin{equation}
	a_2(u,v) 
	=
	\int_{\Omega} 2\mu \epsilon(u):\epsilon(v) + \lambda\hspace{0.75pt} \mathrm{div}(u)\, \mathrm{div}(v)\hspace{0.5pt} \dd x
	\qquad \text{for all } u, v\in \big[H^1_0(\Omega)\big]^n
	.
\label{eq:BilinearFormLinearElasticity}
\end{equation}
This bilinear form obviously satisfies assumptions~\cref{eq:BilinearFormStructure,eq:SupportAssumption}.
If we assume that the Lam\'e parameters $\mu,\lambda:\Omega\to\R$ are piecewise polynomials on a collection of disjoint subdomains $T\in\mcT_H$, then each associated stencil function is also a piecewise polynomial.

\subsection[p-Laplacian diffusion]{$p$-Laplacian diffusion} \label{sub:p-Laplacian_diffusion}

For any $1<p<\infty$, let $\Delta_p\, u = \div(|\nabla u|^{p-2}\nabla u)$ be the $p$-Laplacian operator.
Fix a valid parameter $p$ and consider the nonlinear diffusion equation $\frac{\partial u}{\partial t} - \Delta_p u = f$, where $f\in L^p(\Omega)$.
A simple Euler time-stepping scheme replaces the time derivative $\frac{\partial u}{\partial t}$ by the quotient $\frac{u_{k+1}-u_k}{d t}$, where $dt>0$ is a fixed time-step parameter.
Choosing backward Euler time-stepping and defining $f_k = f(k\cdot dt)$, we arrive at a semi-discrete nonlinear elliptic PDE for the solution variable $u_k$, which must be solved at each step $k\in\N$: $u_{k} - dt\hspace{0.5pt}  \Delta_p u_{k} = dt\hspace{0.5pt} f_k + u_{k-1}$.
Upon fixed point linearization of the weak form of this equation, we uncover the following bilinear form:
\begin{equation}
	b(u,v)
	=
	\int_{\Omega} dt\, |\nabla \tilde{u}|^{p-2}\, \nabla u\cdot\nabla v + u \, v\hspace{0.5pt} \dd x
	,
	\qquad
	\text{for all } u,v\in W^{1,p}(\Omega)
	\,.
\label{eq:BilinearFormpLaplace}
\end{equation}
Here, the variable coefficient $\tilde{u}\in W^{1,p}(\Omega)$ is usually identified with the previous solution iteration in the associated fixed point algorithm (cf. \Cref{sec:plaplacian_diffusion_example}).

The bilinear form $b(\cdot,\cdot)$ can easily be placed into the form of~\cref{eq:BilinearFormStructure}
and each matrix entry can therefore be superceded by stencil function evaluations, $\Phi^\delta_T(x)$, as in~\cref{eq:StencilFunction}.
Alternatively, one may split $b(\cdot,\cdot)$ into a mass term $m(\cdot,\cdot)$ and a $dt$-weighted stiffness term $a_3(\cdot,\cdot)$.
Specifically, $b(u,v) = m(u,v) + dt\cdot a_3(u,v)$, where
\begin{equation}
	m(u,v)
	=
	\int_{\Omega} u \, v\hspace{0.5pt} \dd x
	\qquad
	\text{and}
	\qquad
	a_3(u,v)
	=
	\int_{\Omega} |\nabla \tilde{u}|^{p-2}\, \nabla u\cdot\nabla v\hspace{0.5pt} \dd x
	\,.
\label{eq:SplitBilinearFormspLaplace}
\end{equation}
With this observation in hand, we see that $b(u,v)$ may be discretized by a linear combination of independent surrogates; one for $m(\cdot,\cdot)$ and one for $a_3(\cdot,\cdot)$ (cf. \Cref{sec:plaplacian_diffusion_example}).
In either approach, the variable coefficient $|\nabla \tilde{u}(x)|^{p-2}$ will generally not remain a polynomial in a subdomain of $\Omega$ and the accuracy of a surrogate stencil function $\tilde{\Phi}^\delta_T$ will reflect the local regularity of the solution from the previous iteration, $\tilde{u}$.

\section{Boundary conditions and the zero row sum property} \label{sec:the_zero_row_sum_property}

It is generally appropriate to define the surrogate stiffness matrix component-wise by the rule given in~\cref{eq:SurrogateStiffnessMatrix}.
Nevertheless, in some problems the operator to be discretized has a kernel which is not guaranteed to be respected by the surrogate.
In such scenarios, it is possible that better performance and accuracy can be achieved if elements of this kernel are incorporated into the construction of the surrogate.
This occurrence is most easily illustrated with the Poisson example from \Cref{sub:poisson_s_equation}.

Consider the bilinear form $a_1:H^1(\Omega)\times H^1(\Omega)\to \R$, defined in~\cref{eq:BilinearFormPoisson}.
Define $V_h^\ext = \{ v \in H^1(\Omega) : v|_t \in \mcP_1(t) \text{ for each } t\in\mcS^m(\mesh_H)\}$
and $V_{h} = \{ v \in H^1_0(\Omega) : v|_t \in \mcP_1(t) \text{ for each } t\in\mcS^m(\mesh_H)\} \subset V_h^\ext$.
Let the corresponding vertex function bases be $\{\phi_i\} \subset \{\phi_i^\ext\}$, with $\phi_i = \phi^\ext_i$ for $1\leq i\leq N$.
In most finite element software, a space like $V_h^\ext$ is used to impose Dirichlet boundary conditions.
Indeed, a ``lift'' of the Dirichlet data, say $u_h^\ext = \sum_i \sfu^\ext_i\phi_i^\ext$, is generally constructed from a linear combination of the set $\{\phi_i^\ext\}\setminus\{\phi_i\}$.
Then, taking $\sfA^\ext_{ij} = a(\phi_j^\ext,\phi_i)$ for each valid $i,j$, a modified load vector $\sff^\ext = \sff - \sfA^\ext\sfu^\ext$ is used in computation.

It is obvious that $\nabla\, 1 = 0$ and so $a_1(1,v) = 0$ for any $v\in H^1\Omega)$.
Therefore, by the partition of unity property $\sum_i \phi_i^\ext = 1$, the zero row sum of the matrix $\sfA^\ext$ also vanishes.
Namely,
$
	\sum_j \sfA_{ij}^\ext
		=
	0
	$.
This property may be induced in the corresponding surrogate matrix if we simply define
$
			\tilde{\sfA}_{ii}^\ext
	=
	- \sum_{j\neq i} \tilde{\sfA}_{ij}^\ext
	$,
for every $x_i \in \bbX_m$, where $\tilde{\sfA}_{ij}^\ext = \sfA_{ij}^\ext$ for every $j$ where $\sfA_{ij}$ is not defined, and $\tilde{\sfA}_{ij}^\ext = \tilde{\sfA}_{ij}$ otherwise.
With this extra condition, the surrogate matrix~\cref{eq:SurrogateStiffnessMatrix} actually requires one fewer independent stencil function; i.e., $\Phi^0_T = -\sum_{\delta \in \scD(T_m)\setminus\{0\}} \Phi^\delta_T$, for every $T\in\mesh_H$.
By this definition, although $\tilde{\sfA}$ does not satisfy the zero row sum property, the matrix $\sfA -\tilde{\sfA}$ does.
Indeed,
\begin{equation}
\label{eq:ZeroRowSumPropertySurrogateDifference}
	\sfA_{ii} - \tilde{\sfA}_{ii}
			=
	-
	\sum_{j\neq i} \big(\sfA_{ij}^\ext - \tilde{\sfA}_{ij}^\ext\big)
	=
	-
	\sum_{j\neq i} \big(\sfA_{ij} - \tilde{\sfA}_{ij}\big)
	\,.
\end{equation}
In this way, the stiffness matrix coming from linearized elasticity~\cref{eq:BilinearFormLinearElasticity} is similar to the stiffness matrix coming from the Laplacian.
Indeed, the zero row sum property can be incorporated into its surrogate via a straight-forward generalization.

\section{Analyzing the surrogate discretization} \label{sec:stability_of_the_surrogate_operator}

In this section, we define and motivate what we see as some the most essential features in the analysis of our surrogate methods.
We begin with a review of discrete stability in the context of~\cref{eq:SurrogateVF}.
We then touch on the concept of spectral convergence of the surrogate matrix $\tilde{\sfA}\to\sfA$, which helps us motivate the need to control $\|\sfA-\tilde{\sfA}\|_{\max}$.
This specific quantity will repeatedly appear in the \textit{a priori} error analysis in \Cref{sec:a_priori_error_estimation}.

\subsection{Discrete stability} \label{sub:stability}
Let $S = \{v\in V \,:\, \|v\|_V = 1\}$ be the surface of the unit ball in $V$.
Recall~\cref{eq:DiscreteVariationalProblems} and assume that the discretization $\sfA\sfu = \sff$ is stable.
In the present context, this is equivalent to the existence of a constant $\alpha>0$ such that
\begin{equation}
	\sup_{v_h\in V_h\cap S} a(w_h,v_h) \geq \alpha \|w_h\|_V
	\quad
	\text{for all } w_h\in V_h
	\,.
			\end{equation}
Likewise, in order for the surrogate discretization $\tilde{\sfA}\tilde{\sfu} = \sff$ to be stable, we must show that there exists a constant $\tilde{\alpha}>0$ such that 
\begin{equation}
	\sup_{v_h\in V_h\cap S} \tilde{a}(w_h,v_h) \geq \tilde{\alpha}\|w_h\|_V
	\quad
	\text{for all } w_h\in V_h
	\,.
			\label{eq:DiscreteStability}
\end{equation}
Inequality~\cref{eq:DiscreteStability} guarantees that $\tilde{\sfA}\tilde{\sfu} = \sff$ has a unique solution and  that $\|\tilde{u}_h\|_V \leq \tilde{\alpha}^{-1}\|F\|_{V^\ast}$.
Equally important, however, it is a necessary precursor to Strang's First Lemma, which in some cases can be used, in part, to show that $\tilde{u}_h$ converges to the exact solution $u$ (see, e.g., \Cref{sub:convergence_rates_of_the_solution_in_the_h_1_norm}).

\subsection{Spectral convergence} \label{sub:general_principles}

By analyzing the singular values of $\sfA$ directly,~\cref{eq:DiscreteStability} can sometimes be proven by showing that the spectrum of $\tilde{\sfA}$ converges to the spectrum of $\sfA$ at a fast enough rate.
The main takeaway from this section is that spectral convergence can be guaranteed by showing that $\tilde{\sfA}\to \sfA$ in the matrix maximum norm, $\|\cdot\|_{\max}$.
Before moving on, denote the $k$-smallest eigenvalue of a matrix $\sfM\in\R^{N\times N}$ as $\lambda_k(\sfM)$ and let $\ell(\sfM) = \max_{1\leq i \leq N}\,\#\{\sfM_{ij} \neq 0 \text{ where } 1\leq j \leq N\}$ be the maximum number of nonzero components in $\sfM$, taken across all individual rows.

\begin{proposition}
\label{prop:SpectralConvergence}
	Let $\sfM,\sfN\in\R^{N\times N}$ be symmetric matrices.
	Then, for each $k=1,\ldots,N$, it holds that
	\begin{equation}
		|\lambda_k(\sfM) - \lambda_k(\sfN)|
		\leq
		\|\sfM-\sfN\|_{\infty}
		\,.
	\end{equation}
\end{proposition}

The proof is placed in \Cref{app:appendix}.
Because $\sfA$ and $\tilde{\sfA}$ have the same sparsity pattern, the next result follows trivially.
Note that when $n=2$, $\ell(\sfA-\tilde{\sfA}) \leq 7$, and when $n=3$, $\ell(\sfA-\tilde{\sfA}) \leq 15$,
whereas $\ell(\sfA)$ can be larger depending on the structure of the macro-mesh.

\begin{corollary}
\label{cor:BandwidthBound}
	Let $\sfA$ and $\tilde{\sfA}$ be the true and surrogate stiffness matrices in~\cref{eq:DiscreteVariationalProblems}, respectively.
	If both $\sfA$ and $\tilde{\sfA}$ are real symmetric matrices, then
		\begin{equation}
		|\lambda_k(\sfA) - \lambda_k(\tilde{\sfA})|
		\leq
		\ell(\sfA-\tilde{\sfA})\cdot
		\|\sfA-\tilde{\sfA}\|_{\max}
		\,,
		\qquad
				k=1,\ldots,N
		.
	\label{eq:BandwidthBound}
	\end{equation}
														\end{corollary}

\begin{remark}
	As stated above, if $\|\sfA-\tilde{\sfA}\|_{\max} \to 0$ fast enough, then \Cref{cor:BandwidthBound} can be used in proving the stability condition~\cref{eq:DiscreteStability}.
		However,~\cref{eq:BandwidthBound} is generally a very pessimistic bound and, when available, we recommend using more direct means to prove discrete stability (see, e.g., \Cref{cor:Coercivity}).
	Nevertheless, this result illustrates the importance of controlling $\|\sfA-\tilde{\sfA}\|_{\max}$, which is a central feature in all of the coming analysis.
\end{remark}

\subsection[Controlling the consistency error in the max norm with the variable coefficient Poisson equation]{Controlling $\|\sfA-\tilde{\sfA}\|_{\max}$ with the variable coefficient Poisson equation}

Before we begin, some new notation is required.
For any tensor $K:\Omega \to \R^{n\times n}$, define $\|K\|_{L^\infty(\Omega)} = \max_{a,b} \|K_{ab}\|_{L^\infty(\Omega)}$, likewise, for any $r\geq 0$, define $|K|_{W^{r+1,\infty}(T)} = \max_{a,b} |K_{ab}|_{W^{r+1,\infty}(T)}$.
From now on, the notation $A \lesssim B$ will be used when two mesh-dependent quantities $A$ and $B$ satisfy an inequality $A \leq C B$, where $C$ some positive $H$-independent constant.
Likewise, when $A \lesssim B$ and $B \lesssim A$, we write $A \eqsim B$.
Recall that for a macro-mesh $\mesh_H$, the diameter of a single element $T\in\mesh$ is denoted $H_T$ and the mesh size is denoted $H = \max_{T\in\mesh_H} H_T$.
We also denote the fine-scale element diameter $h_T = 2^{-m}H_T$, for each $T \in \mesh_H$ and $h = 2^{-m}H$.
\begin{lemma}
\label{lem:InfinityNormConvergence}
	Let $\sfA$ and $\tilde{\sfA}$, respectively, be the true and surrogate stiffness matrices corresponding to the bilinear form~\cref{eq:BilinearFormPoisson}.
	Namely, let each component of $\sfA$ be given by~\cref{eq:TrueStiffnessMatrixPoisson} and each component of $\tilde{\sfA}$ be defined by~\cref{eq:SurrogateStiffnessMatrix} with $G(x,u(y),v(y)) := {\nabla u(y)^\top K(x) \nabla v(y)}$.
	Fix $T\in \mesh_H$ and $0\leq r\leq q$.
	If $K_{ab}|_T\in W^{r+1,\infty}(T)$ for each index $a,b$, then
	\begin{subequations}
	\begin{equation}
		\big\|\sfA - \tilde{\sfA}\big\|_{\max,T_m}
		\lesssim
		h_T^{n-2} H_T^{r+1}|K|_{W^{r+1,\infty}(T)}
		\,,
	\label{eq:InfinityNormBoundLocal}
	\end{equation}
	where $\|\sfC\|_{{\max},T_m} = \max\big\{ |\sfC_{ij}| \,:\, x_i, x_j \in T_m \big\}$, for any matrix $\sfC$.
	Moreover, if each component $K_{ab}\in W^{r+1,\infty}(\mesh_H) = \prod_{T\in\mesh_H} W^{r+1,\infty}(T)$, then
	\begin{equation}
		\big\|\sfA - \tilde{\sfA}\big\|_{\max}
		\lesssim
		h^{n-2} H^{r+1}|K|_{W^{r+1,\infty}(\mesh_H)}
		\,.
	\label{eq:InfinityNormBoundGlobal}
	\end{equation}
	\end{subequations}
	\end{lemma}
\begin{proof}
	We prove only~\cref{eq:InfinityNormBoundLocal},~\cref{eq:InfinityNormBoundGlobal} then follows immediately.
	Recall~\cref{eq:SurrogateStiffnessMatrix} and fix $T\in\mesh_H$.
	Let $i$ and $j$ be the indices of the maximal value $|\sfA_{ij}-\tilde{\sfA}_{ij}| = \big\|\sfA - \tilde{\sfA}\big\|_{\max,T_m}$.
		Next, because the theorem trivially holds in the degenerate case $\big\|\sfA - \tilde{\sfA}\big\|_{\max,T_m} = 0$, we proceed under the assumption that $\sfA_{ij} \neq \tilde{\sfA}_{ij}$.
	Notably, it follows from~\cref{cor:Interpolation} that if each $K_{ab}|_T \in \poly_q(T)$, then $\tilde{\sfA}_{ij} = \sfA_{ij}$ and, therefore, we find ourselves in the scenario where the diffusion tensor $K|_T$ is not a polynomial (of degree at most $q$).
	Here, we may also freely assume that $i\neq j$ because of~\cref{eq:ZeroRowSumPropertySurrogateDifference}.
	Indeed, for each $i$,
		$
		|\sfA_{ii}-\tilde{\sfA}_{ii}|
		\leq
								\sum_{j\neq i}
		|\sfA_{ij}-\tilde{\sfA}_{ij}|
		\leq
		\ell(\sfA-\tilde{\sfA})\cdot \max_{j\neq i}{|\sfA_{ij}-\tilde{\sfA}_{ij}|}
			$.
	
	To fix notation in the remainder of the proof, we take $\phi = \phi_i$, and express
	\begin{equation}
		\tilde{\sfA}_{ij}
						=
		\bigg[
		\Pi^\delta_T
		\int_{\Omega_{\delta}} \nabla \phi_{\delta}(y)^\top K(\cdot+y)\, \nabla \phi(y)\hspace{0.5pt} \dd y
		\bigg](x_i)
						,
	\end{equation}
	for some nonzero $\delta\in\scD(T_m)$.
	Let $\mcI_T:C^0(\overline{T})\to \mcP_r(T)$ be the local Lagrange interpolant and define $\big[\mcI_T^{n\times n} K\big]_{ab} = \mcI_T K_{ab}$, for each index $a,b$.
	Splitting the two matrix entries $\sfA_{ij}$ and $\tilde{\sfA}_{ij}$ into polynomial and non-polynomial parts and rewriting
	\begin{align}
		\int_\Omega \nabla \phi_j(x)^\top \big[K-\mcI_T^{n\times n} K\big]\msspace(x)\sspace \nabla \phi_i(x) \dd x
		&=
		\int_{\Omega_{\delta}} \nabla \phi_{\delta}(y)^\top \big[K-\mcI_T^{n\times n} K\big]\msspace(x_i+y)\sspace \nabla \phi(y) \dd y
		\,,
		\\
		\int_\Omega \nabla \phi_j(x)^\top \big[\mcI_T^{n\times n} K\big](x)\sspace \nabla \phi_i(x) \dd x
		&=
		\int_{\Omega_{\delta}} \nabla \phi_{\delta}(y)^\top \big[\mcI_T^{n\times n} K\big] (x_i+y)\sspace \nabla \phi(y)\hspace{0.5pt} \dd y
		\,,
	\end{align}
	we find that
	\begin{equation}
	\begin{aligned}
		\sfA_{ij} &= \int_\Omega \nabla \phi_j^\top \,\mcI_T^{n\times n} K\; \nabla \phi_i \dd x + \int_{\Omega_{\delta}} \nabla \phi_{\delta}(y)^\top \big[K-\mcI_T^{n\times n} K\big](x_i+y)\sspace \nabla \phi(y)\hspace{0.5pt} \dd y,\\
		\tilde{\sfA}_{ij} &= \int_\Omega \nabla \phi_j^\top \,\mcI_T^{n\times n} K\; \nabla \phi_i \dd x + \bigg[\Pi^\delta_T\int_{\Omega_{\delta}} \nabla \phi_{\delta}(y)^\top \big[K-\mcI_T^{n\times n} K\big](\cdot+y)\sspace \nabla \phi(y)\hspace{0.5pt} \dd y\bigg](x_i)
				.
	\end{aligned}
	\label{eq:InterpolationStep}
	\end{equation}
	Upon canceling the first two terms in the expressions above, we arrive at the inequality 
	\begin{equation}
		\big|\sfA_{ij} - \tilde{\sfA}_{ij}\big| \leq |\beta_{ij}(x_i)| + |(\Pi^\delta_T\beta_{ij})(x_i)|
		\,,
	\end{equation}
		where $\beta_{ij}(x) = \int_{\Omega_{\delta}} \nabla \phi_{\delta}(y)^\top \big[K-\mcI_T^{n\times n} K\big](x+y)\sspace \nabla \phi(y)\hspace{0.5pt} \dd y$.
		Recall that the projection $\Pi^\delta_T:C^0(\overline{T}_\delta)\to\mcP(T_\delta)$ is continuous in the $L^\infty(\Omega)$ norm.
	Therefore,
	\begin{gather}
				|(\Pi^\delta_T\beta_{ij})(x_i)| \leq \|\Pi^\delta_T\beta_{ij}\|_{L^\infty(T_\delta)} \lesssim  \|\beta_{ij}\|_{L^\infty(T_\delta)}\lesssim
		\|K-\mcI_T^{n\times n} K\|_{L^\infty(T)} \|\nabla\phi_{\delta} \cdot \nabla\phi\|_{L^1(\Omega_\delta)}
				\,.
	\end{gather}
	A standard scaling argument shows that $\|\nabla\phi_{\delta} \cdot \nabla\phi\|_{L^1(\Omega_\delta)} \lesssim h_T^{n-2}$.
	This, together with the well-known property $\|K_{ab}-\mcI_T K_{ab}\|_{L^\infty(T)} \lesssim H^{r+1}_T|K_{ab}|_{W^{r+1,\infty}(T)}$, yields the sufficient result
	\begin{align}
	|\sfA_{ij} - \tilde{\sfA}_{ij}|
	\lesssim
	h_T^{n-2}H^{r+1}_T |K|_{W^{r+1,\infty}(T)}
	\,.
	\end{align}
\end{proof}

\begin{remark}
The proof of \Cref{lem:InfinityNormConvergence} can be read as a blueprint which extends to the settings of the bilinear forms $a_2(\cdot,\cdot)$, $a_3(\cdot,\cdot)$, defined in \cref{eq:BilinearFormLinearElasticity,eq:SplitBilinearFormspLaplace}. 
Indeed, when $a_2(\cdot,\cdot)$ is considered, the only significant modification to the proof above is that an interpolation operator $\mcI_T:C^0(\overline{T})\to \mcP_r(T)$ must be introduced for each Lam\'e parameter $\mu$ and $\lambda$.
The adaption to the setting $a(\cdot,\cdot) = a_3(\cdot,\cdot)$, is obvious.
Ultimately,
\begin{equation}
	\big\|\sfA - \tilde{\sfA}\big\|_{\max}
	\lesssim
	h^{n-2} H^{r+1}
	\cdot
	\left\{
	\begin{alignedat}{3}
																		\,& 		\big|\lambda\big|_{W^{r+1,\infty}(\mesh_H)} + \big|\mu\big|_{W^{r+1,\infty}(\mesh_H)}
						\qquad
		&&
		\text{if } 
		a(\cdot,\cdot) = a_2(\cdot,\cdot)
		\,,
		\\
		& \big|\sspace|\nabla\tilde{u}|^{p-2}\big|_{W^{r+1,\infty}(\mesh_H)}
				\qquad
		&&
		\text{if } 
		a(\cdot,\cdot) = a_3(\cdot,\cdot)
		\,.
	\end{alignedat}
	\right.
\label{eq:OtherMaxNormBounds}
\end{equation}
Moreover, the proof may be easily modified to permit surrogates $\tilde{\sfA}$ without the zero row sum property.
Likewise, scenarios involving fewer derivatives (which generally do not possess the zero row sum property), e.g., $a(\cdot,\cdot) = m(\cdot,\cdot)$, have similar bounds but invoke a different scaling in $h$.
\end{remark}

\section{\textit{A priori} error estimation for the variable coefficient Poisson equation} \label{sec:a_priori_error_estimation}

In this section, we present a thorough analysis of a surrogate discretization of the variable coefficient Poisson equation.
Given a load $f\in L^2(\Omega)$ and symmetric positive-definite tensor $K:\Omega \to \R^{n\times n}$, the corresponding weak form may be written as follows.
\begin{equation}
	\text{Find } u\in H^1_0(\Omega)
	\text{ satisfying }
	\quad
					a(u,v) = F(v)
	\quad
	\text{for all }
	v\in H^1_0(\Omega)
	\,,
\label{eq:PoissonWeakForm}
\end{equation}
where $a(u,v) = \int_\Omega \nabla u^\top K\sspace \nabla v \dd x$ and $F(v) = \int_\Omega f\hspace{0.25pt} v \dd x$.
As done in \Cref{sec:the_zero_row_sum_property}, define $V_{h} = \{ v \in H^1_0(\Omega) : v|_t \in \mcP_1(t) \text{ for each } t\in\mcS^m(\mesh_H)\}$ and let the corresponding vertex function basis be $\{\phi_i\}$.

\subsection{Coercivity} \label{sub:coercivity}

In the present setting, observe that each $v_h\in V_h$ can be expressed as $v_h(x) = \sum_i v_h(x_i)\phi_i(x)$.
Therefore, due to the zero row/column sum property~\cref{eq:ZeroRowSumPropertySurrogateDifference}, we find
\begin{equation}
\label{eq:DifferenceIdentity}
	\begin{aligned}
		\asurr(\fesol{v}, \fesol{w}) - \a(\fesol{v}, \fesol{w})
		&=
		\sum_{i,j} 
		\big(\tilde{\sfA}_{ij} - \sfA_{ij}\big)\sspace v_h(x_i)\sspace w_h(x_j)
												\\
		&=
		\frac{1}{2}\sum_{i\neq j} \big(\sfA_{ij} - \tilde{\sfA}_{ij}\big) (v_h(x_i) - v_h(x_j))\sspace (w_h(x_i) - w_h(x_j))
		\,.
	\end{aligned}
\end{equation}
Due to the mutual sparsity of the matrices $\tilde{\sfA}$ and $\sfA$, every nonzero term in the sum above can be rewritten as
$
	\big(\sfA_{ij} - \tilde{\sfA}_{ij}\big) (v_h(x_i) - v_h(x_i+\delta))\sspace (w_h(x_i) - w_h(x_i+\delta))
	,
$
for some nonzero $\delta$.
Because $|\delta| \eqsim h$ by construction, one easily arrives at the following upper bound:
\begin{align}
	\a(\fesol{v}, \fesol{w}) - \asurr(\fesol{v}, \fesol{w})
	\lesssim
	h^{2-n}
	\big\|\sfA - \tilde{\sfA}\big\|_{{\max}}\sspace \|\nabla v_h\|_{0}\hspace{1pt} \|\nabla w_h\|_{0}
				\,.
						\label{eq:FiniteDifferenceBound}
\end{align}
We now arrive at the main result of this subsection.

\begin{theorem}
\label{cor:Coercivity}
	Let $0\leq r\leq q$.
		Assume that $a(\cdot,\cdot)$ is coercive and $K\in \big[W^{r+1,\infty}(\mesh_H)\big]^{n\times n}$.
	Then, for any fine enough macro-mesh $\mesh_H$, the surrogate bilinear form $\tilde{a}:V_h\times V_h \to \R$ is also coercive.
\end{theorem}
\begin{proof}
	Let $S = \{v\in H^1 \,:\, \|v\|_1 = 1\}$ be the surface of the unit ball in $H^1$.
	Recall that since $a(\cdot,\cdot)$ is coercive, there exists a coercivity constant $\alpha>0$ such that $a(v,v) \geq \alpha$ for all $v\in S$.
	Notice that $\alpha \leq a(v_h,v_h) \leq \tilde{a}(v_h,v_h) + |a(v_h,v_h) - \tilde{a}(v_h,v_h)|$ for all $v_h \in V_h\cap S$ and, therefore,
	\begin{equation}
		\alpha
		-
				|a(v_h,v_h) - \tilde{a}(v_h,v_h)|
		\leq
				\tilde{a}(v_h,v_h)
		\quad
		\text{for all } v_h \in V_h\cap S
		\,.
	\label{eq:Coercivity1}
	\end{equation}
	Here, the second term on the left may be bounded from above using~\cref{eq:FiniteDifferenceBound,lem:InfinityNormConvergence} as follows,
	\begin{equation}
			|a(v_h,v_h) - \tilde{a}(v_h,v_h)|
												\lesssim
		h^{2-n}\|\sfA - \tilde{\sfA}\|_{\max} \|\nabla v_h\|_0^2
		\lesssim
		H^{r+1}|K|_{W^{r+1,\infty}(\Omega)}
		\,.
		\label{eq:Coercivity2}
	\end{equation}
		Thus, for any small enough $H$, we see that $0 < \alpha - |a(v_h,v_h) - \tilde{a}(v_h,v_h)| \leq \tilde{a}(v_h,v_h)$, as necessary.
	\end{proof}

\subsection[Convergence of the surrogate solution in the H1 norm]{Convergence of the surrogate solution in the $H^1$ norm} \label{sub:convergence_rates_of_the_solution_in_the_h_1_norm}
The purpose of this subsection is to derive a mesh-dependent upper bound on the error in the surrogate solution $\tilde{u}_h$ of the form $\|u - \approxsol{u}\|_1 \leq C(K,\Omega,u)\sspace h + \tilde{C}(K,\Omega,u)H^{r+1}$.
In doing so, we choose to emphasize the primary difference from the classical $\|u - u_h\|_1 \leq C(K,\Omega,u)\sspace h$ error estimate by absorbing the coercivity and continuity constants (which depend on both $K$ and $\Omega$) into the $\lesssim$ symbol.
We begin with a particular version of the First Strang Lemma \cite{strang1972variational}.

\begin{lemma}
\label{lem:Strang1}
	Let $S = \{v\in H^1 \,:\, \|v\|_1 = 1\}$ be the surface of the unit ball in $H^1$.
	Assume that $\tilde{a}:V_{h}\times V_{h}\to \R$ is coercive.
	The following error estimate holds for the surrogate solution $\tilde{u}_h$ of the variable coefficient model problem~\cref{eq:PoissonWeakForm}:
	\begin{equation}
		\honenorm{\weaksol{u} - \approxsol{u}}
		\lesssim
		\inf_{w_h\in V_h}
		\Big[
						\|\weaksol{u} - w_h\|_1
			+
			\sup_{v_h\in V_h\cap S}
			|\tilde{a}(w_h,v_h) - a(w_h,v_h)|
		\Big]
		\,.
	\label{eq:Strang1}
	\end{equation}
\end{lemma}

\begin{theorem}
\label{thm:APrioriBoundH1}
		Let $0\leq r\leq q$ and assume that $K \in \big[W^{r+1,\infty}(\Omega)\big]^{n\times n}$ is symmetric and positive definite with $\lambda_1(K)$ bounded away from zero almost everywhere.
		Let $u\in H^1(\Omega)$ and $\tilde{u}_h\in V_h$ be the unique solutions to~\cref{eq:ContinousVF,eq:SurrogateVF}, respectively, where $a(u,v) = \int_\Omega \nabla u^\top K\sspace \nabla v \dd x$ and $F(v) = \int_\Omega f\hspace{0.25pt} v \dd x$.
		Then, for any sufficiently fine macro-mesh $\mesh_H$, the following upper bound holds:
	\begin{equation}
		\|u - \approxsol{u}\|_{1} \lesssim h\sspace|u|_{2} + H^{r+1}|K|_{W^{r+1,\infty}(\Omega)} |u|_{1}
		\,.
	\label{APrioriBoundH1Local}
	\end{equation}
															\end{theorem}
\begin{proof}

	With the assumptions above, $a(\cdot,\cdot)$ is coercive.
	Therefore, by \Cref{cor:Coercivity}, if the macro-mesh $\mesh_H$ is taken fine enough, then $\tilde{a}:V_{h}\times V_{h}\to \R$ is coercive.
	We now bound the right-hand side of~\cref{eq:Strang1}.
	Invoking~\cref{eq:FiniteDifferenceBound}, we find
	\begin{equation}
		\honenorm{\weaksol{u} - \approxsol{u}} \lesssim \|\weaksol{u} - w_h\|_1 + h^{2-n} \big\|\sfA - \tilde{\sfA}\big\|_{\max} \|\nabla w_h\|_{0}
		\,,
	\end{equation}
	for every $w_h\in V_h$.
	Setting $w_h = \mcS\hspace{-0.4pt}\mcZ_h u$, the Scott--Zhang interpolant of $u$ \cite{scott1990finite}, we see that
	\begin{equation}
		\honenorm{\weaksol{u} - \approxsol{u}}
		\lesssim
		\|\weaksol{u} - \mcS\hspace{-0.4pt}\mcZ_h u\|_1 + h^{2-n} \big\|\sfA - \tilde{\sfA}\big\|_{\max} |\mcS\hspace{-0.4pt}\mcZ_h u|_{1}
		\lesssim
		h\sspace|u|_{2} + h^{2-n}\big\|\sfA - \tilde{\sfA}\big\|_{\max} |u|_{1}
				\,.
	\end{equation}
	In order to finish the proof, recall that $h^{2-n} \|\sfA - \tilde{\sfA}\|_{\max} \lesssim H^{r+1}|K|_{W^{r+1,\infty}(\Omega)}$, by \Cref{lem:InfinityNormConvergence}.~
\end{proof}

\subsection[Convergence of the surrogate solution in the L2 norm]{Convergence of the surrogate solution in the $L^2$ norm} \label{sub:convergence_of_the_solution_in_the_}
In this subsection, we prove an $L^2$ error estimate of the form $\|u - \approxsol{u}\|_0 \leq C(K,\Omega,u)\sspace h^2 + \tilde{C}(K,\Omega,u)H^{r+1}$.
A second result, which elicits accelerated $H$-convergence, is also proved under the additional assumption $\sum_{x_i,\in T^\delta_m} \big[\Phi^\delta_T-\Pi^\delta_T\Phi^\delta_T\big](x_i) = 0$, where $T^\delta_m = T_m \cap \overline{T}_\delta$.
This is a property which naturally arises for the specific class of least-squares projections introduced in \Cref{sub:PLSR}.
Again, we emphasize the primary differences from the corresponding classical error estimate by absorbing the standard constants into the $\lesssim$ symbol.
\begin{theorem}
\label{thm:APrioriBoundL2}
	Under the conditions of \Cref{thm:APrioriBoundH1}, if $\Omega\subset \R^n$ is a convex domain,
	 			then the following additional upper bound on the error in the surrogate solution holds:
	\begin{subequations}
	\label{eq:APrioriBoundL2}
	\begin{equation}
		\|u - \approxsol{u}\|_0
		\lesssim
		h^2\sspace|u|_{2} + H^{r+1}|K|_{W^{r+1,\infty}(\Omega)} |u|_{1}
				\,.
	\label{eq:APrioriBoundL2_1}
	\end{equation}
		Moreover, if $r>0$ and $\sum_{x_i\in T^\delta_m} \big[\Phi^\delta_T-\Pi^\delta_T\Phi^\delta_T\big](x_i) = 0$, for each $T\in\mesh_H$ and $\delta\in\scD(\mathring{T}_m)$, then
	\begin{equation}
												\|u - \approxsol{u}\|_0
		\lesssim
		h^2\sspace|u|_{2} + H^{r+2}|K|_{W^{r+1,\infty}(\Omega)} \|\nabla u\|_{1}
		\,.
	\label{eq:APrioriBoundL2_2}
	\end{equation}
	\end{subequations}
\end{theorem}
\begin{proof}
	By the triangle inequality, $\|u-\tilde{u}_h\|_0 \leq \|u-u_h\|_0 + \|u_h-\tilde{u}_h\|_0$,
										where $u_h\in V_h$ is the discrete solution coming from~\cref{eq:DiscreteVF}.
	It can be shown that if $\Omega$ is convex, then $u\in H^2(\Omega)\cap H^1_0(\Omega)$; see, e.g., \cite{grisvard2011elliptic}.
															It then follows from standard arguments that $\|u-u_h\|_0 \lesssim h^2\sspace |u|_2$; see, e.g., \cite[Theorem~5.7.6]{brenner2007mathematical}.
		Therefore, we only need to analyze the term $\|u_h-\tilde{u}_h\|_0$.
	Since, $\|u_h-\tilde{u}_h\|_0 \leq \|u_h-\tilde{u}_h\|_1$, proceeding as in the proof of \Cref{thm:APrioriBoundH1}, we quickly arrive at~\cref{eq:APrioriBoundL2_1}.

	In order to prove~\cref{eq:APrioriBoundL2_2}, first define $w_h\in V_h$ satisfying $a(w_h,v_h) = (u_h-\tilde{u}_h,v_h)_\Omega$, for all $v_h\in V_h$.
	Observe that the exact solution of the problem $a(w,v) = (u_h-\tilde{u}_h,v)_\Omega$, for all $v\in H^1_0(\Omega)$, belongs to the space $H^2(\Omega)\cap H^1_0(\Omega)$, $\|w\|_2\lesssim\|u_h-\tilde{u}_h\|_{0}$.
	Moreover,
																				\begin{equation}
		\begin{aligned}
		\|u_h-\tilde{u}_h\|_0^2
		&=
		a(w_h,u_h-\tilde{u}_h)
		=
		\tilde{a}(w_h,\tilde{u}_h) - a(w_h,\tilde{u}_h)
		\\
		&=
		\frac{1}{2}
		\sum_{i\neq j} (\sfA_{ij} - \tilde{\sfA}_{ij}) (\tilde{u}_h(x_i) - \tilde{u}_h(x_j)) (w_h(x_i) - w_h(x_j))
		\,,
		\end{aligned}
	\label{eq:SecondTermManipulationAubinNitsche}
	\end{equation}
	where the final line follows from~\cref{eq:DifferenceIdentity}.
	As remarked previously, each nonzero term in this sum can be written as $\big(\sfA_{ij} - \tilde{\sfA}_{ij}\big) (\tilde{u}_h(x_i) - \tilde{u}_h(x_i+\delta))\sspace (w_h(x_i) - w_h(x_i+\delta))$, for some $T\in\mesh_H$ and nonzero $\delta \in \scD(T_m)$.
	We can make better use of this expression with the identity $v_h(x_i) - v_h(x_i+\delta) = \nabla v_h(y_{i,\delta})$, wherein each $y_{i,\delta}$ is chosen from the edge connecting $x_i$ and $x_i+\delta$, and with the relationship $\sfA_{ij} - \tilde{\sfA}_{ij} = \big[\Phi^\delta_T-\Pi^\delta_T\Phi^\delta_T\big](x_i)$, since $\sfA_{ij} - \tilde{\sfA}_{ij} \neq 0$ and $i\neq j$.
	With these observations in hand, we have
	\begin{align}
		2\|u_h&-\tilde{u}_h\|_0^2
		=
		\sum_{T\in\mesh_H}
		\sum_{\delta\in \scD(\mathring{T}_m)}
		\sum_{x_i\in T^\delta_m}
		\big[\Phi^\delta_T-\Pi^\delta_T\Phi^\delta_T\big](x_i)\,
		(\nabla \tilde{u}_h(y_{i,\delta}) \cdot \delta)(\nabla w_h(y_{i,\delta})\cdot \delta)
		\\
		&\leq
		\sum_{T\in\mesh_H}
		\sum_{\delta\in \scD(\mathring{T}_m)}
		\sum_{x_i\in T^\delta_m}
		\big[\Phi^\delta_T-\Pi^\delta_T\Phi^\delta_T\big](x_i)
		\Big(
		(\nabla \tilde{u}_h(y_{i,\delta}) \cdot \delta)(\nabla w_h(y_{i,\delta})\cdot \delta) - C
		\Big)
		\\
		&\lesssim
		h^{-n} 
		\|\sfA-\tilde{\sfA}\|_{\max}
		\sum_{T\in\mesh_H}
		\sum_{\delta\in \scD(\mathring{T}_m)}
		\|(\nabla \tilde{u}_h \cdot \delta)(\nabla w_h\cdot \delta)-C\|_{L^1(T)}
		\,.
	\end{align}
																						If we set the constant above to the following average value, $C := \frac{1}{\mathrm{vol}(T)} \int_T(\nabla u \cdot \delta)(\nabla w\cdot \delta) \dd x$, then $\|(\nabla u \cdot \delta)(\nabla w\cdot \delta) - C\|_{L^1(T)} \lesssim H_T\sspace |(\nabla u \cdot \delta)(\nabla w\cdot \delta)|_{W^{1,1}(T)}$.
											Therefore,
										\begin{align}
		\|(\nabla \tilde{u}_h \cdot \delta)&(\nabla w_h\cdot \delta)-C\|_{L^1(T)}
		\leq
		\|(\nabla \tilde{u}_h \cdot \delta)(\nabla w_h\cdot \delta) - (\nabla u \cdot \delta)(\nabla w_h\cdot \delta)\|_{L^1(T)}
		\\
		&+
		\| (\nabla u \cdot \delta)(\nabla w_h\cdot \delta) - (\nabla u \cdot \delta)(\nabla w\cdot \delta)\|_{L^1(T)}
		+
		\|(\nabla u \cdot \delta)(\nabla w\cdot \delta) - C\|_{L^1(T)}
		\\
		&\lesssim
		\|\nabla (u - \tilde{u}_h) \cdot \delta\|_{0,T}\,\|\nabla w_h\cdot \delta\|_{0,T}
		+
		\|\nabla (w - w_h)\cdot \delta\|_{0,T}\, \|\nabla u \cdot \delta\|_{0,T}
		\\
		&+
		H_T\sspace\|\nabla u \cdot \delta\|_{1,T}\sspace \|\nabla w\cdot \delta\|_{1,T}
		\,.
	\end{align}
	After summing over all $T\in \mesh_H$ and $\delta\in\scD(T_m)$ and taking into account $|\delta| \eqsim h$, we arrive at the bound
																																													\begin{align}
		\|u_h-\tilde{u}_h\|_0^2
		&\lesssim
		h^{2-n} 
		\|\sfA-\tilde{\sfA}\|_{\max}
		\big(
		\|u - \tilde{u}_h\|_{1}\, | w_h|_{1}
		+
		\|w - w_h\|_{1}\, |u|_{1}
		+
		H\|\nabla u\|_{1} \|\nabla w\|_{1}
		\big)
		\\
		&\lesssim
		h^{2-n} 
		\|\sfA-\tilde{\sfA}\|_{\max}
		\big(
		h\,|u|_2 + H^{r+1}|K|_{W^{r+1,\infty}(\Omega)}|u|_1 + H\, \|\nabla u\|_{1}
		\big)
		\sspace
		\|w\|_{2}
		\\
		&\lesssim
		H^{r+1}|K|_{W^{r+1,\infty}(\Omega)}
		\big(
		H\, \|\nabla u\|_{1} + H^{r+1}|K|_{W^{r+1,\infty}(\Omega)}|u|_1
		\big)
		\sspace
		\|w\|_{2}
		\,.
	\end{align}
		The proof is completed by recalling that $\|w\|_2\lesssim\|u_h-\tilde{u}_h\|_{0}$.
																																			\end{proof}

\begin{remark}
	As stated previously, the error estimates in \Cref{thm:APrioriBoundH1,thm:APrioriBoundL2} not do track the constants appearing in the corresponding classical estimates.
		When $r>0$, none of the constants in either classical estimate depend on the higher order norms $|K|_{W^{r+1,\infty}(\Omega)}$.
		\end{remark}

\begin{remark}
The main ingredients in the error analysis presented in this section are \Cref{lem:InfinityNormConvergence}, Strang's First Lemma, and~\cref{eq:DifferenceIdentity}, which simply follows from the zero row sum property and symmetry.
Simple generalizations of \Cref{lem:InfinityNormConvergence}, for the bilinear forms $a_2(\cdot,\cdot)$ and $a_3(\cdot,\cdot)$ have been stated in~\cref{eq:OtherMaxNormBounds}.
Meanwhile,~\cref{eq:DifferenceIdentity} simply follows from the zero row sum property of $\tilde{\sfA}$ and symmetry.
Therefore, we see no reason to doubt that our analysis here can be generalized to the other problems of interest in this paper.
Hence, we proceed with numerical verification and demonstration.
\end{remark}

\section{Implementation} \label{sec:implementation}
The performance of a numerical method largely depends on its implementation.
Therefore, in this section, we highlight the important features of ours.
We use the HyTeG finite-element software framework \cite{Kohl2018HyTeGfinite} as the core framework for all the numerical experiments in \Cref{sec:numerical_experiments}.
It offers efficient distributed data structures for simplicial meshes in 2D and 3D, which serve as a basis for the implementation of massively parallel fast iterative solvers.
Its main concept is based on the idea that a coarse input mesh is split into its geometrical primitives, i.e., vertices, edges, and faces, and each of these primitives is uniformly refined. Because the primitives of the same dimension are decoupled from the others, all primitives of the same dimension may be processed in parallel.
This partitioning and the hierarchy of locally structures meshes allows for efficient parallel implementations of geometric multigrid methods.
More importantly, these data structures fit perfectly to the concept of macro-elements introduced in \Cref{sub:stencil_functions_on_uniformly_structured_meshes}.
The problems in this paper are mainly solved by employing a geometric multigrid solver using V-cycles with a hybrid Gauss--Seidel smoother. On the coarsest grid, either a preconditioned conjugate gradient method or the direct solver  MUMPS \cite{MUMPS01,MUMPS02}, as provided by the PETSc interface \cite{petsc-user-ref,petsc-efficient}, is used. For improved parallel scalability of the coarse grid solver, agglomeration techniques as provided by \texttt{PCTELESCOPE} \cite{MaySananRuppKnepleySmith2016} are used in runs with many processes.

\subsection{Polynomial least squares regression}
\label{sub:PLSR}
An important factor in the performance of the surrogate approach is the approximation of the stencil functions $\Phi^\delta_T$ by polynomials $\tilde{\Phi}^\delta_T$.
This step in the solver process must be very fast and is usually done in a pre-process step before the actual solve.
After various preparatory experiments, we have seen satisfactory performance and accuracy from simply computing $\tilde{\Phi}^\delta_T = \Pi^\delta_T \Phi^\delta_T$ via solving a simple least-squares problem, which we now describe.

Let $T\in\mesh_H$ be a macro-element and recall that $T_m$ is the associated lattice on level $m$. Suppose that $\Phi_T^\delta$ is the stencil function in direction $\delta \in \scD(T_m)$ which we want to approximate.
For the least-squares regression, we fix a level $m_{\mathrm{LS}}$ with $m \geq m_{\mathrm{LS}} \geq 2$ and define the set of least-squares points ${T}{}^\delta_{\mathrm{LS}} \coloneqq {T}_{m_{\mathrm{LS}}}\cap \overline{T}_\delta$.
Furthermore, let $\{p_k\}_{k=1}^{M}$ be a basis of $\mathcal{P}_q(T)$, the space of polynomials with maximal degree $q$.
Assume that $m_{\mathrm{LS}}$ is chosen large enough such that $\big|{T}{}^\delta_{\mathrm{LS}}\big| \geq M$ and introduce the following norm on $\mathcal{P}_q(T)$: $\|p\|^2_{{T}{}^\delta_{\mathrm{LS}}} \coloneqq \sum_{x_i \in {T}{}^\delta_{\mathrm{LS}}} p(x_i)^2$.
The least-squares regression problem, which in turn defines $\Pi^\delta_T$, is formalized as follows:
\begin{align}
\label{eqn:leastsquaresproblem}
\text{Find } \sfc \in \mathbb{R}^M \text{ satisfying }
\quad
\sfc = \argmin_{\sfd \in \mathbb{R}^M} \left\|\Phi^\delta_{T} - \sum_{k=1}^{M} d_k p_k\right\|^2_{{T}{}^\delta_{\mathrm{LS}}}
\,.
\end{align}
The approximated stencil function is then defined as $\tilde{\Phi}^\delta_{T} \coloneqq \sum_{k=1}^{M} c_k p_k$.
This problem is equivalent to solving the possibly overdetermined linear system of
equations $\sfB \sfc = \sff$ in a least-squares sense,
where $\sfB_{ij} = p_j(x_i)$ and $\sff_i = \Phi^\delta_{T}(x_i)$ for $1 \leq i \leq |{T}{}^\delta_{\mathrm{LS}}|$ and $1 \leq j \leq N$. The choice of the polynomial basis is arbitrary. However, for an easier implementation, we employ the monomial basis, even knowing that the resulting linear system is ill conditioned. Since it is crucial to get numerically precise results, a stable solver for this problem has to be chosen. For this purpose, we apply the
\texttt{colPivHouseholderQr} method from the Eigen 3.3.5 library \cite{eigenweb}, which offers a good balance between speed and accuracy. Obviously, each of these linear systems is independent of others, therefore they may be solved in parallel.
\begin{remark}
\label{rem:OptimalityCondition}
Taking into account $\tilde{\Phi}^\delta_{T} = \sum_{k=1}^{M} c_k p_k$, the first order optimality condition for~\cref{eqn:leastsquaresproblem} can be stated as $\sum_{x_i\in{T}{}^\delta_{\mathrm{LS}}} \big[\Phi^\delta_T-\Pi^\delta_T\Phi^\delta_T\big](x_i) = 0$.
If $m=m_{\mathrm{LS}}$, then the secondary assumption in \Cref{thm:APrioriBoundL2} is satisfied and we see higher order convergence in $H$, as stated in~\cref{eq:APrioriBoundL2_2}.
Usually, when $m_{\mathrm{LS}}$ is close but not equal to $m$, we see preasymptotic $H$-convergence in between the two estimates given in~\cref{eq:APrioriBoundL2}; see \Cref{fig:varying_hls}.
\end{remark}
\begin{remark}
\label{remark:symmetricregression}
In the case where the bilinear form $a(\cdot,\cdot)$ is symmetric, we need only approximate a single stencil function $\Phi^\delta_T$ for both directions $\delta$ and $-\delta$.
Indeed, as observed in \Cref{rem:Symmetry}, the corresponding stencil functions are identical, up to a shift by $\delta$.
Furthermore, the symmetry requirement \cref{eq:SymmetricInterpolants}, from \Cref{rem:SurrogateSymmetry}, is satisfied with the projection operator defined above.
Indeed, one may verify that for every $\delta$, $T_\delta = T_{-\delta} - \delta$.
Therefore,
\begin{equation}
	\Big\|\Phi^\delta_{T} - \tilde{\Phi}^\delta_{T}\Big\|^2_{{T}{}^\delta_{\mathrm{LS}}}
	=
	\Big\|\Phi^{-\delta}_{T} - \tilde{\Phi}^{-\delta}_{T}\Big\|^2_{{T}{}^{-\delta}_{\mathrm{LS}}}
	\qquad
	\text{and}
	\qquad
	\tilde{\Phi}^\delta_{T}(x) = \tilde{\Phi}^{-\delta}_{T}(x+\delta)
	\,.
\end{equation}
Thus, on simplicial meshes in 2D, only four instead of seven polynomials per macro-element have to be determined and stored in memory. In some cases, where the zero row sum property holds, the number of required polynomials may be even reduced to three.
\end{remark}
\subsection{Fast polynomial evaluation}
An even more important factor with respect to the performance of our implementation is the fast evaluation of the surrogate stencil functions $\tilde{\Phi}^\delta_{T}$. Contrary to the computation of each $\tilde{\Phi}^\delta_{T}$, which will happen only once per solve, these evaluations will be made during every matrix-vector multiplication. Therefore, the costs of evaluating the stiffness matrix entries associated to a degree of freedom, may not exceed the costs of evaluating the bilinear forms with the respective ansatz functions. In this case not only the reduction of floating point operations per degree of freedom is of importance, but also the required memory traffic has to be taken into account.

When performing a matrix-vector multiplication in HyTeG, the degrees of freedom in a macro-element are processed in a row-wise fashion as illustrated in the left of \cref{fig:poly_eval_1d}. In each row, the stencil function may be interpreted as a 1D function. We assume without loss of generality, that the 1D stencil functions are aligned with the $x$-axis. This property is also inherited by the approximated stencil function.
\begin{figure}\centering
\includegraphics[width=0.33\linewidth]{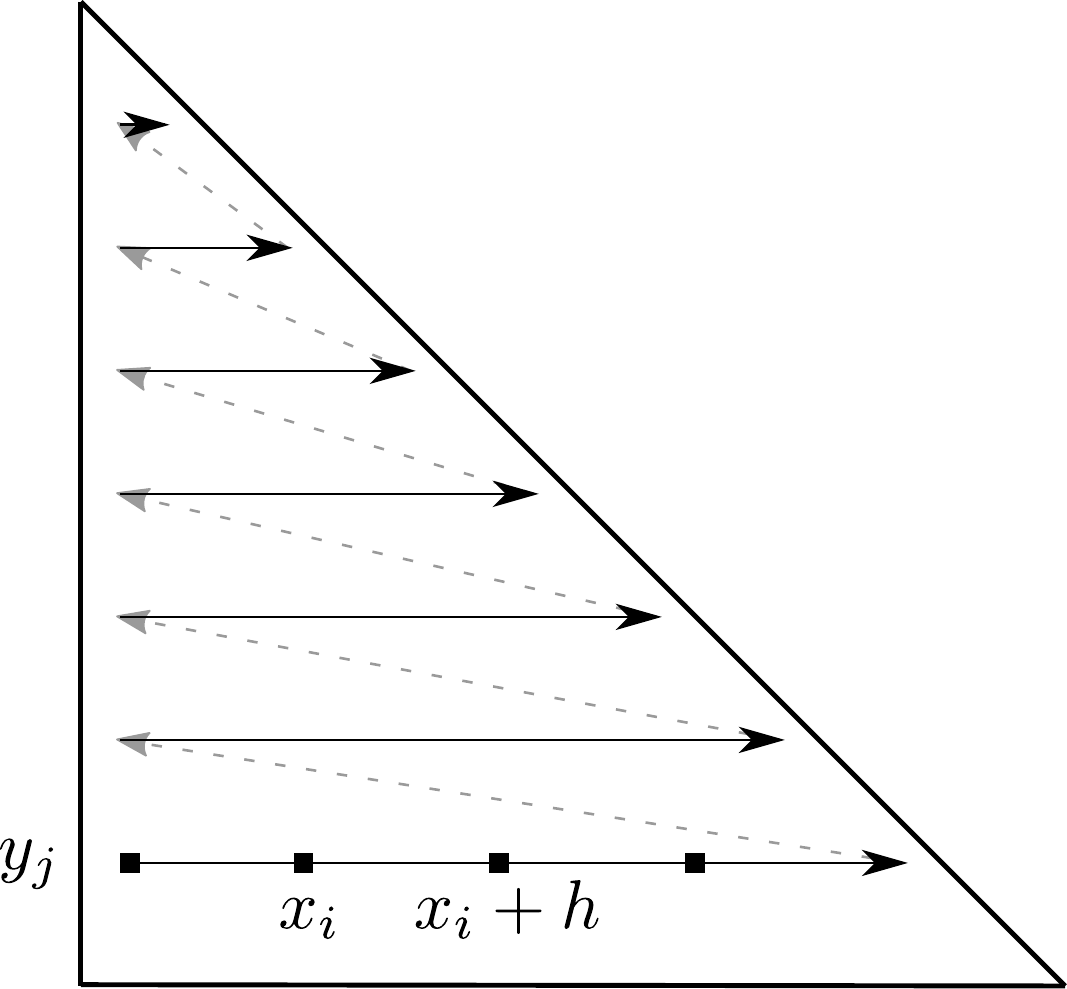}\hspace*{2em}
\includegraphics[width=0.375\linewidth]{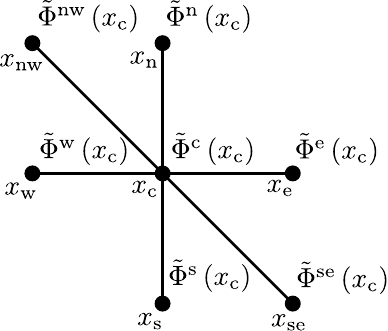}
\caption{\label{fig:poly_eval_1d} Illustration of a loop through the degrees of freedom in a macro-element. In each row of the loop, the 2D stencil function may be interpreted as a 1D function (left). Seven stencil functions have to be evaluated in order to obtain the whole stencil for a degree of freedom (right).}
\end{figure}
To further optimize the evaluation of the 1D polynomial, we exploit that the stencil functions have to be evaluated on a line subdivided into uniformly sized intervals of length $h$. Let $(x_i, y_j)$ be a vertex node in the lattice and let $p_{y_j}(\cdot) \coloneqq \tilde{\Phi}_\delta(\cdot, y_j)$ be the approximated 1D stencil function associated to row $y_j$.

Assuming that we already have evaluated the stencil function $p_{y_j}$ at a point $x_i$, we want to evaluate it at the next point $x_i + h$ as efficiently as possible. Since the grid points are equidistantly distributed, we can use a special case of the divided differences, called forward differences \cite[pg. 126]{RichardL.Burden2015Numerical}.

First, we need $q+1$ helper variables $\left\{\Delta_{x_0}^{(k)}\right\}$ for $k \in \{0,1,\ldots,q\}$ which are defined in a preprocessing step as follows:
\begin{align}
\Delta^{(0)}_{x_0} &\coloneqq p_{y_j}(x_0)\,,\\
\Delta^{(k)}_{x_0} &\coloneqq \Delta_{x_0+h}^{(k-1)} - \Delta_{x_0}^{(k-1)}, \quad k \in \{1,\ldots,q\}\,.
\end{align}

The value at position $p_{y_j}(x_0)$ is then given by $\Delta_{x_0}^{(0)}$. In order to obtain the value at $p_{y_j}(x_0+h)$, one has to update all the helper variables in the following way:
\begin{align}
\Delta_{x_0}^{(k)} &\coloneqq \Delta_{x_0}^{(k)} + \Delta_{x_0}^{(k+1)}, \quad k \in \{0,1,\ldots,q\}\,.
\end{align}
After that, the value of $p_{y_j}(x_0+h)$ is given by $\Delta_{x_0}^{(0)}$. Doing this recursively yields the approximated stencil function values at all mesh points on a single row using only $q+1$ helper variables and $q+1$ floating point additions.
When iterating through a row, in general seven polynomial evaluations, one for each direction, are required, cf.~right of \cref{fig:poly_eval_1d}. In the symmetric case this may be reduced to six polynomial evaluations, since the western stencil weight may be obtained from the previous eastern evaluation.
Keep also in mind that in the symmetric case the polynomials of approximated stencil functions in opposite directions are the same but only evaluated at different positions, cf. \cref{remark:symmetricregression}.
Therefore, $6\cdot(q+1)$ helper variables are required for a single row. For $q=8$, these results in $54\cdot8$ bytes of memory which fits easily into a modern L1 CPU cache. When moving from one lattice point to another, $6\cdot(q+1)$ vectorizable floating point additions have to be performed, to obtain the updated polynomial evaluations.
Furthermore, in our implementation, the polynomial degree is realized as a C++ template parameter, therefore, all loops concerning the evaluation of a polynomial of a certain degree may be optimized at compile time.
Since our focus lies in the theoretical analysis of the surrogate approach, thorough performance studies employing performance models should be considered beyond the scope of this paper.
Similar performance studies have been carefully completed in \cite{bauer2018new,bauer2017two}.
Thus, in the next section, we only report on relative run-times of the surrogate approach compared to the standard method, also implemented on HyTeG, using on-the-fly quadrature of the integrals stemming from the bilinear forms.

\section{Numerical experiments} \label{sec:numerical_experiments}
In this section, we numerically verify \Cref{thm:APrioriBoundH1,thm:APrioriBoundL2}, both related to the variable coefficient Poisson equation.
Additionally, we present proof-of-concept results for a linearized elasticity application and a simple $p$-Laplacian diffusion problem.
While not covered by the theory, we include these latter examples to demonstrate the breadth of generality of the methodology.

All run-time measurements in this sections were obtained on a machine equipped with two Intel\textsuperscript{\textregistered} Xeon\textsuperscript{\textregistered} Gold 6136 processors with a nominal base frequency of 3.0 GHz.
Each processor has 12 physical cores which results in a total of 24 physical cores.
The total available memory of \SI{251}{\giga\byte} is split into two NUMA domains, one for each socket.
We use the GCC 7.3.0 compiler and specify the following compiler arguments: \texttt{-O3 -march=native}.
All the examples in this section were executed in parallel using all available 24 physical cores.

When comparing run times from the standard and the surrogate approaches, many factors are responsible for the relative speed-up of the surrogate approach.
Increasing the polynomial order $q$ of the surrogate stencil functions not only increases the run time of a multigrid iteration but also the time spent in the setup phase (i.e., computing each $\Phi_T^\delta$).
The cost of the setup phase, however, is mostly dominated by the sampling level $m_{\mathrm{LS}}$.
Therefore, when the ratio of time spent in the iterative solver to the time spent in the setup phase(s) for solving a particular problem is large, the setup cost is almost negligible and we see the best performance.
Since the problems in the following subsections differ in complexity and have different ratios of solver to setup time, the observed relative speed-ups are not directly comparable.
Nonetheless, the reported speed-ups for all tested examples range between a factor of 14 and 20.
Such significant speed-ups are in particular important in case of dynamic or stochastic applications.
Most stochastic applications demand an enormous number of deterministic solves resulting quite often in extreme long run times.
Having such a surrogate approach at hand can help to make stochastic approaches such as, e.g., multilevel Monte Carlo and its variants, more accessible for complex applications.

\subsection{Quantitative benchmark problem}
\label{sub:quantitativebenchmark}
In this subsection, we examine the surrogate method for the variable coefficient Poisson equation which has been described and analyzed above.
The strong form of the problem is
\begin{equation}
\begin{alignedat}{2}
-\mathrm{div}\left(K \nabla u\right) &= f &&\quad \text{in } \Omega,\\
u &= g &&\quad \text{on } \partial \Omega.
\end{alignedat}
\label{eqn:strongform}
\end{equation}
We consider both the bilinear form coming from the scalar coefficient scenario (i.e., $K = k\cdot \mathrm{Id}$), introduced in~\cref{eq:BilinearFormPoisson}, and the tensorial coefficient scenario, introduced in~\cref{eq:BilinearFormPoissonTensor}.
In the scalar coefficient experiments, we use the unit-square domain $\Omega = (0,1)^2$.
In the tensorial coefficient experiments, the domain $\Omega$ has a curvilinear boundary.

\begin{figure}\centering
\begin{minipage}{0.25\linewidth}
	\centering
	\includegraphics[width=\linewidth]{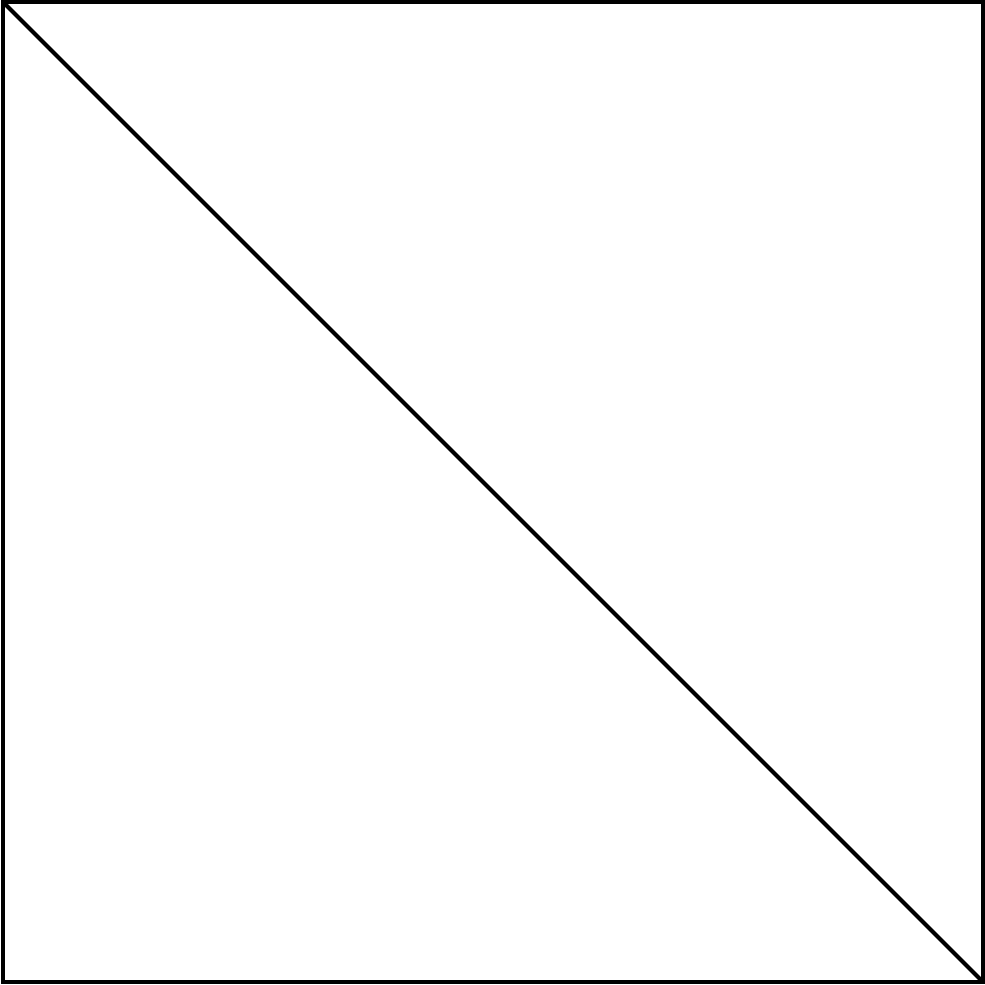}
\end{minipage}\hspace*{2em}
\begin{minipage}{0.25\linewidth}
	\centering
	\includegraphics[width=\linewidth]{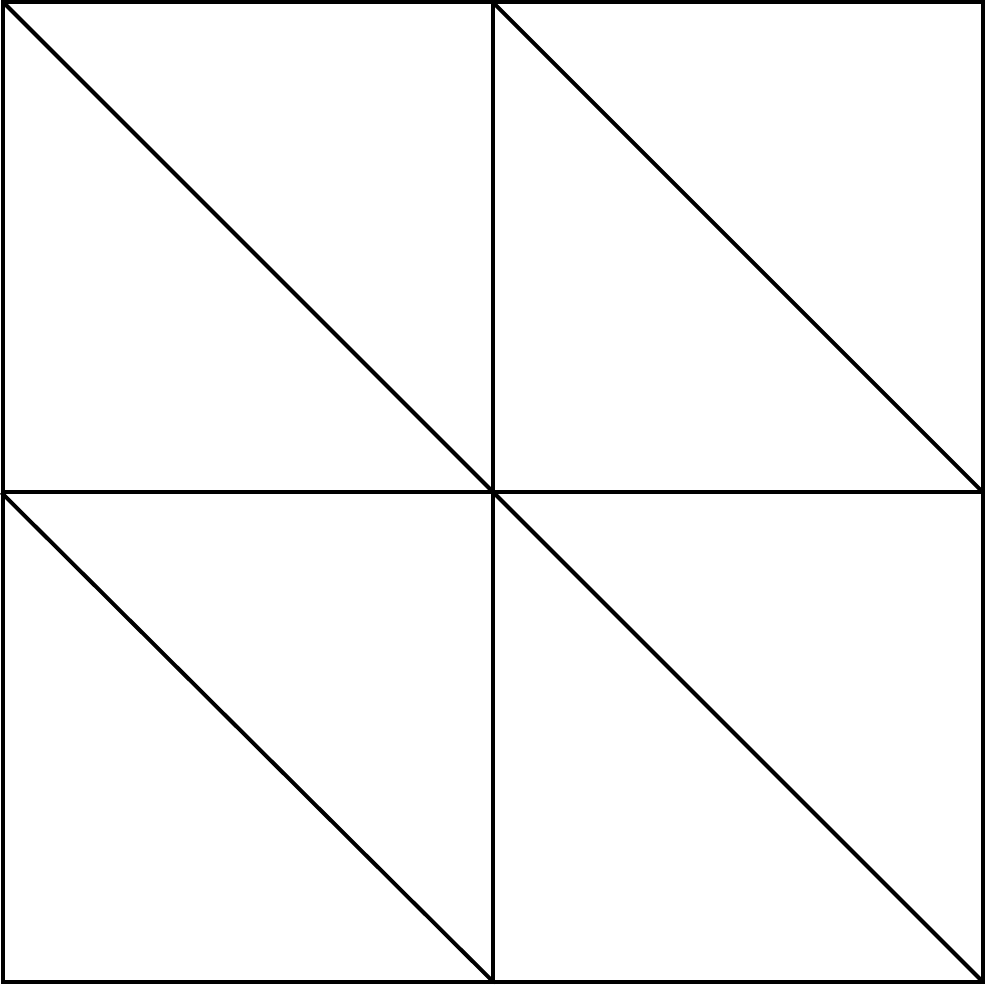}
\end{minipage}\hspace*{2em}
\begin{minipage}{0.25\linewidth}
	\centering	
	\includegraphics[width=\linewidth]{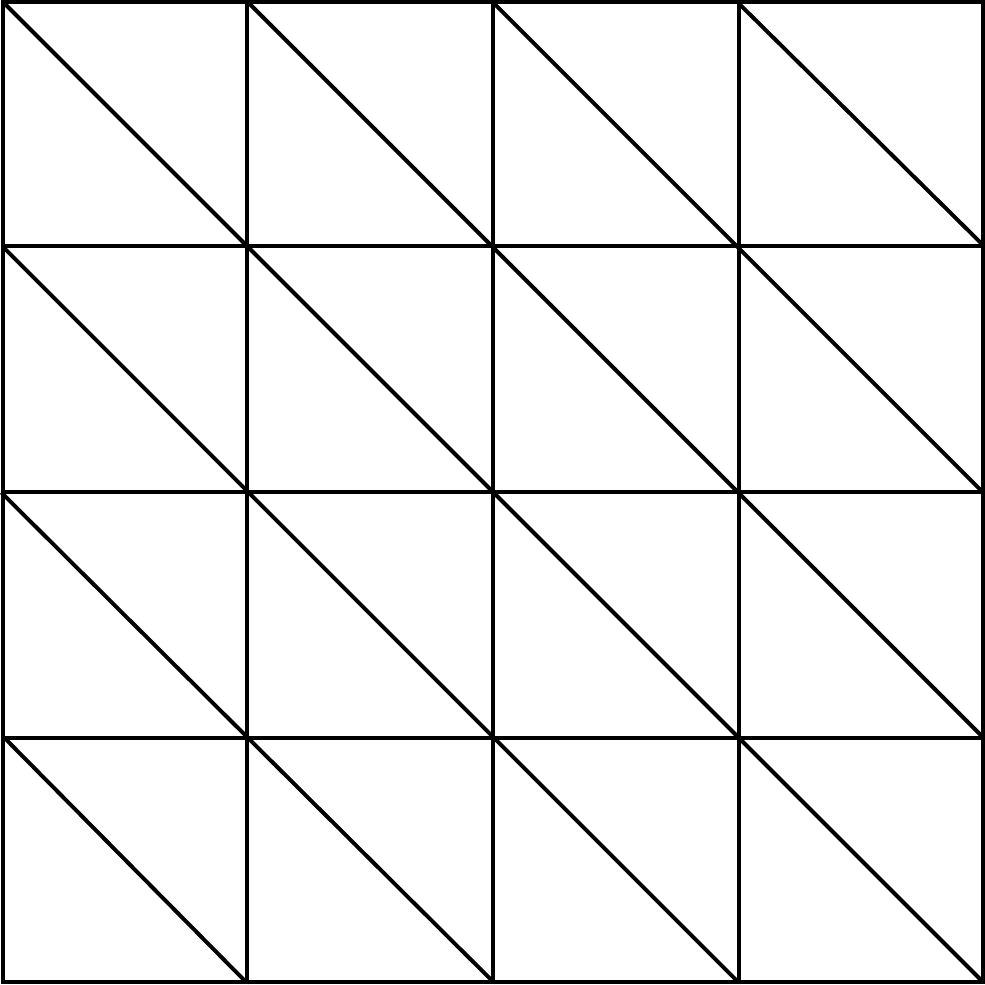}
\end{minipage}
\caption{\label{fig:unitsquare_meshes} Coarse macro-meshes of the unit-square with mesh sizes $H = H_0$ (left), $H = \nicefrac{H_{0}}{2}$ (middle), and $H = \nicefrac{H_{0}}{4}$ (right). Meshes with a smaller $H$ follow the same uniform refinement pattern.}
\end{figure}

\subsubsection{Scalar coefficient on unit square}
\label{sec:scalar_coefficient_benchmark}
In the first benchmark problem ($K = k\cdot \mathrm{Id}$), we take $\Omega = (0,1)^2$ and employ the scalar coefficient function
\begin{align}
k(x,y) = \exp{\left(xy\right)} + \sin{\left (3\pi x y \right )} + \cos{\left (\pi x^{2} y \right )} + 1
\label{eq:ScalarCoefficient}
\end{align}
in problem \cref{eqn:strongform}.
The manufactured solution $u$ is chosen as $u(x,y) = \sin(x)\sinh(y)$. The restriction of $u$ to the boundary is chosen as Dirichlet datum $g$. The right-hand-side $f$ is directly computed by inserting $u$ into the equation.
In this benchmark, we fix the finest mesh size $h$ and report on the errors depending on $H$ and $q$ to show the proven $\mathcal{O}\left(H^{q+1}\right)$ estimate in the $H^1$ norm and $\mathcal{O}\left(H^{q+2}\right)$ in the $L^2$ norm. For this purpose, $h$ is chosen to be very small in order for the error to be mostly dominated by the surrogate part.

The reference macro-mesh size is given by $H_0$, as illustrated in the left of~\cref{fig:unitsquare_meshes}. All finer macro-meshes, with associated mesh sizes $H<H_0$, stem from uniformly refining this reference mesh; see middle and right of~\cref{fig:unitsquare_meshes}. The fine mesh, with associated mesh size $h\ll H$, is the 13 times uniformly refined reference macro-mesh, i.e.,~$h = 2^{-13} H_0$. 
This fine mesh, has about \num{6.71e7} degrees of freedom.
The approximation of the stencil functions through least-squares regression, is done on the mesh associated to mesh size $H_{\mathrm{LS}} = 2^{-8} H$.
Note that this keeps the number of sampling points in each macro-element constant to \num{32639}.
Each linear system is solved by applying geometric multigrid V(2,2) iterations until a relative residual of \num{1e-13} is obtained.

\begin{table}
\centering
\caption{\label{tab:capitalhconvergence_scalar_h1}Relative $H^1$ errors and experimental orders of convergence for fixed $h$ and varying $q$ and $H$ in the case of problem~\cref{eqn:strongform} with the scalar coefficient~\cref{eq:ScalarCoefficient}. Here, the relative $H^1$ error with the classical FEM is \num{1.23e-08}.}
\resizebox{\linewidth}{!}{\begin{tabular}{r|c|c|c|c|c|c|c|c}	\toprule & \multicolumn{2}{c|}{$q = 1$} & \multicolumn{2}{c|}{$q = 2$} & \multicolumn{2}{c|}{$q = 3$} & \multicolumn{2}{c}{$q = 4$}\\
	\multicolumn{1}{c|}{$\frac{H}{H_0}$} & rel. $H^1$ err. & eoc & rel. $H^1$ err. & eoc & rel. $H^1$ err. & eoc & rel. $H^1$ err. & eoc \\\midrule
	\begin{filecontents*}{./results/benchmarks/results/scalar_h1.csv}
capitalh,qOne,eocOne,qTwo,eocTwo,qThree,eocThree,qFour,eocFour
$2^{-1}$,\num{4.58E-02},--,\num{2.24E-02},--,\num{4.52E-03},--,\num{1.68E-03},--
$2^{-2}$,\num{1.61E-02},1.50,\num{2.75E-03},3.02,\num{4.46E-04},3.34,\num{4.70E-05},5.16
$2^{-3}$,\num{4.43E-03},1.86,\num{3.74E-04},2.88,\num{2.88E-05},3.95,\num{1.59E-06},4.89
$2^{-4}$,\num{1.15E-03},1.95,\num{4.86E-05},2.94,\num{1.84E-06},3.97,\num{5.22E-08},4.93
$2^{-5}$,\num{2.90E-04},1.98,\num{6.17E-06},2.98,\num{1.17E-07},3.98,\num{1.23E-08},2.09
\end{filecontents*}
\csvreader[late after line=\\,late after last line=\\\bottomrule,head to column names]{./results/benchmarks/results/scalar_h1.csv}{}
	{\capitalh & \qOne & \eocOne & \qTwo & \eocTwo & \qThree & \eocThree & \qFour & \eocFour}
\end{tabular}}
\end{table}

\begin{table}
\centering
\caption{\label{tab:capitalhconvergence_scalar_l2}Relative $L^2$ errors and experimental orders of convergence for fixed  $h$ and varying $q$ and $H$ in the case of problem~\cref{eqn:strongform} with the scalar coefficient~\cref{eq:ScalarCoefficient}. Here, the relative $L^2$ error with the classical FEM is \num{4.10e-09}.}
\resizebox{\linewidth}{!}{\begin{tabular}{r|c|c|c|c|c|c|c|c}	\toprule & \multicolumn{2}{c|}{$q = 1$} & \multicolumn{2}{c|}{$q = 2$} & \multicolumn{2}{c|}{$q = 3$} & \multicolumn{2}{c}{$q = 4$}\\
	\multicolumn{1}{c|}{$\frac{H}{H_0}$} & rel. $L^2$ err. & eoc & rel. $L^2$ err. & eoc & rel. $L^2$ err. & eoc & rel. $L^2$ err. & eoc \\\midrule
	\begin{filecontents*}{./results/benchmarks/results/scalar_l2.csv}
capitalh,qOne,eocOne,qTwo,eocTwo,qThree,eocThree,qFour,eocFour
$2^{-1}$,\num{5.75E-03},--,\num{1.92E-03},--,\num{2.90E-04},--,\num{9.46E-05},--
$2^{-2}$,\num{1.08E-03},2.42,\num{1.26E-04},3.93,\num{1.62E-05},4.16,\num{1.47E-06},6.01
$2^{-3}$,\num{1.60E-04},2.75,\num{8.84E-06},3.83,\num{5.63E-07},4.85,\num{2.47E-08},5.90
$2^{-4}$,\num{2.16E-05},2.89,\num{5.96E-07},3.89,\num{1.88E-08},4.91,\num{4.19E-09},2.56
$2^{-5}$,\num{2.80E-06},2.95,\num{3.87E-08},3.94,\num{4.10E-09},2.19,\num{4.05E-09},0.05
\end{filecontents*}
\csvreader[late after line=\\,late after last line=\\\bottomrule,head to column names]{./results/benchmarks/results/scalar_l2.csv}{}
	{\capitalh & \qOne & \eocOne & \qTwo & \eocTwo & \qThree & \eocThree & \qFour & \eocFour}
\end{tabular}}
\end{table}

In~\cref{tab:capitalhconvergence_scalar_h1,tab:capitalhconvergence_scalar_l2}, the relative $H^1$ and $L^2$ errors for decreasing mesh sizes $H$ are shown. Both tables show the expected convergence rates. In the case of the $L^2$ norm for $q=3$ and $q=4$, the convergence rate deteriorates for small macro-mesh sizes $H$ because the discretization error is dominating the total error.

In order to show the dependence of the least-squares approach on the sampling level, we provide \Cref{fig:varying_hls}.
Here, we show two plots of the relative $L^2$ errors for fixed $q \in \{1,3\}$, $h=2^{-13}H_0$, and varying $H_\mathrm{LS}$.
From this figure, one can see that it is crucial to tune the sampling level fine enough in order to achieve optimal $\mathcal{O}\left(H^{q+2}\right)$ convergence in the $L^2$-norm.

\begin{remark}
The choice of sampling level $m_{\mathrm{LS}}$ or, equivalently, $H_{\mathrm{LS}}$ is very important, since the cost of the polynomial regression grows exponentially with $m_{\mathrm{LS}}$.
However, choosing a too large $H_{\mathrm{LS}}$ may violate the discrete $L^2$ projection property required in \cref{thm:APrioriBoundL2} in order to obtain an increased order of convergence.
Therefore, it is crucial to choose a suitable $H_{\mathrm{LS}}$ for an optimal ratio between the accuracy of the solution and the run time of the stencil function approximation.
In each of our experiments, setting $H_{\mathrm{LS}}$ two to four times larger than the fine mesh size $h$ yielded satisfactory results with respect to accuracy and run time.
\end{remark}

Furthermore, in \cref{fig:stencilplot_var_q_H} we want to illustrate the dependence of the polynomial degree $q$ and the macro-mesh size $H$ within the surrogate approach. For this purpose, we plot the central true and surrogate stencil functions over the subdomain $\{(x,y)^\top \in \Omega : x + y \geq 1\}$ for different pairings of $q$ and $H$. It can be observed, that there is no visible difference of both functions when either the pairing $H = H_0$ and $q = 8$, or the pairing $H = \nicefrac{H_0}{8}$ and $q = 2$ is chosen. Obviously, the quality of $\tilde{\sfA}$ can be improved by either increasing $q$ or decreasing $H$. For smooth coefficients $K$, increasing $q$ is the more efficient option, like in the $hp$-FEM context.

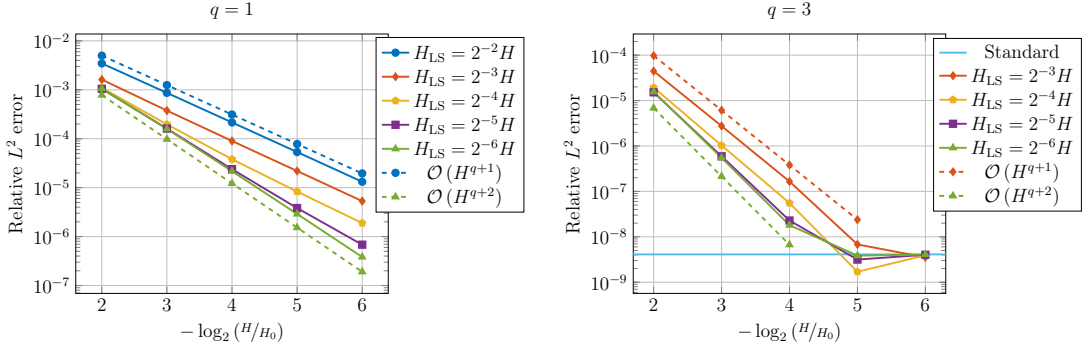
\begin{figure}\centering
\tikzset{font=\large}
\begin{minipage}{0.48\textwidth}
\begin{scaletikzpicturetowidth}{\textwidth}
\begin{tikzpicture}[scale=\tikzscale]
\begin{semilogyaxis}[
xlabel={$-\log_2{\left(\nicefrac{H}{H_0}\right)}$},
ylabel={Relative $L^2$ error},
xmajorgrids,
ymajorgrids,
title={$q=1$},
legend style={at={(0.96,0.65)},anchor=west},
xtick = {2,...,6}
]
\addplot[color1, mark=*, very thick] table [x index = {0}, y index={1}, col sep=comma] {./results/benchmarks/l2errors.csv};
\addlegendentry{$H_\mathrm{LS} = 2^{-2} H$};
\addplot[color2, mark=diamond*, very thick] table [x index = {0}, y index={2}, col sep=comma] {./results/benchmarks/l2errors.csv};
\addlegendentry{$H_\mathrm{LS} = 2^{-3} H$};
\addplot[color3, mark=pentagon*, very thick] table [x index = {0}, y index={3}, col sep=comma] {./results/benchmarks/l2errors.csv};
\addlegendentry{$H_\mathrm{LS} = 2^{-4} H$};
\addplot[color4, mark=square*, very thick] table [x index = {0}, y index={4}, col sep=comma] {./results/benchmarks/l2errors.csv};
\addlegendentry{$H_\mathrm{LS} = 2^{-5} H$};
\addplot[color5, mark=triangle*, very thick] table [x index = {0}, y index={5}, col sep=comma] {./results/benchmarks/l2errors.csv};
\addlegendentry{$H_\mathrm{LS} = 2^{-6} H$};
\addplot[color1, mark=*, very thick, dashed,mark options={solid}] table [x index = {0}, y index={8}, col sep=comma] {./results/benchmarks/l2errors.csv};
\addlegendentry{$\mathcal{O}\left(H^{q+1}\right)$};
\addplot[color5, mark=triangle*, very thick, dashed,mark options={solid}] table [x index = {0}, y index={9}, col sep=comma] {./results/benchmarks/l2errors.csv};
\addlegendentry{$\mathcal{O}\left(H^{q+2}\right)$};
\end{semilogyaxis}
\end{tikzpicture}
\end{scaletikzpicturetowidth}
\end{minipage}\hfill
\begin{minipage}{0.48\textwidth}
\begin{scaletikzpicturetowidth}{\textwidth}
\begin{tikzpicture}[scale=\tikzscale]
\begin{semilogyaxis}[
xlabel={$-\log_2{\left(\nicefrac{H}{H_0}\right)}$},
ylabel={Relative $L^2$ error},
xmajorgrids,
ymajorgrids,
title={$q=3$},
legend style={at={(0.96,0.65)},anchor=west},
xtick = {2,...,6},
xmin=1.7,xmax=6.3
]
\addplot[color6, mark=none, very thick, samples=2, restrict x to domain=1.7:6.3] plot coordinates {
	(1.7, 4.1e-9)
	(6.3, 4.1e-9)};
\addlegendentry{Standard};
\addplot[color2, mark=diamond*, very thick] table [x index = {0}, y index={2}, col sep=comma,restrict x to domain=1:6] {./results/benchmarks/l2errors_q3.csv};
\addlegendentry{$H_\mathrm{LS} = 2^{-3} H$};
\addplot[color3, mark=pentagon*, very thick] table [x index = {0}, y index={3}, col sep=comma,restrict x to domain=1:6] {./results/benchmarks/l2errors_q3.csv};
\addlegendentry{$H_\mathrm{LS} = 2^{-4} H$};
\addplot[color4, mark=square*, very thick] table [x index = {0}, y index={4}, col sep=comma,restrict x to domain=1:6] {./results/benchmarks/l2errors_q3.csv};
\addlegendentry{$H_\mathrm{LS} = 2^{-5} H$};
\addplot[color5, mark=triangle*, very thick] table [x index = {0}, y index={5}, col sep=comma,restrict x to domain=1:6] {./results/benchmarks/l2errors_q3.csv};
\addlegendentry{$H_\mathrm{LS} = 2^{-6} H$};
\addplot[color2, mark=diamond*, very thick, dashed,mark options={solid},restrict x to domain=1:5] table [x index = {0}, y index={8}, col sep=comma] {./results/benchmarks/l2errors_q3.csv};
\addlegendentry{$\mathcal{O}\left(H^{q+1}\right)$};
\addplot[color5, mark=triangle*, very thick, dashed,mark options={solid},restrict x to domain=1:4] table [x index = {0}, y index={9}, col sep=comma] {./results/benchmarks/l2errors_q3.csv};
\addlegendentry{$\mathcal{O}\left(H^{q+2}\right)$};
\end{semilogyaxis}
\end{tikzpicture}
\end{scaletikzpicturetowidth}
\end{minipage}
\caption{\label{fig:varying_hls}Relative $L^2$ errors for fixed $h=2^{-13}H_0$, varying $H_\mathrm{LS}$, $q=1$ (left), and $q=3$ (right) in the case of the variable coefficient Poisson equation on the unit-square with a scalar coefficient. For $q=3$ the relative $L^2$ error obtained from the standard approach is included, since the discretization error is dominating the surrogate error on the meshes with $H \leq 2^{-5}H_0$.
}
\end{figure}

\begin{figure}
\includegraphics[width=0.25\textwidth]{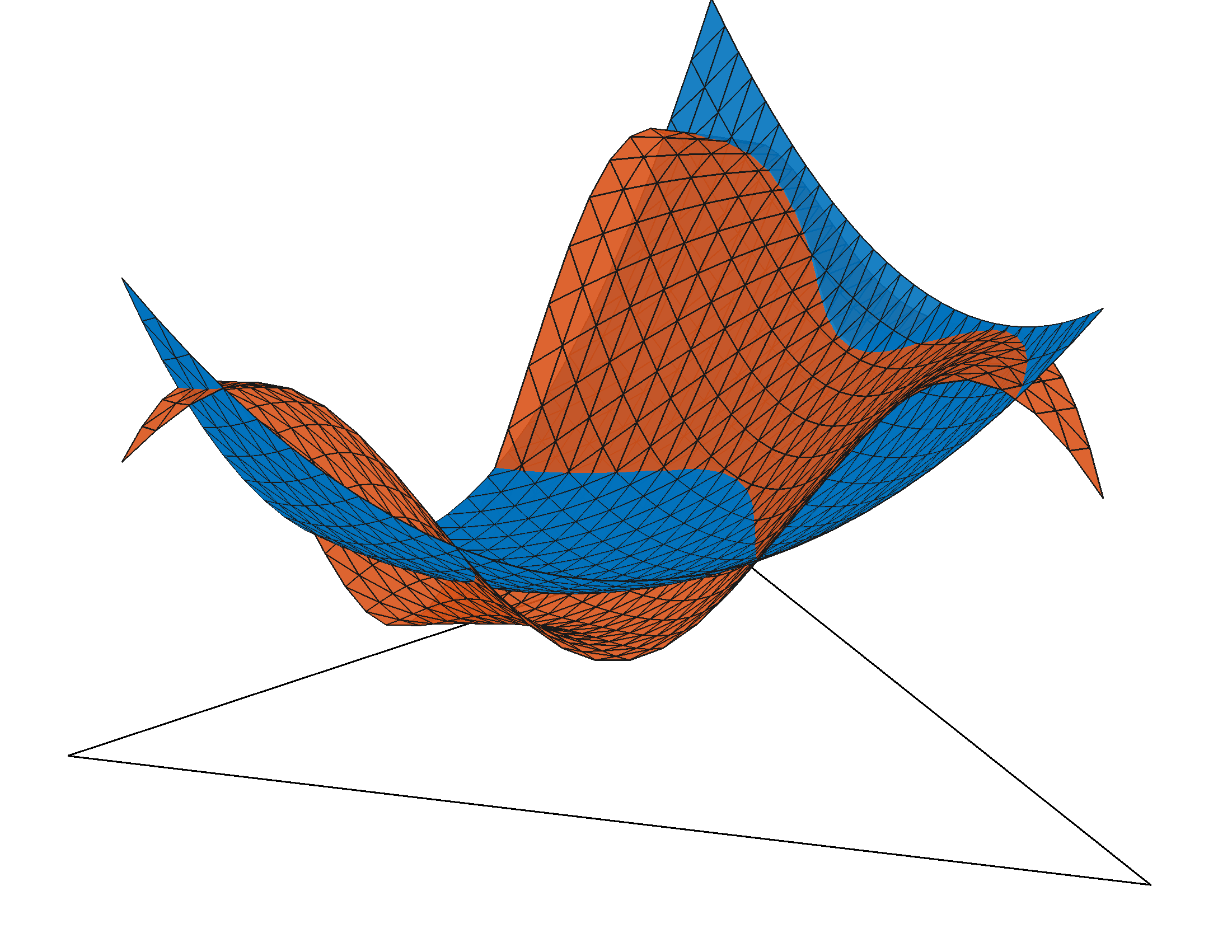}\includegraphics[width=0.25\textwidth]{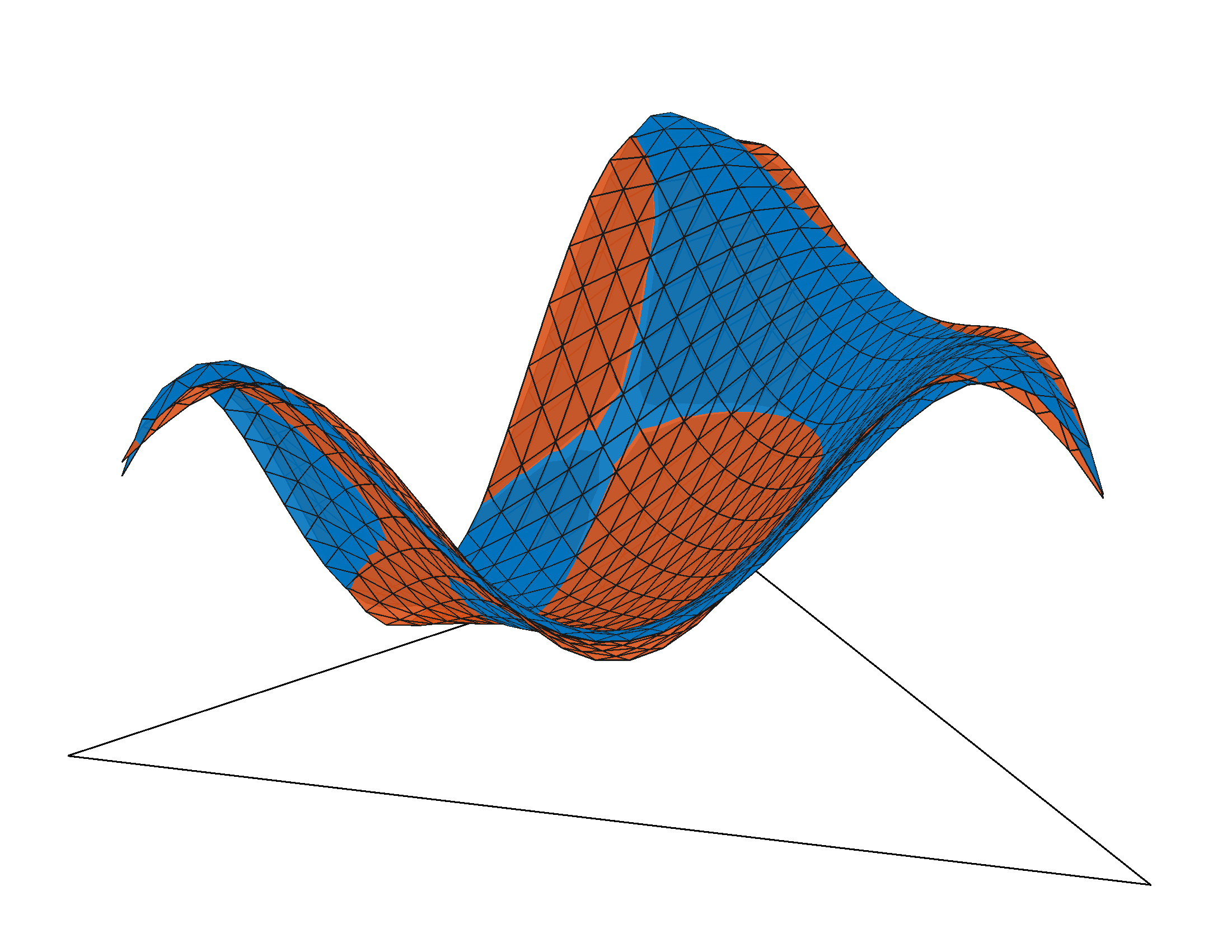}\includegraphics[width=0.25\textwidth]{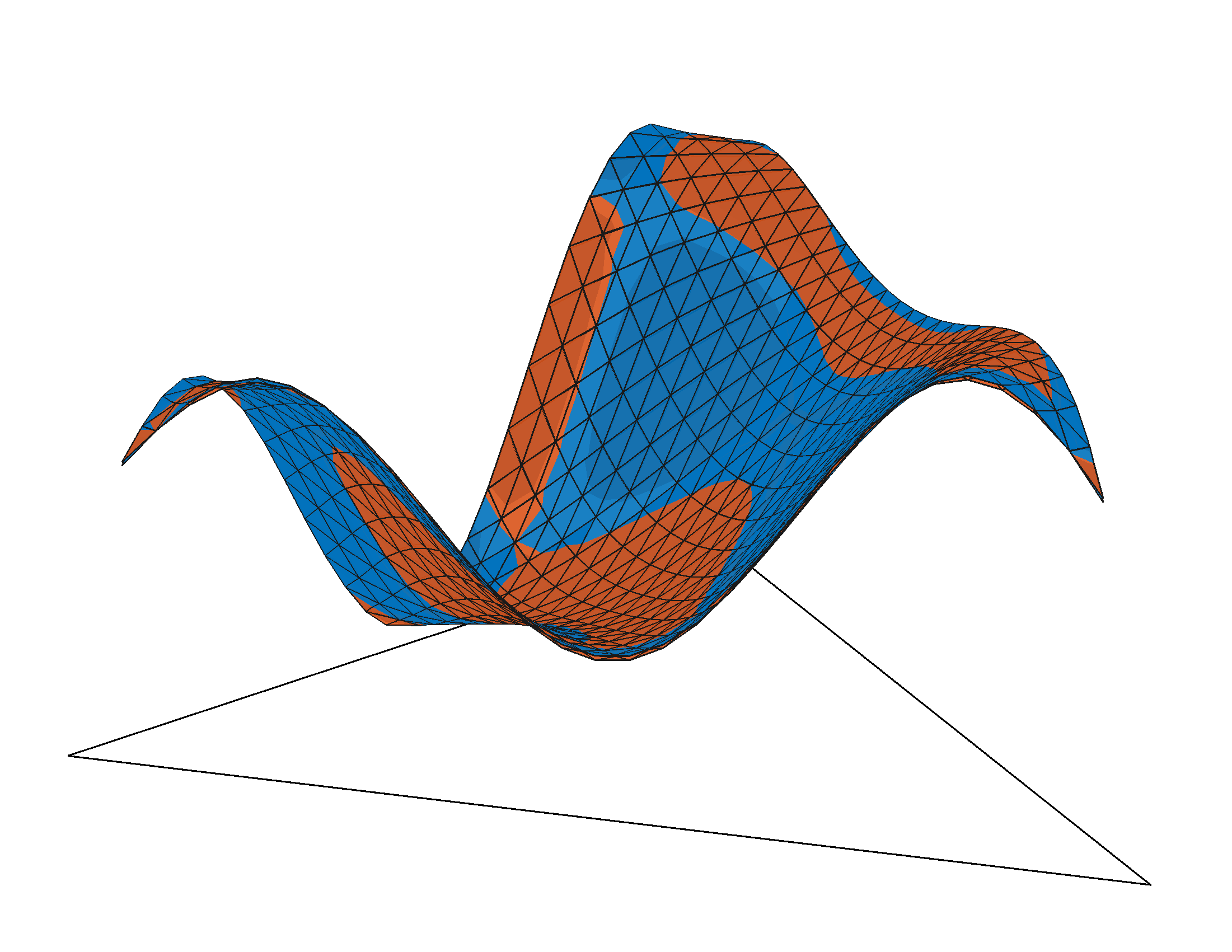}\includegraphics[width=0.25\textwidth]{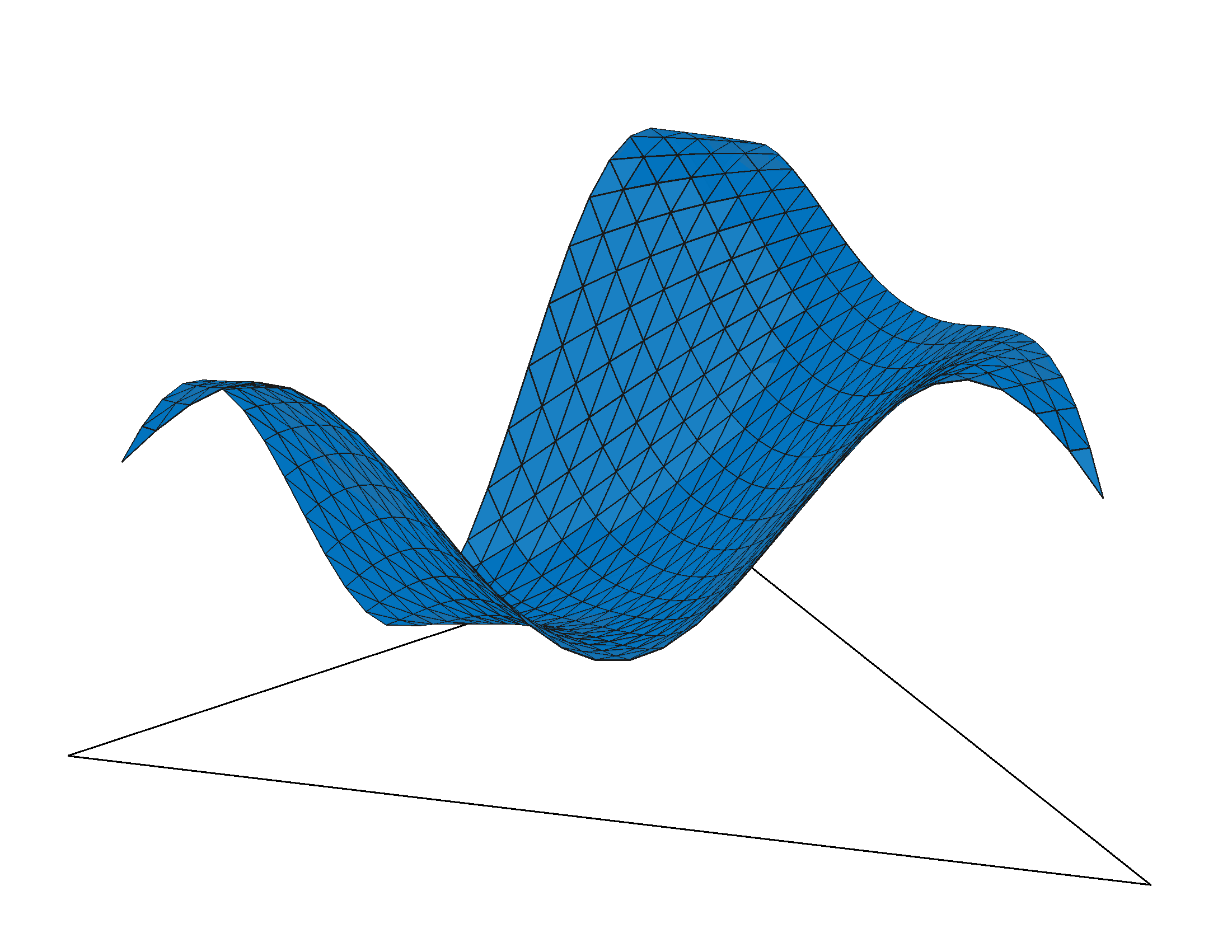}\\
\includegraphics[width=0.25\textwidth]{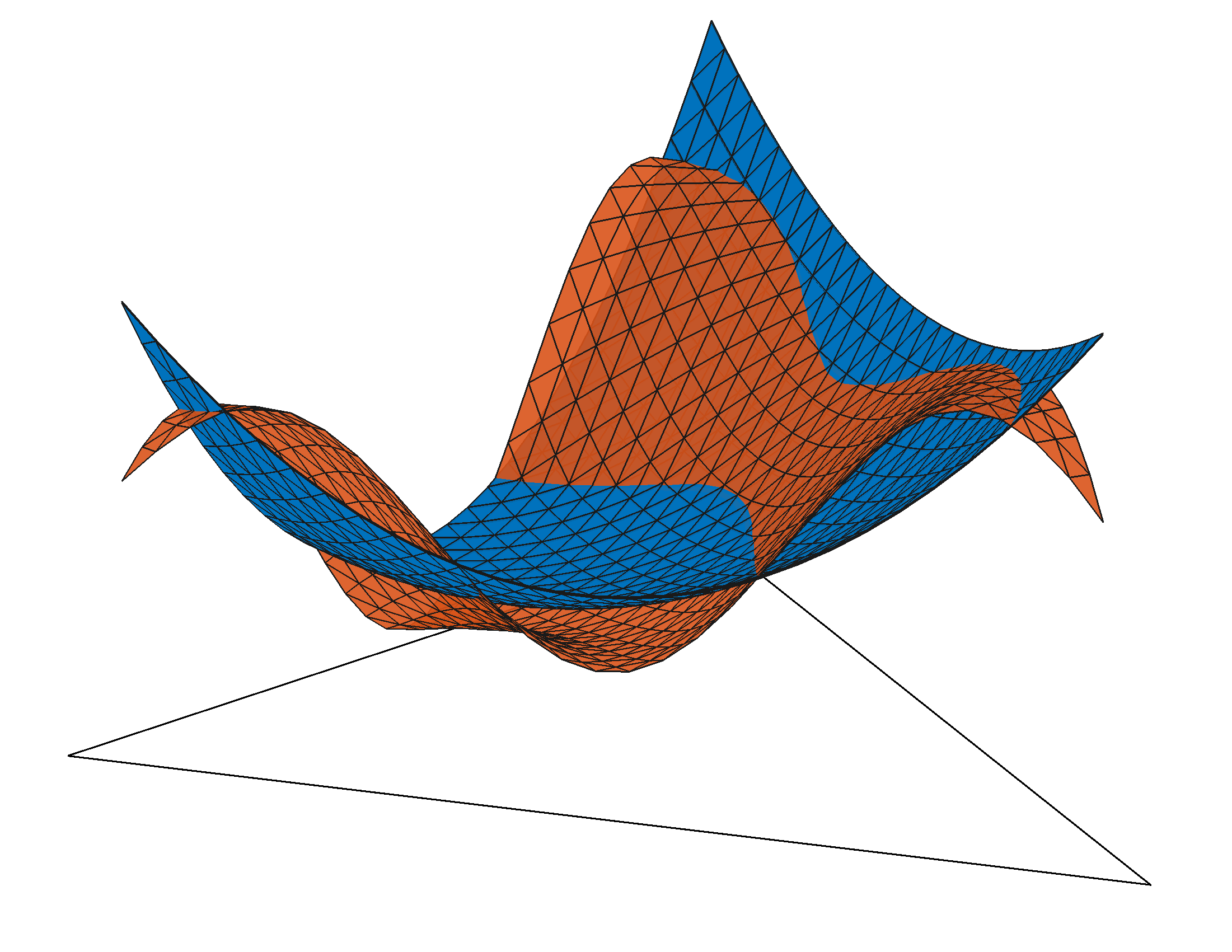}\includegraphics[width=0.25\textwidth]{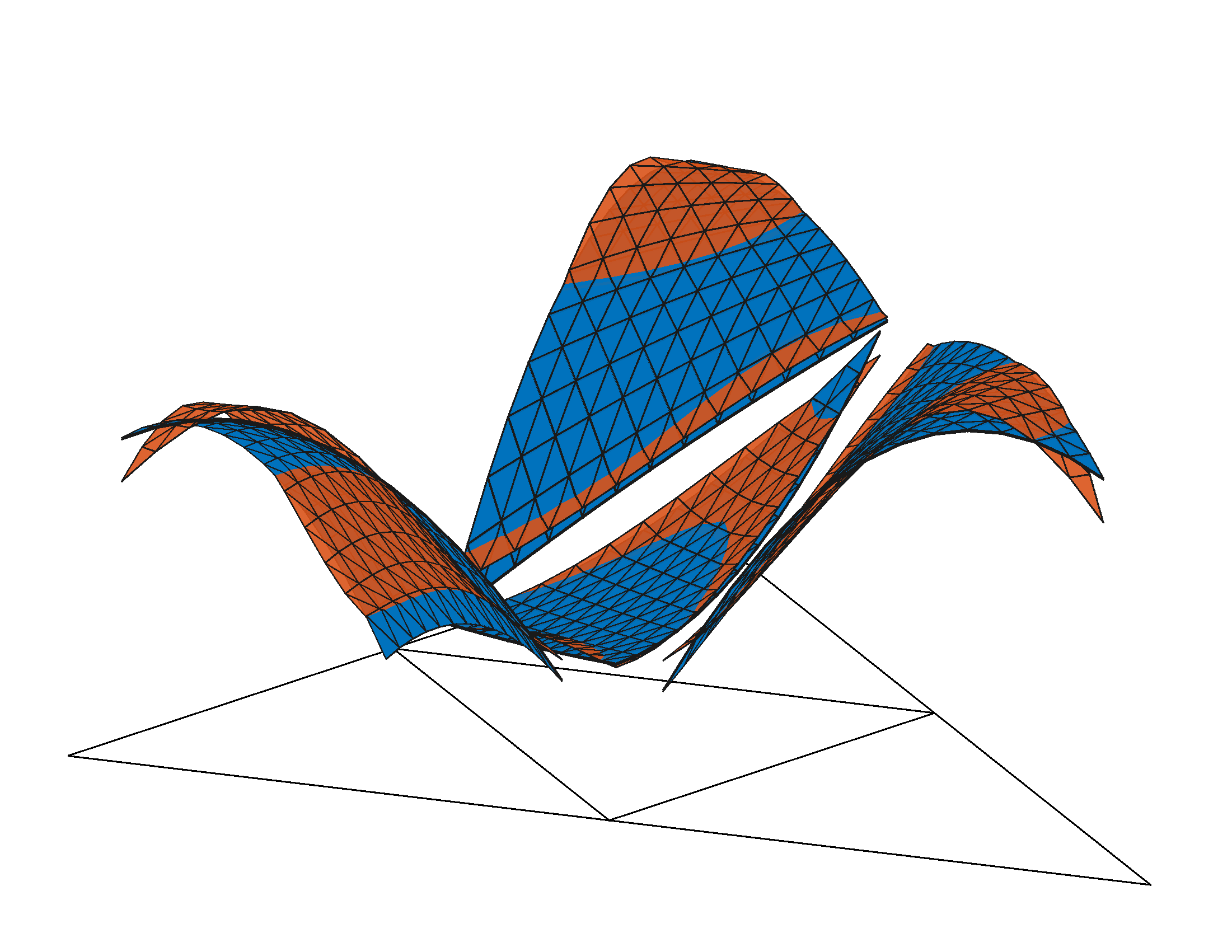}\includegraphics[width=0.25\textwidth]{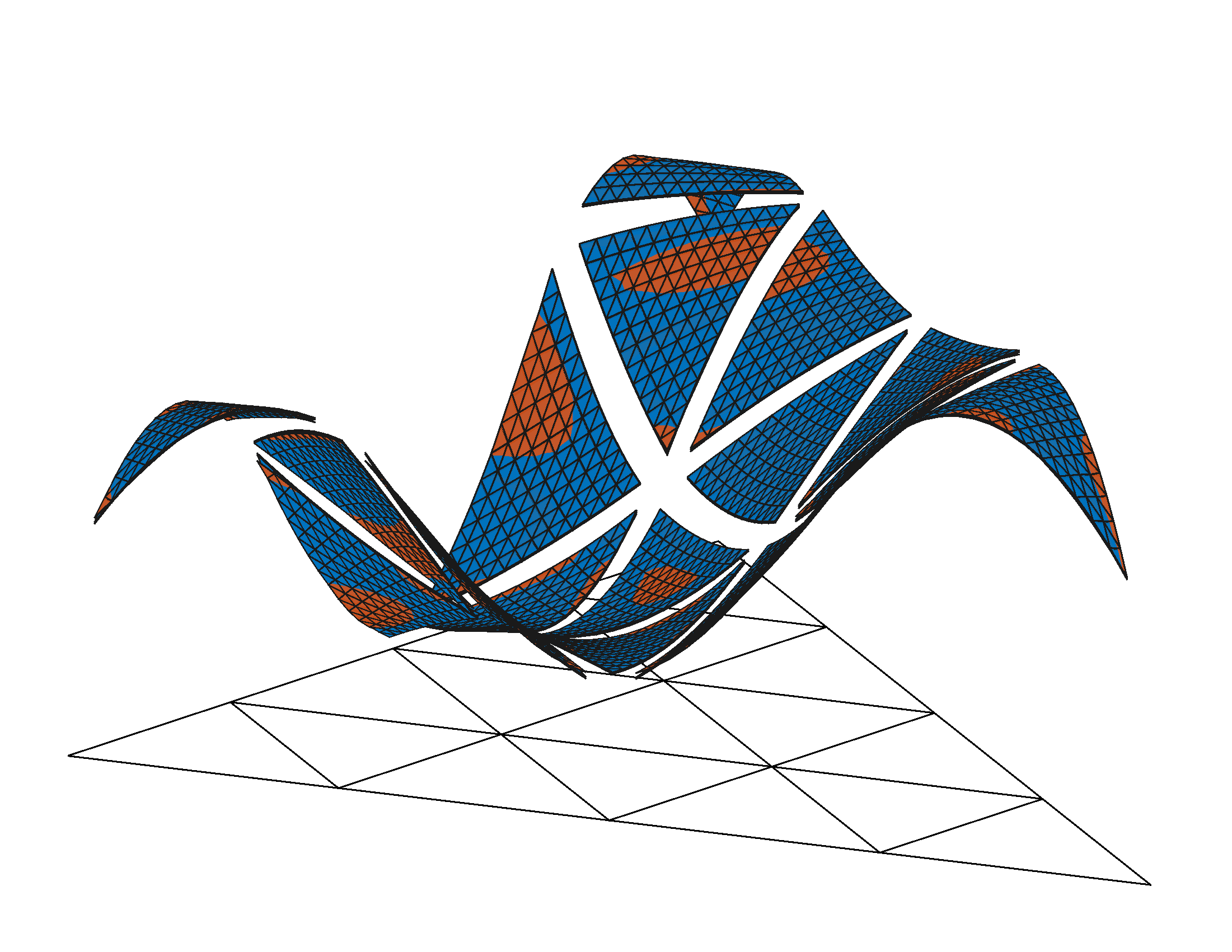}\includegraphics[width=0.25\textwidth]{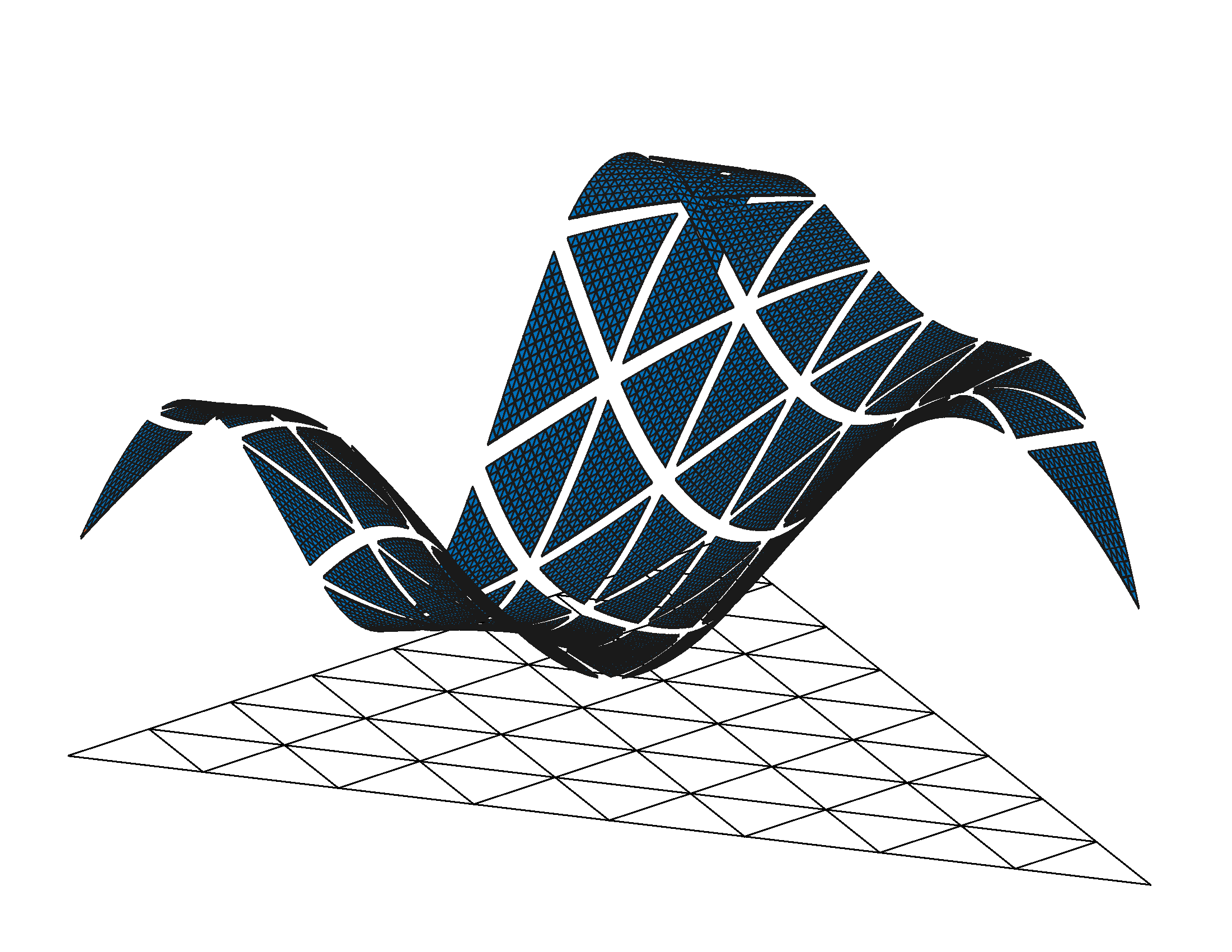}
\caption{\label{fig:stencilplot_var_q_H}Plots of true stencil functions in orange and surrogate stencil functions in blue for $\delta = 0$ over the subdomain $\{(x,y)^\top \in \Omega : x + y \geq 1\}$ in the case of the variable coefficient Poisson equation. Top row: Fixed $H = H_0$ and varying $q = 2$, $4$, $6$, and $8$ from left to right. Bottom row: Fixed $q=2$ and varying $H = H_0$, $\nicefrac{H_0}{2}$, $\nicefrac{H_0}{4}$, and $\nicefrac{H_0}{8}$ from left to right.}
\end{figure}

\subsubsection{Tensor coefficient on domain with curved boundaries}
\label{sec:tensor_coefficient_benchmark}
In the second benchmark problem, we study problem \cref{eqn:strongform} with the symmetric and positive definite tensor coefficient
\begin{align}
\label{eqn:benchmark_tensor_coefficient}
K(x,y) = \left[\begin{matrix}3 x^{2} + 2 y^{2} + 1 & - x^{2} - y^{2}\\- x^{2} - y^{2} & 4 x^{2} + 5 y^{2} + 1\end{matrix}\right]
\,.
\end{align}
Moreover, we consider the domain $\Omega$ with the curved boundary illustrated in \cref{fig:benchmark_curved_domain}. In the following scenarios, $a = 0.1$ is used as the amplitude of the boundary perturbation.
The mapping from the reference unit-square to the perturbed domain is defined by $\varphi$ in \cref{eq:benchmark_mapping_perturbed}. To map the coefficient onto the perturbed domain, we replace the coefficient $K$ in \cref{eqn:strongform} by a new coefficient, $K_0$, induced by the domain transformation, \textit{viz.},
\begin{align}
\label{eq:benchmark_mapping_perturbed}
K_0 = \frac{D\varphi^{-1} (K\circ \varphi) \sspace D\varphi^{-\top}}{|\det{\left(D\varphi^{-1}\right)}|}
\,,
\qquad\text{where}\qquad
\varphi(x,y) &= \left[\begin{matrix}
x\\
\left(2ay - a\right) \sin^2{\left( 2\pi x\right)} + y
\end{matrix}\right]
\,.
\end{align}
The manufactured solution $u$ is chosen to be $u(x,y) = \sin(\varphi_1(x,y))\sinh(\varphi_2(x,y))$. The restriction of $u$ to the boundary is chosen as Dirichlet datum $g$ and the right-hand-side $f$ is directly computed by inserting $u$ into the strong form of the equation~\cref{eqn:strongform}.

Here, we perform the same verification as in the previous subsection.
That is, fixing $h$ and varying $q$ and $H$ with the same meshes and solver settings. In~\cref{tab:capitalhconvergence_tensor_h1,tab:capitalhconvergence_tensor_l2}, the relative $H^1$ and $L^2$ errors for decreasing mesh sizes $H$ are shown. Both tables show the expected convergence rates. In the case of $q=4$, the convergence rate deteriorates for small macro-mesh sizes $H$, because the discretization error is dominating.

Additionally, we present results for fixed $H = 2^{-3}H_0$ and varying $h$ and $q$. \Cref{tab:littlehconvergence_tensor_pullback_l2} shows the relative $L^2$ errors and convergence rates of the standard approach and the surrogate approach with $q \in \{3,5,7\}$. Only for $q = 7$, the $L^2$ error coincides with the errors from the standard approach for all $h$. The relative time-to-solution (rtts) shown for the surrogate approaches is defined as the time-to-solution (tts) including the setup-phase of the surrogate approach divided by the time-to-solution of the standard approach. In the case with the smallest $h$, the surrogate approach took at most only $7\%$ of the time of the standard approach.
That is, a speed-up by more than a factor of $14$.
\begin{figure}	\centering
	\includegraphics[width=0.6\linewidth]{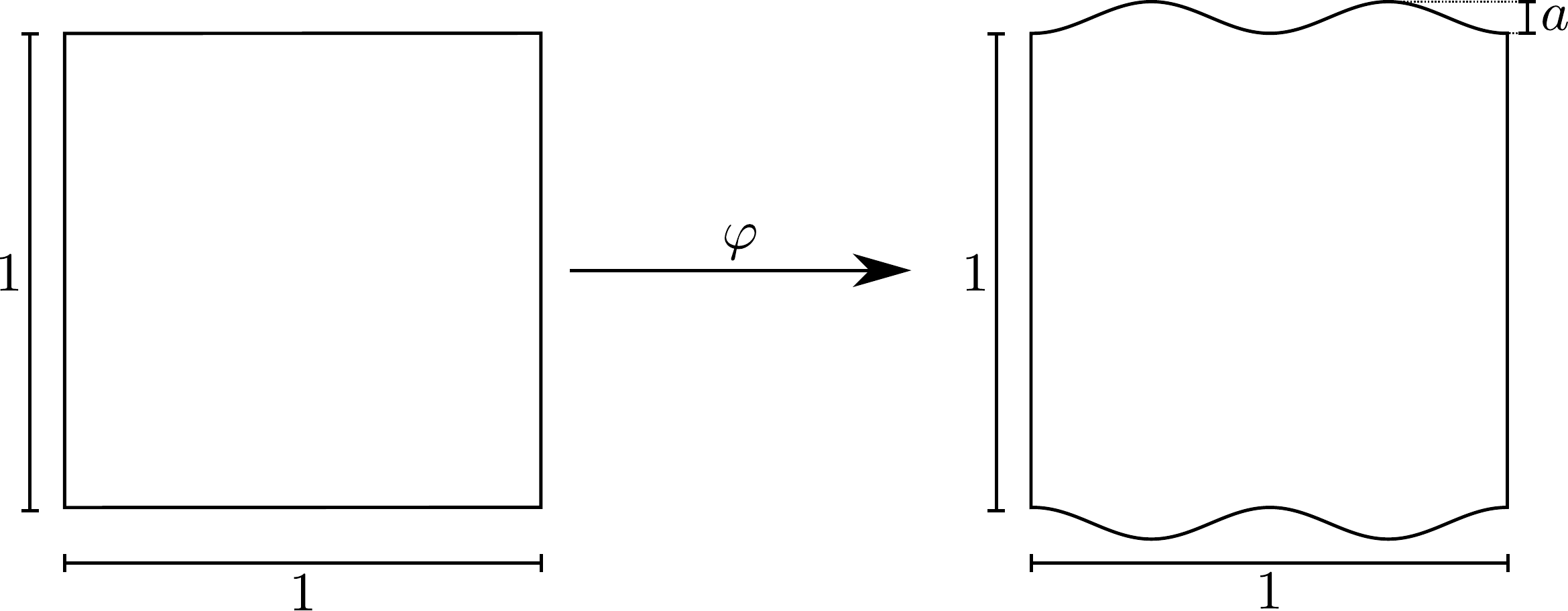}
	\caption{\label{fig:benchmark_curved_domain}Illustration of the mapping $\varphi$ from the unit-square to the perturbed unit-square. The top boundary is parametrized by $y = a\cdot\sin\left(2\pi x\right)^2 + 1$ and the bottom boundary by $y = -a\cdot\sin\left(2\pi x\right)^2$.}
\end{figure}

\begin{table}
	\centering
	\caption{\label{tab:capitalhconvergence_tensor_h1}Relative $H^1$ errors and experimental orders of convergence for fixed $h$ and varying $q$ and $H$ in the case of problem~\cref{eqn:strongform} with the tensorial coefficient~\cref{eqn:benchmark_tensor_coefficient} and curved boundary. The relative $H^1$ error with the classical FEM is \num{1.59e-08}.}
\resizebox{\linewidth}{!}{\begin{tabular}{r|c|c|c|c|c|c|c|c}	\toprule & \multicolumn{2}{c|}{$q = 1$} & \multicolumn{2}{c|}{$q = 2$} & \multicolumn{2}{c|}{$q = 3$} & \multicolumn{2}{c}{$q = 4$}\\
	\multicolumn{1}{c|}{$\frac{H}{H_0}$} & rel. $H^1$ err. & eoc & rel. $H^1$ err. & eoc & rel. $H^1$ err. & eoc & rel. $H^1$ err. & eoc \\\midrule
	\begin{filecontents*}{./results/benchmarks/results/tensor_h1.csv}
capitalh,qOne,eocOne,qTwo,eocTwo,qThree,eocThree,qFour,eocFour
$2^{-1}$,\num{1.04E-01},--,\num{7.09E-02},--,\num{2.96E-02},--,\num{9.31E-03},--
$2^{-2}$,\num{5.41E-02},0.95,\num{1.12E-02},2.67,\num{4.31E-03},2.78,\num{1.07E-03},3.12
$2^{-3}$,\num{1.37E-02},1.98,\num{2.29E-03},2.29,\num{3.14E-04},3.78,\num{4.66E-05},4.52
$2^{-4}$,\num{3.72E-03},1.89,\num{2.94E-04},2.96,\num{2.25E-05},3.80,\num{1.64E-06},4.83
$2^{-5}$,\num{9.56E-04},1.96,\num{3.77E-05},2.96,\num{1.46E-06},3.94,\num{5.55E-08},4.89
\end{filecontents*}
\csvreader[late after line=\\,late after last line=\\\bottomrule,head to column names]{./results/benchmarks/results/tensor_h1.csv}{}
	{\capitalh & \qOne & \eocOne & \qTwo & \eocTwo & \qThree & \eocThree & \qFour & \eocFour}
\end{tabular}}
\end{table}

\begin{table}
\centering
\caption{\label{tab:capitalhconvergence_tensor_l2}Relative $L^2$ errors and experimental orders of convergence for fixed $h$ and varying $q$ and $H$ in the case of problem~\cref{eqn:strongform} with the tensorial coefficient~\cref{eqn:benchmark_tensor_coefficient} and curved boundary. The relative $L^2$ error with the classical FEM is \num{4.56e-09}.}
\resizebox{\linewidth}{!}{\begin{tabular}{r|c|c|c|c|c|c|c|c}	\toprule & \multicolumn{2}{c|}{$q = 1$} & \multicolumn{2}{c|}{$q = 2$} & \multicolumn{2}{c|}{$q = 3$} & \multicolumn{2}{c}{$q = 4$}\\
	\multicolumn{1}{c|}{$\frac{H}{H_0}$} & rel. $L^2$ err. & eoc & rel. $L^2$ err. & eoc & rel. $L^2$ err. & eoc & rel. $L^2$ err. & eoc \\\midrule
	\begin{filecontents*}{./results/benchmarks/results/tensor_l2.csv}
capitalh,qOne,eocOne,qTwo,eocTwo,qThree,eocThree,qFour,eocFour
$2^{-1}$,\num{1.44E-02},--,\num{6.43E-03},--,\num{3.08E-03},--,\num{5.41E-04},--
$2^{-2}$,\num{3.74E-03},1.95,\num{6.60E-04},3.28,\num{1.39E-04},4.47,\num{3.54E-05},3.94
$2^{-3}$,\num{5.37E-04},2.80,\num{5.75E-05},3.52,\num{6.46E-06},4.43,\num{6.13E-07},5.85
$2^{-4}$,\num{7.35E-05},2.87,\num{3.77E-06},3.93,\num{2.09E-07},4.95,\num{1.28E-08},5.58
$2^{-5}$,\num{9.52E-06},2.95,\num{2.40E-07},3.98,\num{8.05E-09},4.70,\num{4.49E-09},1.51
\end{filecontents*}
\csvreader[late after line=\\,late after last line=\\\bottomrule,head to column names]{./results/benchmarks/results/tensor_l2.csv}{}
	{\capitalh & \qOne & \eocOne & \qTwo & \eocTwo & \qThree & \eocThree & \qFour & \eocFour}
\end{tabular}}
\end{table}

\begin{table}
\centering
\caption{\label{tab:littlehconvergence_tensor_pullback_l2}Relative $L^2$ errors, experimental orders of convergence, and relative time-to-solutions for fixed $H$ and varying $q$ and $h$ in the case of problem~\cref{eqn:strongform} with the tensorial coefficient~\cref{eqn:benchmark_tensor_coefficient} and curved boundary.}
\resizebox{\linewidth}{!}{\begin{tabular}{r|r|r|c|c|c|c|c|c|c|c|c|c|c}	\toprule & & & \multicolumn{2}{c|}{standard} & \multicolumn{3}{c|}{$q = 3$} & \multicolumn{3}{c|}{$q = 5$} & \multicolumn{3}{c}{$q = 7$}\\
	\multicolumn{1}{c|}{$\frac{h}{H_0}$} & \multicolumn{1}{c|}{$\frac{H_{\mathrm{LS}}}{h}$} & \multicolumn{1}{c|}{DoFs} & rel. $L^2$ err. & eoc & rel. $L^2$ err. & eoc & rtts & rel. $L^2$ err. & eoc & rtts & rel. $L^2$ err. & eoc & rtts \\\midrule
	\begin{filecontents*}{./results/benchmarks/results/tensor_comparison.csv}
h, HLS, dofs, ref, eoc, tts, qzero, eoczero, ttszero, relerrzero, relttszero, qone, eocone, ttsone, relerrone, relttsone, qtwo, eoctwo, ttstwo, relerrtwo, relttstwo, qthree, eocthree, ttsthree, relerrthree, relttsthree, qfour, eocfour, ttsfour, relerrfour, relttsfour, qfive, eocfive, ttsfive, relerrfive, relttsfive, qsix, eocsix, ttssix, relerrsix, relttssix, qseven, eocseven, ttsseven, relerrseven, relttsseven, qeight, eoceight, ttseight, relerreight, relttseight
$2^{-6}$, $2^{0}$, \num{4.2e+03}, \num{7.14e-05}, -, \num{0.54}, \num{3.31e-03}, -, 0.46, 46.39, 0.85, \num{3.71e-04}, -, 0.49, 5.19, 0.91, \num{7.84e-05}, -, 0.50, 1.10, 0.92, \num{7.15e-05}, -, 0.50, 1.00, 0.93, \num{7.14e-05}, -, 0.53, 1.00, 0.98, \num{7.14e-05}, -, 0.51, 1.00, 0.94, \num{7.14e-05}, -, 0.53, 1.00, 0.98, \num{7.14e-05}, -, 0.56, 1.00, 1.03, \num{7.14e-05}, -, 0.56, 1.00, 1.05
$2^{-7}$, $2^{0}$, \num{1.7e+04}, \num{1.81e-05}, 1.98, \num{0.70}, \num{3.82e-03}, -0.21, 0.54, 211.29, 0.77, \num{4.51e-04}, -0.28, 0.57, 24.92, 0.82, \num{4.87e-05}, 0.69, 0.58, 2.70, 0.83, \num{1.87e-05}, 1.94, 0.61, 1.03, 0.87, \num{1.81e-05}, 1.98, 0.64, 1.00, 0.92, \num{1.81e-05}, 1.98, 0.61, 1.00, 0.87, \num{1.81e-05}, 1.98, 0.65, 1.00, 0.93, \num{1.81e-05}, 1.98, 0.67, 1.00, 0.96, \num{1.81e-05}, 1.98, 0.68, 1.00, 0.97
$2^{-8}$, $2^{1}$, \num{6.6e+04}, \num{4.55e-06}, 1.99, \num{1.01}, \num{3.99e-03}, -0.06, 0.62, 878.42, 0.61, \num{4.60e-04}, -0.03, 0.69, 101.22, 0.68, \num{4.61e-05}, 0.08, 0.63, 10.14, 0.63, \num{6.76e-06}, 1.46, 0.64, 1.49, 0.64, \num{4.48e-06}, 2.01, 0.69, 1.00, 0.68, \num{4.55e-06}, 1.99, 0.68, 1.00, 0.67, \num{4.55e-06}, 1.99, 0.73, 1.00, 0.73, \num{4.55e-06}, 1.99, 0.75, 1.00, 0.75, \num{4.55e-06}, 1.99, 0.75, 1.00, 0.75
$2^{-9}$, $2^{1}$, \num{2.6e+05}, \num{1.14e-06}, 2.00, \num{2.30}, \num{4.18e-03}, -0.06, 0.72, 3663.69, 0.31, \num{4.99e-04}, -0.12, 0.76, 437.85, 0.33, \num{5.13e-05}, -0.15, 0.78, 45.03, 0.34, \num{5.70e-06}, 0.25, 0.81, 5.00, 0.35, \num{1.21e-06}, 1.88, 0.85, 1.06, 0.37, \num{1.14e-06}, 1.99, 0.85, 1.00, 0.37, \num{1.14e-06}, 2.00, 0.91, 1.00, 0.40, \num{1.14e-06}, 2.00, 0.95, 1.00, 0.41, \num{1.14e-06}, 2.00, 0.97, 1.00, 0.42
$2^{-10}$, $2^{1}$, \num{1.1e+06}, \num{2.85e-07}, 2.00, \num{7.26}, \num{4.27e-03}, -0.03, 0.95, 14953.52, 0.13, \num{5.20e-04}, -0.06, 1.03, 1822.98, 0.14, \num{5.49e-05}, -0.10, 1.09, 192.17, 0.15, \num{6.11e-06}, -0.10, 1.11, 21.40, 0.15, \num{6.29e-07}, 0.95, 1.18, 2.20, 0.16, \num{3.00e-07}, 1.93, 1.19, 1.05, 0.16, \num{2.86e-07}, 2.00, 1.29, 1.00, 0.18, \num{2.85e-07}, 2.00, 1.39, 1.00, 0.19, \num{2.85e-07}, 2.00, 1.42, 1.00, 0.20
$2^{-11}$, $2^{1}$, \num{4.2e+06}, \num{7.14e-08}, 2.00, \num{25.58}, \num{4.32e-03}, -0.02, 1.63, 60413.67, 0.06, \num{5.31e-04}, -0.03, 1.82, 7436.24, 0.07, \num{5.66e-05}, -0.05, 1.82, 793.01, 0.07, \num{6.35e-06}, -0.06, 1.88, 88.89, 0.07, \num{6.01e-07}, 0.06, 2.03, 8.42, 0.08, \num{1.16e-07}, 1.37, 2.10, 1.62, 0.08, \num{7.21e-08}, 1.99, 2.28, 1.01, 0.09, \num{7.14e-08}, 2.00, 2.47, 1.00, 0.10, \num{7.14e-08}, 2.00, 2.68, 1.00, 0.10
$2^{-12}$, $2^{1}$, \num{1.7e+07}, \num{1.79e-08}, 2.00, \num{100.44}, \num{4.34e-03}, -0.01, 3.78, 242796.54, 0.04, \num{5.37e-04}, -0.01, 4.49, 30027.85, 0.04, \num{5.75e-05}, -0.02, 4.52, 3219.52, 0.05, \num{6.47e-06}, -0.03, 4.74, 361.80, 0.05, \num{6.13e-07}, -0.03, 5.17, 34.28, 0.05, \num{9.49e-08}, 0.29, 5.56, 5.31, 0.06, \num{1.95e-08}, 1.89, 6.20, 1.09, 0.06, \num{1.79e-08}, 2.00, 7.00, 1.00, 0.07, \num{1.79e-08}, 2.00, 8.44, 1.00, 0.08
$2^{-13}$, $2^{1}$, \num{6.7e+07}, \num{4.50e-09}, 1.99, \num{374.65}, \num{4.35e-03}, -0.00, 11.55, 967795.91, 0.03, \num{5.39e-04}, -0.01, 13.83, 119975.90, 0.04, \num{5.80e-05}, -0.01, 13.62, 12897.06, 0.04, \num{6.52e-06}, -0.01, 14.76, 1450.77, 0.04, \num{6.20e-07}, -0.02, 16.24, 137.84, 0.04, \num{9.43e-08}, 0.01, 18.45, 20.99, 0.05, \num{8.78e-09}, 1.15, 21.68, 1.95, 0.06, \num{4.54e-09}, 1.98, 25.64, 1.01, 0.07, \num{4.62e-09}, 1.95, 30.90, 1.03, 0.08

\end{filecontents*}
\csvreader[late after line=\\,late after last line=\\\bottomrule,head to column names]{./results/benchmarks/results/tensor_comparison.csv}{}
	{\h & \HLS & \dofs & \ref & \eoc & \qthree & \eocthree & \relttsthree & \qfive & \eocfive & \relttsfive & \qseven & \eocseven & \relttsseven}
\end{tabular}}
\end{table}

\subsection{Linearized elasticity example}
In this subsection, we compare a standard method with a surrogate method applied to the linearized elasticity problem presented in \Cref{sub:linearized_elasticity}.
In our surrogate method, we employ the zero row sum property described in \Cref{sec:the_zero_row_sum_property}.
We choose an annular domain composed of two distinct and concentric materials under uniform pressure loading.
The problem is inspired by a similar 3D experiment documented in \cite{fuentes2017coupled}.
Let $\mathcal{B}_r \subset \mathbb{R}^{2}$ be the two-dimensional open ball of radius $r$ with the midpoint at the origin. The computational domain is then defined as $\Omega = \mathcal{B}_{R_{\mathrm{out}}} \setminus \mathcal{B}_{R_{\mathrm{in}}}$. We split this domain into two disjoint sets $\Omega_I = \mathcal{B}_{R_{\mathrm{mid}}} \setminus \mathcal{B}_{R_{\mathrm{in}}}$ and $\Omega_O = \mathcal{B}_{R_{\mathrm{out}}} \setminus \mathcal{B}_{R_{\mathrm{mid}}}$ corresponding to each material. In our experiments, we fix $R_{\mathrm{in}} = \SI{1}{\centi\meter}$, $R_{\mathrm{mid}} = \SI{1.75}{\centi\meter}$, and $R_{\mathrm{out}} = \SI{2}{\centi\meter}$. Refer to the leftmost diagram in \cref{fig:linear_elasticity_domain} for an illustration of the setup. Here, the macro-elements adjacent to the boundary and the material interface are mapped to the physical geometry by using the transformation described in \cite{zlamal1973curved}.

\begin{figure}\centering
\begin{minipage}{0.45\linewidth}
\centering
\includegraphics[width=\linewidth]{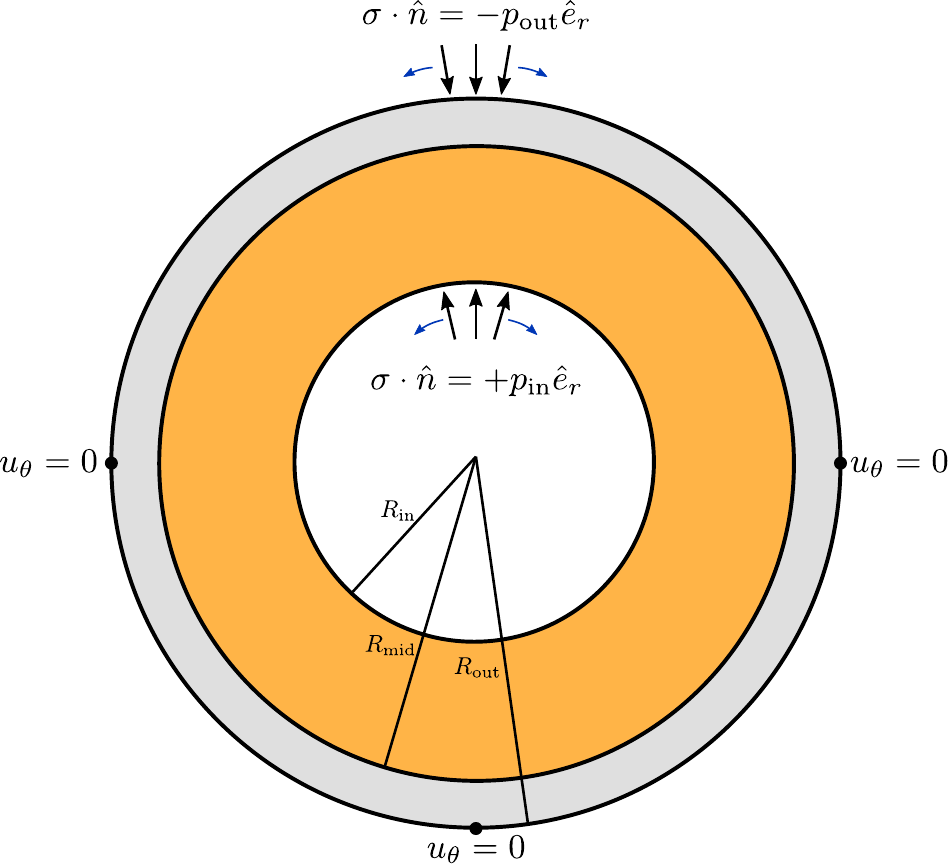}
\end{minipage}\hspace*{0.5em}
\begin{minipage}{0.45\linewidth}
\centering
\includegraphics[width=\linewidth]{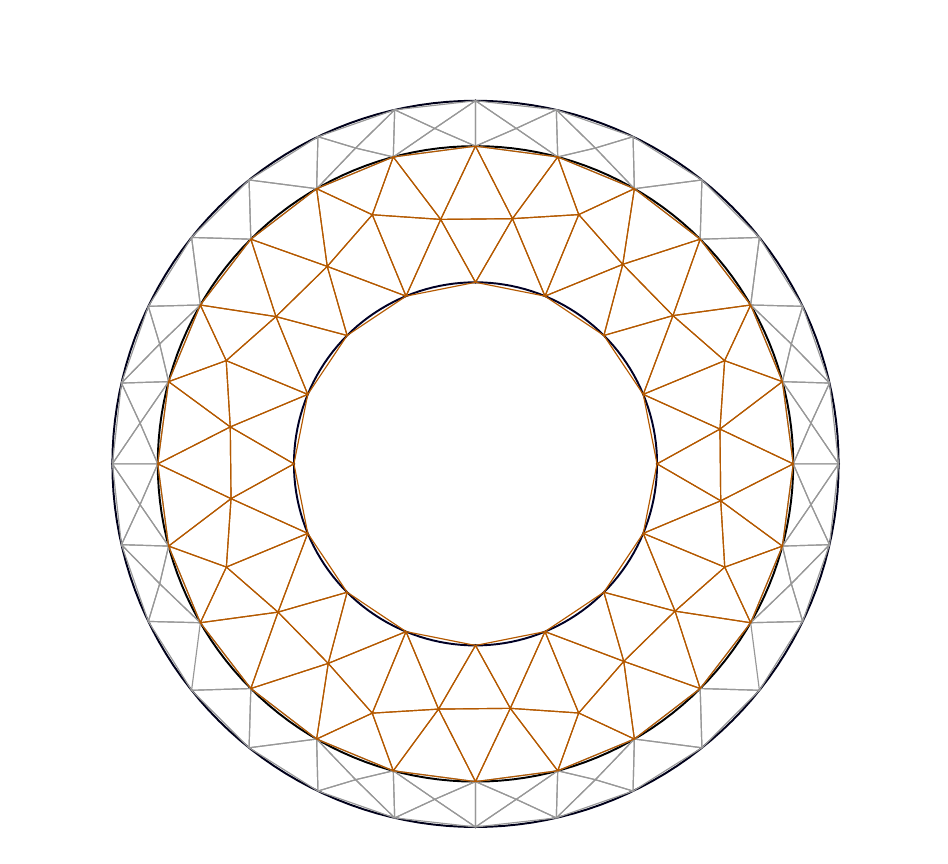}
\end{minipage}
\caption{\label{fig:linear_elasticity_domain}Linear elasticity problem setup (left) and initial macro-mesh $\mesh_H$ (right).}
\end{figure}

Let $r(x) = |x|$.
The strong form of the problem is
\begin{equation}
\begin{alignedat}{2}
-\mathrm{Div}\left(\sigma\right) &= \vec{f} &&\quad \text{in } \Omega,\\
u_{\theta} &= 0 &&\quad \text{on } \{(-R_{\mathrm{out}},0)^\top, (0,-R_{\mathrm{out}})^\top, (R_{\mathrm{out}},0)^\top\},\\
\sigma \cdot \hat{n} &= p_{\mathrm{in}} \hat{e}_r &&\quad \text{on }\{x \in \partial\Omega : r = R_{\mathrm{in}} \},\\
\sigma \cdot \hat{n} &= -p_{\mathrm{out}} \hat{e}_r &&\quad \text{on } \{x \in \partial\Omega : r = R_{\mathrm{out}} \}.
\end{alignedat}
\label{eqn:linearelasticity}
\end{equation}
Here, the stress tensor $\sigma$ is given by Hooke's law for isotropic materials as defined in \Cref{sub:linearized_elasticity}. The unit vector in radial direction is denoted $\hat{e}_r$ and the outward pointing unit normal vector is denoted $\hat{n}$. We neglect body forces and therefore set $\vec{f} = \vec{0}$. The displacement is described in polar coordinates where $u_r$ is the radial displacement and $u_\theta$ is the tangential displacement. In order to make the system uniquely solvable, we enforce the tangential displacement $u_{\theta}$ to be zero at three points; see \cref{fig:linear_elasticity_domain}.
The materials are chosen to be cork in the inner domain $\Omega_I$ with an A36 steel layer in the outer domain $\Omega_O$.
Poisson's ratio and Young's modulus for this scenario are $E_I = \SI{0.02}{\giga\pascal}$, $E_O = \SI{200.0}{\giga\pascal}$, $\nu_I = 0$, and $\nu_O = 0.26$. The Lam\'{e} parameters are obtained by the expressions $\mu = \frac{E}{2(1+\nu)}$ and $\lambda = \frac{E\nu}{(1-2\nu)\cdot(1+\nu)}$.
Note that while these expressions induce piecewise-constant Lam\'e parameters, $\lambda = \lambda(r)$ and $\mu = \mu(r)$, once the smooth domain is mapped to the computational domain depicted on the right of \cref{fig:linear_elasticity_domain}, these parameters will not be piecewise-constant anymore due to the transformation.
The pressure on the outer boundary is set to $p_{\mathrm{out}} = \SI{0}{\mega\pascal}$ and on the inner boundary $p_{\mathrm{in}} = \SI{1}{\mega\pascal}$ is prescribed.

In this particular scenario, there is an analytic solution available for the radial displacement $u_r$ which has the form
\begin{equation}
\label{eqn:linear_elasticity_analytic_solution}
\begin{aligned}
u_r(r) = \begin{cases} 
A \cdot r + B \cdot r^{-1} &\quad \text{if }r \in [R_{\mathrm{in}}, R_{\mathrm{mid}}]\,, \\
C \cdot r + D \cdot r^{-1} &\quad \text{if }r \in (R_{\mathrm{mid}}, R_{\mathrm{out}}]\,.
\end{cases}
\end{aligned}
\end{equation}
The tangential displacement $u_\theta$ is zero everywhere due to the symmetry of the problem.
In~\cref{eqn:linear_elasticity_analytic_solution}, the coefficients $A$, $B$, $C$, and $D$ are uniquely determined by the following system of linear equations:
\begin{equation}
\label{eqn:linear_elasticity_analytic_system}
\begin{aligned}
\begin{bmatrix}
E_{I}R_{\mathrm{in}}^{2} & E_{I} \left(2\nu_{I} - 1\right) & 0 & 0\\
0 & 0 & E_{O}R_{\mathrm{out}}^{2} & E_{O} \left(2\nu_{O}-1\right)\\
R_{\mathrm{mid}}^2 & 1 & -R_{\mathrm{mid}}^2 & -1\\
- E_{I} R_{\mathrm{mid}}^{2} d_{O} & -d_{O}  E_{I} \left(2\nu_{I} - 1\right) & E_{O} R_{\mathrm{mid}}^{2} d_{I} & d_{I} E_{O} \left(2\nu_{O} - 1\right)
\end{bmatrix} \begin{bmatrix}
A\\
B\\
C\\
D
\end{bmatrix} &= \begin{bmatrix}
p_{in} R_{\mathrm{in}}^{2} d_{I}\\
p_{out} R_{\mathrm{out}}^{2} d_{O}\\
0\\
0
\end{bmatrix},
\end{aligned}
\end{equation}
where $d_{I} \coloneqq 2 \nu_{I}^{2} + \nu_{I} - 1$ and $d_{O} \coloneqq 2 \nu_{O}^{2} + \nu_{O} - 1$.
This system is derived after deducing the continuity of the displacement and surface traction at the material interface, then by incorporating the prescribed external forces at the boundaries.
\begin{figure}\centering
\begin{minipage}{0.45\linewidth}
\centering
\includegraphics[width=\linewidth]{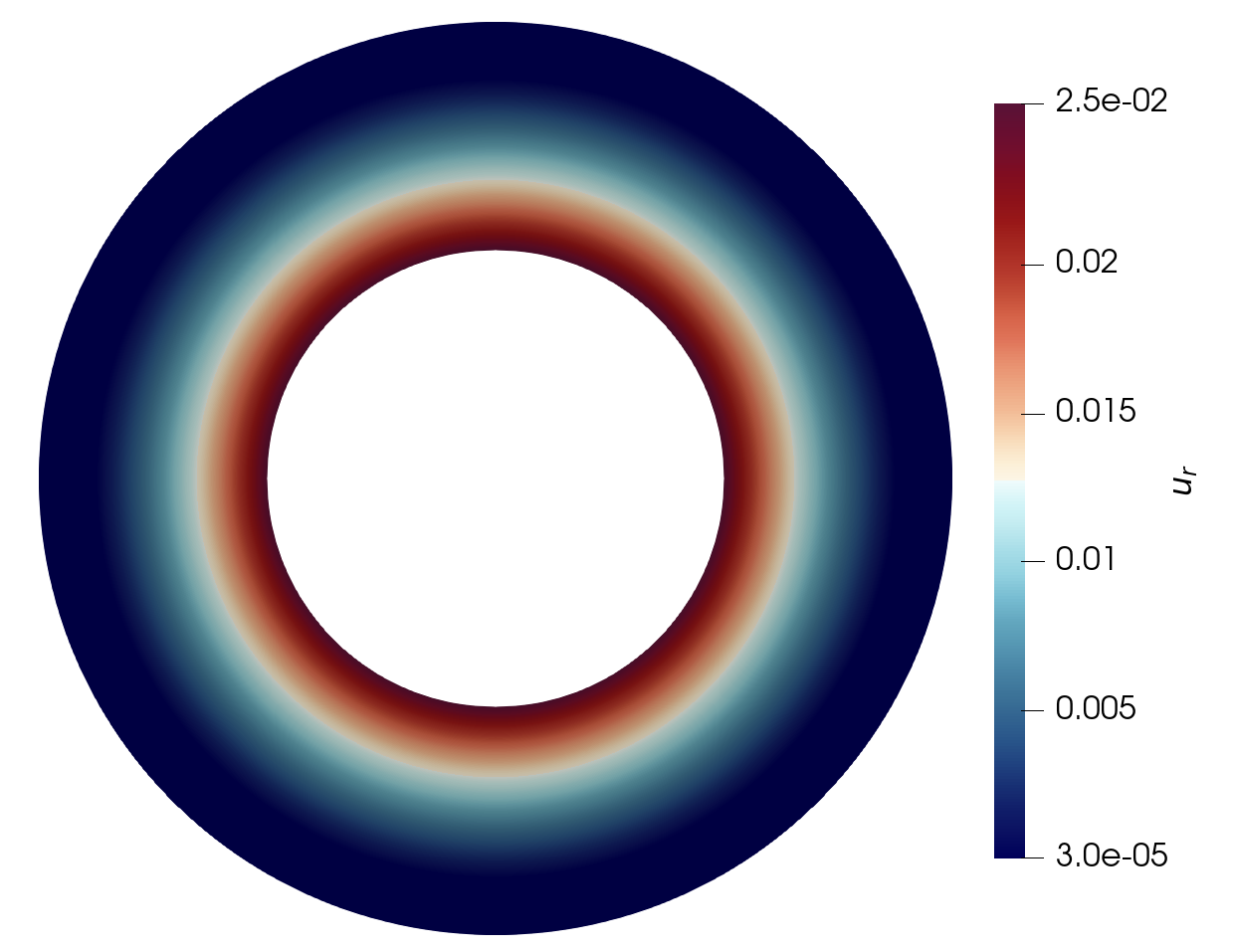}
\end{minipage}\hspace*{0.5em}
\begin{minipage}{0.45\linewidth}
\centering	
\includegraphics[width=\linewidth]{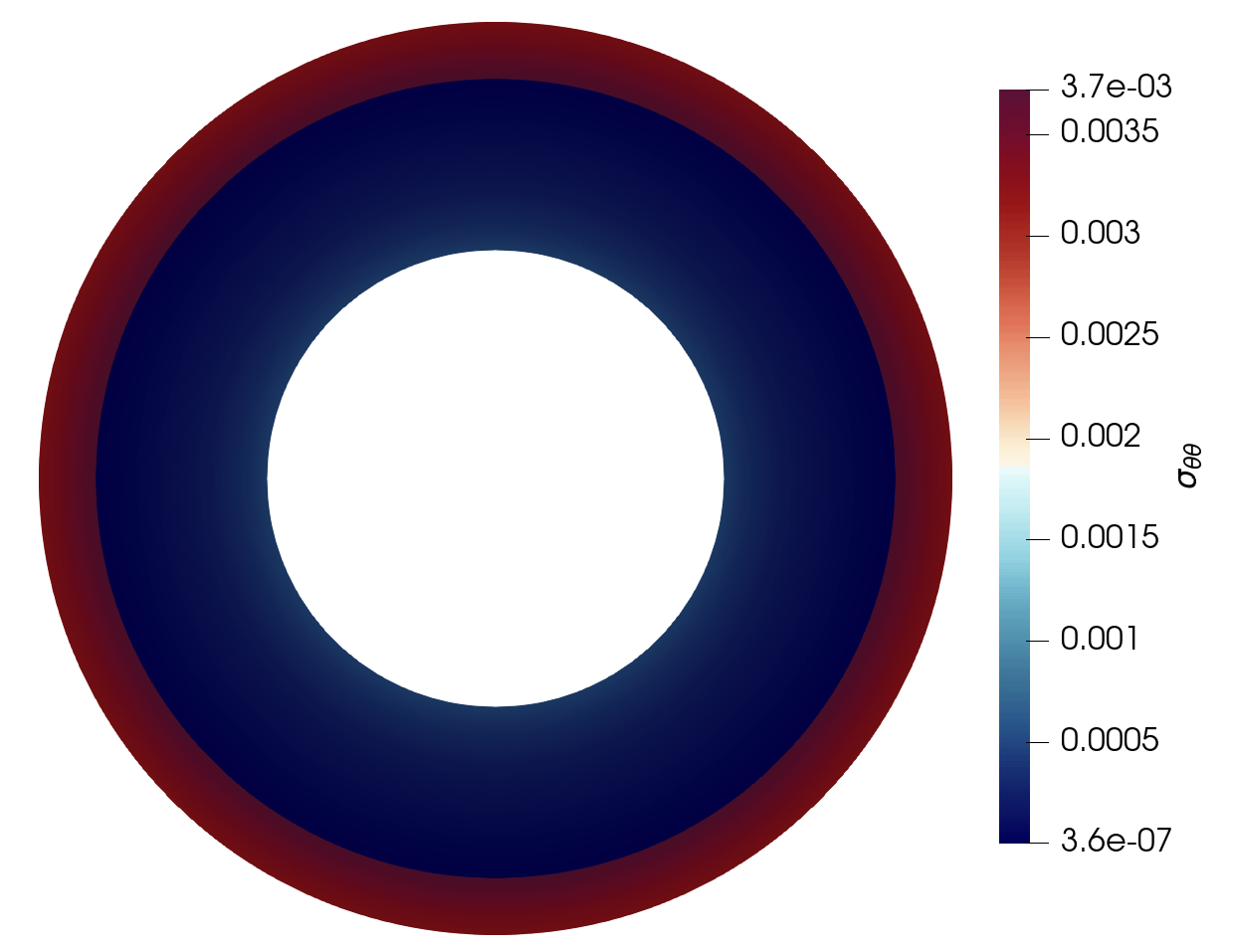}
\end{minipage}
\caption{\label{fig:linear_elasticity_results}Plots of radial displacement $u_r$ (left) and the tangential stress $\sigma_{\theta\theta}$ (right) computed on the fine mesh $\mcS^6(\mesh_H)$, corresponding to $h = 2^{-6}H$, with $q = 4$.}
\end{figure}
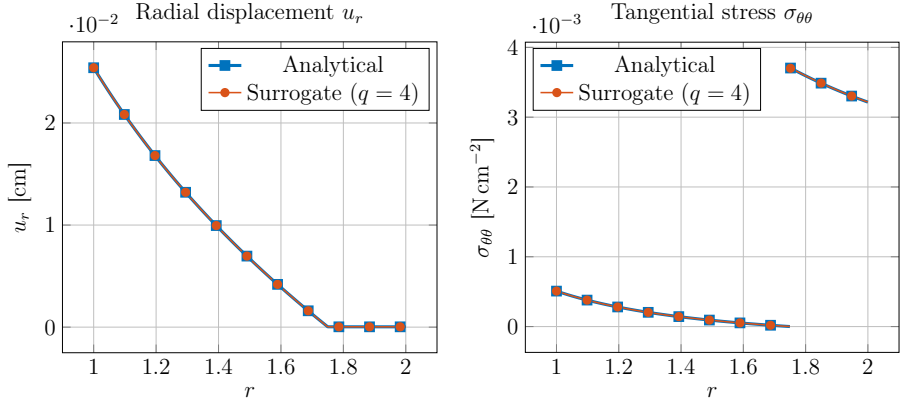
\begin{figure}\centering
\tikzset{font=\large}
\begin{minipage}{0.4\textwidth}
\begin{scaletikzpicturetowidth}{\textwidth}
\begin{tikzpicture}[scale=\tikzscale]
\begin{axis}[
xlabel={$r$},
ylabel={$u_r$ [$\si{\centi\meter}$]},
xmajorgrids,
ymajorgrids,
title={Radial displacement $u_r$},
mark repeat=49,
]
\addplot[color1, mark=square*, ultra thick] table [x index = {0}, y index={1}, col sep=comma] {./results/linearelasticity/plot.csv};
\addlegendentry{Analytical};
\addplot[color2, mark=*, thick] table [x index = {0}, y index={2}, col sep=comma] {./results/linearelasticity/plot.csv};
\addlegendentry{Surrogate ($q = 4$)};
\end{axis}
\end{tikzpicture}
\end{scaletikzpicturetowidth}
\end{minipage}\hspace*{1em}\begin{minipage}{0.4\textwidth}
\begin{scaletikzpicturetowidth}{\textwidth}
\begin{tikzpicture}[scale=\tikzscale]
\begin{axis}[
xlabel={$r$},
ylabel={$\sigma_{\theta\theta}$ [$\si{\newton\per\centi\meter\squared}$]},
xmajorgrids,
ymajorgrids,
title={Tangential stress $\sigma_{\theta\theta}$},
mark repeat=49,
legend style={at={(0.02,0.88)},anchor=west}
]
\addplot[color1, mark=square*, ultra thick, restrict x to domain=1:1.75] table [x index = {0}, y index={3}, col sep=comma] {./results/linearelasticity/plot.csv};
\addplot[color1, mark=square*, ultra thick, restrict x to domain=1.75:2, forget plot] table [x index = {0}, y index={3}, col sep=comma]
{./results/linearelasticity/plot.csv};
\addlegendentry{Analytical};
\addplot[color2, mark=*, thick] table [x index = {0}, y index={4}, col sep=comma, restrict x to domain=1:1.75] {./results/linearelasticity/plot.csv};
\addplot[color2, mark=*, thick] table [x index = {0}, y index={4}, col sep=comma, restrict x to domain=1.75:2, forget plot] {./results/linearelasticity/plot.csv};
\addlegendentry{Surrogate ($q = 4$)};
\end{axis}
\end{tikzpicture}
\end{scaletikzpicturetowidth}
\end{minipage}
\caption{\label{fig:linear_elasticity_lineplots}Plots over line of the radial displacement $u_r$ (left) and the tangential stress $\sigma_{\theta\theta}$ (right) computed on the fine mesh $\mcS^6(\mesh_H)$, corresponding to $h = 2^{-6}H$, with $q = 4$.}
\end{figure}

In order to verify the accuracy of the surrogate method, we select the polynomial degree $q=7$ and the macro-mesh $\mesh_H$, illustrated on the right of \cref{fig:linear_elasticity_domain}.
Note that the discontinuity in the material parameters lies along the macro-element interfaces and so $\lambda,\mu \in \prod_{T\in\mesh_H}C^\infty(T) \subsetneq W^{r+1,\infty}(\mesh_H)$, for any $r>0$.

Each linear system is solved by applying geometric multigrid iterations with V(3,3) cycles until the relative residual is reduced by the factor \num{1e-7}. On the coarsest level used in the multigrid hierarchy, we employ MUMPS as a direct solver. In \cref{tab:littlehconvergence_linearelasticity}, we report on the results for varying $h$ and present the relative $L^2$ errors for the standard and surrogate approach, respectively. On the finest mesh involving about \num{1.8e8} degrees of freedom, the surrogate approach required only 5\% of the time required by the standard approach while having the same accuracy.
That is a speed up by a factor of $20$.
\cref{fig:linear_elasticity_results} shows the radial displacement $u_r$ and the tangential stress $\sigma_{\theta\theta}$ computed with the surrogate approach on the fine mesh $\mcS^6(\mesh_H)$, corresponding to $h = 2^{-6}H$, with $q = 4$.
This is illustrated further by the plots in \cref{fig:linear_elasticity_lineplots} which allow a visual comparison between $u_r$ and $\sigma_{\theta\theta}$ in the surrogate and analytical solutions.
\begin{table}
	\centering
	\caption{\label{tab:littlehconvergence_linearelasticity}Relative $L^2$ errors, experimental orders of convergence, and relative time-to-solutions for fixed $H$, $q = 7$, and varying $h$ in the case of the linearized elasticity problem~\cref{eqn:linearelasticity}.}
			\begin{tabular}{r|r|r|c|c|c|c|c}			\toprule & & & \multicolumn{2}{c|}{standard} & \multicolumn{3}{c}{$q = 7$} \\
			\multicolumn{1}{c|}{$\frac{h}{H}$} & \multicolumn{1}{c|}{$\frac{H_{\mathrm{LS}}}{h}$} & \multicolumn{1}{c|}{DoFs} & rel. $L^2$ err. & eoc & rel. $L^2$ err. & eoc & rtts \\\midrule
			\begin{filecontents*}{./results/linearelasticity/results.csv}
h, HLS, dofs, ref, eoc, tts, qseven, eocseven, ttsseven, relerrseven, relttsseven
$2^{-3}$, $2^{0}$, \num{1.1e+04}, \num{3.93e-04}, -, \num{0.48}, \num{3.93e-04}, -, 0.27, 1.00, 0.57
$2^{-4}$, $2^{0}$, \num{4.5e+04}, \num{9.66e-05}, 2.02, \num{2.82}, \num{9.66e-05}, 2.02, 1.08, 1.00, 0.38
$2^{-5}$, $2^{1}$, \num{1.8e+05}, \num{2.39e-05}, 2.02, \num{7.83}, \num{2.39e-05}, 2.02, 2.45, 1.00, 0.31
$2^{-6}$, $2^{2}$, \num{7.1e+05}, \num{5.92e-06}, 2.01, \num{28.56}, \num{5.92e-06}, 2.01, 5.46, 1.00, 0.19
$2^{-7}$, $2^{2}$, \num{2.8e+06}, \num{1.47e-06}, 2.01, \num{114.55}, \num{1.47e-06}, 2.01, 13.18, 1.00, 0.12
$2^{-8}$, $2^{2}$, \num{1.1e+07}, \num{3.68e-07}, 2.00, \num{518.71}, \num{3.68e-07}, 2.00, 40.44, 1.00, 0.08
$2^{-9}$, $2^{2}$, \num{4.5e+07}, \num{9.18e-08}, 2.00, \num{2440.74}, \num{9.19e-08}, 2.00, 143.47, 1.00, 0.06
$2^{-10}$, $2^{2}$, \num{1.8e+08}, \num{2.30e-08}, 2.00, \num{11397.80}, \num{2.31e-08}, 1.99, 566.27, 1.00, 0.05

\end{filecontents*}
\csvreader[late after line=\\,late after last line=\\\bottomrule,head to column names]{./results/linearelasticity/results.csv}{}
			{\h & \HLS & \dofs & \ref & \eoc & \qseven & \eocseven & \relttsseven}
	\end{tabular}\end{table}

\subsection{$p$-Laplacian diffusion example}
\label{sec:plaplacian_diffusion_example}

In this subsection, we consider the time-dependent example introduced in \Cref{sub:p-Laplacian_diffusion}.
Here, we solve the non-linear $p$-Laplacian diffusion problem~\cref{eqn:nonlinearheatstrong}, given in strong form as
\begin{equation}
\begin{alignedat}{2}
\frac{\partial u}{\partial t} - \mathrm{div}\left(|\nabla u|^{p-2} \cdot u\right) &= f &&\quad \text{in } \Omega \times (0,T]\,,\\
u &= 0 &&\quad \text{on } \partial \Omega \times (0,T]\,,\\
u &= u_0 &&\quad \text{in } \Omega \times \{0\}\,.
\end{alignedat}
\label{eqn:nonlinearheatstrong}
\end{equation}
The computational domain is set to the unit disk, i.e., $\Omega \coloneqq \mathcal{B}_1$ and the right-hand-side is set to a specific constant, $f(x) = 2\hat{q}^{\nicefrac{p}{\hat{q}}}$, where $\hat{q} = \frac{p}{p-1}$. The initial solution is set to $u_0(x) = 0.1 \cdot \left(1 - |x|^2\right)$.
For this particular problem, the stationary limit $u_\infty$ has an analytic solution exists with unit magnitude, namely $u_\infty(x) = 1 - |x|^{\hat{q}}$ \cite[Example~3.1]{barrett1993finite}.

Our discretization follows a standard approach where a mass matrix $\sfM_{ij} = \int_\Omega \phi_i \phi_j\sspace \dd x$ and a stiffness matrix $\sfA_{ij}(\tilde{\sfu}) = \int_\Omega |\nabla \tilde{u}_h|^{p-2}\sspace \nabla\phi_j\cdot\nabla \phi_i\sspace \dd x$ are introduced.
At this point, $\tilde{\sfu}$ is the coefficient vector, in the $\{\phi_i\}$ basis, of an arbitrary discrete function $\tilde{u}_h$.
The time derivative is discretized by a backward Euler scheme, and the non-linearity in each time step is resolved by Picard fixed-point iterations.
Let $\sfu^{l}_k$ be the coefficient vector of the discrete solution at the $k$-th time step and $l$-th fixed point iteration.
Employing the bilinear form~\cref{eq:BilinearFormpLaplace} and fixing a time step size $dt>0$, the discrete problem in each time step $k>0$ and fixed point iteration $l>0$ reads as follows:
\begin{equation}
\Big(\sfM + dt\sspace \sfA\msspace\big(\sfu^{l-1}_k\big)\Big)\sspace\sfu^{l}_k = \sfM \sfu_{k-1} + dt\sspace \sfM \sff\,,
\label{eqn:nonlinearheatdiscrete}
\end{equation}
where $\sfu_{k-1}$ is the final coefficient vector from the previous time step.
In each time step, this system is solved multiple times (once for each fixed-point iteration) by the application of five V(2,2) multigrid cycles.
The fixed-point iterations continue until the relative increment $\frac{\|\sfu^l_k - \sfu^{l-1}_k\|_2}{\|\sfu^l_k\|_2}$ is smaller than the fixed tolerance \num{1e-3}.
Then $k$ is incremented and a new $\sfu_{k-1} = \sfu_{k-1}^l$ is defined.

In our surrogate method, the stencil functions of the stiffness matrix $\sfA\msspace\big(\sfu^{k-1}\big)$ are approximated by solving the least-squares problems after every fixed-point iteration, all the while enforcing the zero row sum property (cf. \Cref{sub:quantitativebenchmark}).
Meanwhile, the stencil function of the mass matrix $\sfM$ is only approximated once in a pre-processing step because it does not depend on any free variables in the computation.
The time step surrogate polynomials of both operators are then simply summed together to obtain the time step matrix $\sfM + dt\sspace \sfA\msspace\big(\sfu^{k-1}\big)$.
This particular splitting of the surrogate matrices, which reproduces the zero row sum property in the stiffness matrix, allows for faster re-approximation of the time step matrix stencil function and appears to improve the stability of the method.
In this example, we did not enforce the symmetry condition featured in~\Cref{sub:the_surrogate_stiffness_matrix}.
Instead, whenever a vertex $x_i\in\bbX_m$ was on the boundary of a macro-element $\bdry T_m$, we set the surrogate stiffness matrix to the exact value stiffness matrix $\tilde{\sfA}_{ij}= \sfA_{ij}$.
This minor asymmetry, is more amenable to computation because there is less data transfer and it did not affect the behavior of our multigrid solver.
In fact, this choice improved our results with this problem, which we believe is due to better accuracy in the surrogate near the singularity in the coefficient; i.e., at the origin $x=(0,0)$.
The success of this approach suggests that the definition given in~\cref{eq:SurrogateStiffnessMatrix} may be relaxed in other applications as well.\footnote{See \cite{bauer2017two} for further evidence.}
In the proximity of this singularity and for $p > 2$, the coefficient depending on the solution of the previous fixed-point iteration is getting very close to zero which serves as a challenge for the approximated off-diagonal stencil functions.
Depending on the polynomial degree $q$, they might erroneously take on positive values due to overshoots which possibly results in a loss of positive definiteness of the surrogate matrix.
However, this drawback could not be observed in the scenario considered in the following example.

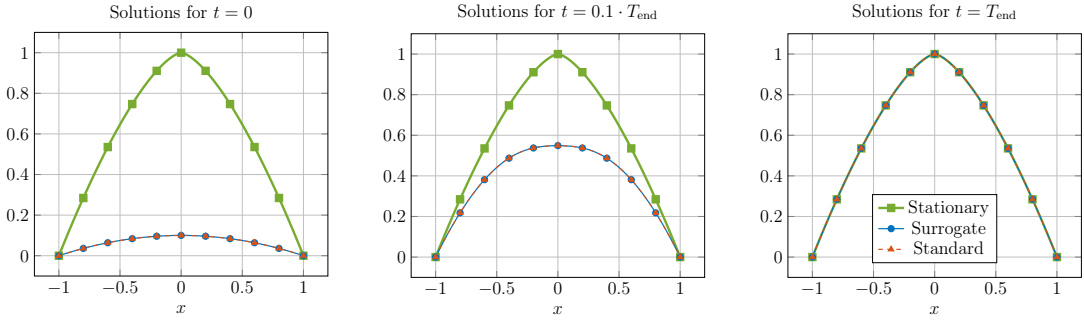
\begin{figure}\centering
\tikzset{font=\large}
\begin{minipage}{0.3\textwidth}
\begin{scaletikzpicturetowidth}{\textwidth}
\begin{tikzpicture}[scale=\tikzscale]
\begin{axis}[
xlabel={$x$},
xmajorgrids,
ylabel={},
ymajorgrids,
ymin=-0.1,
ymax=1.1,
title={Solutions for $t = 0$},
mark repeat=10,
]
\addplot[color5,mark=square*,ultra thick] table [x index = {4}, y index={1}, col sep=comma] {./results/nonlinearheatequation/pthree_surrogate.0.csv};
\addplot[color1,mark=*,thick] table [x index = {4}, y index={0}, col sep=comma] {./results/nonlinearheatequation/pthree_surrogate.0.csv};
\addplot[color2,mark=triangle*,thick,dashed,mark options={solid}] table [x index = {4}, y index={0}, col sep=comma] {./results/nonlinearheatequation/pthree_standard.0.csv};
\end{axis}
\end{tikzpicture}
\end{scaletikzpicturetowidth}
\end{minipage}\hfill
\begin{minipage}{0.3\textwidth}
\begin{scaletikzpicturetowidth}{\textwidth}
\begin{tikzpicture}[scale=\tikzscale]
\begin{axis}[
xlabel={$x$},
xmajorgrids,
ylabel={},
ymajorgrids,
ymin=-0.1,
ymax=1.1,
title={Solutions for $t = 0.1 \cdot T_{\mathrm{end}}$},
mark repeat=10,
]
\addplot[color5,mark=square*,ultra thick] table [x index = {4}, y index={1}, col sep=comma] {./results/nonlinearheatequation/pthree_surrogate.10.csv};
\addplot[color1,mark=*,thick] table [x index = {4}, y index={0}, col sep=comma] {./results/nonlinearheatequation/pthree_surrogate.10.csv};
\addplot[color2,mark=triangle*,thick,dashed,mark options={solid}] table [x index = {4}, y index={0}, col sep=comma] {./results/nonlinearheatequation/pthree_standard.10.csv};
\end{axis}
\end{tikzpicture}
\end{scaletikzpicturetowidth}
\end{minipage}\hfill
\begin{minipage}{0.3\textwidth}
\begin{scaletikzpicturetowidth}{\textwidth}
\begin{tikzpicture}[scale=\tikzscale]
\begin{axis}[
xlabel={$x$},
xmajorgrids,
ylabel={},
ymajorgrids,
ymin=-0.1,
ymax=1.1,
title={Solutions for $t = T_{\mathrm{end}}$},
mark repeat=10,
legend style={at={(0.5,0.2)},anchor=center}
]
\addplot[color5,mark=square*,ultra thick] table [x index = {4}, y index={1}, col sep=comma] {./results/nonlinearheatequation/pthree_surrogate.100.csv};
\addlegendentry{Stationary};
\addplot[color1,mark=*,thick] table [x index = {4}, y index={0}, col sep=comma] {./results/nonlinearheatequation/pthree_surrogate.100.csv};
\addlegendentry{Surrogate};
\addplot[color2,mark=triangle*,thick,dashed,mark options={solid}] table [x index = {4}, y index={0}, col sep=comma] {./results/nonlinearheatequation/pthree_standard.100.csv};
\addlegendentry{Standard};
\end{axis}
\end{tikzpicture}
\end{scaletikzpicturetowidth}
\end{minipage}
\caption{\label{fig:nonlinear_pthree} Plots of standard, surrogate and stationary analytical solution over the line $[0,1] \times \{0\}$ for different times $t = 0$, $t = 0.1 \cdot T$, and $t = T_{\mathrm{end}}$.}
\end{figure}
\begin{figure}\centering
\begin{minipage}{0.5\textwidth}
\includegraphics[width=\textwidth]{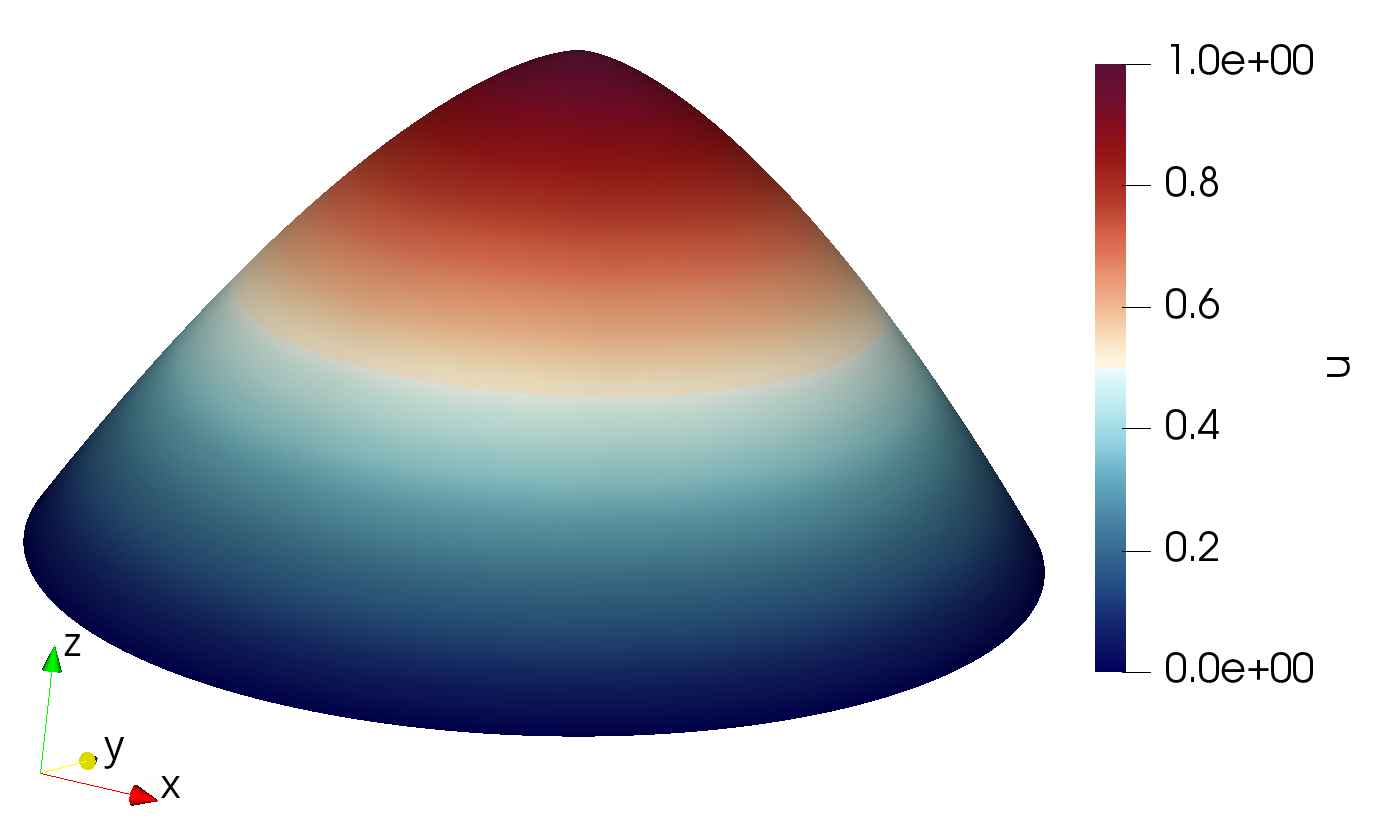}
\end{minipage}\begin{minipage}{0.5\textwidth}
	\includegraphics[width=\textwidth]{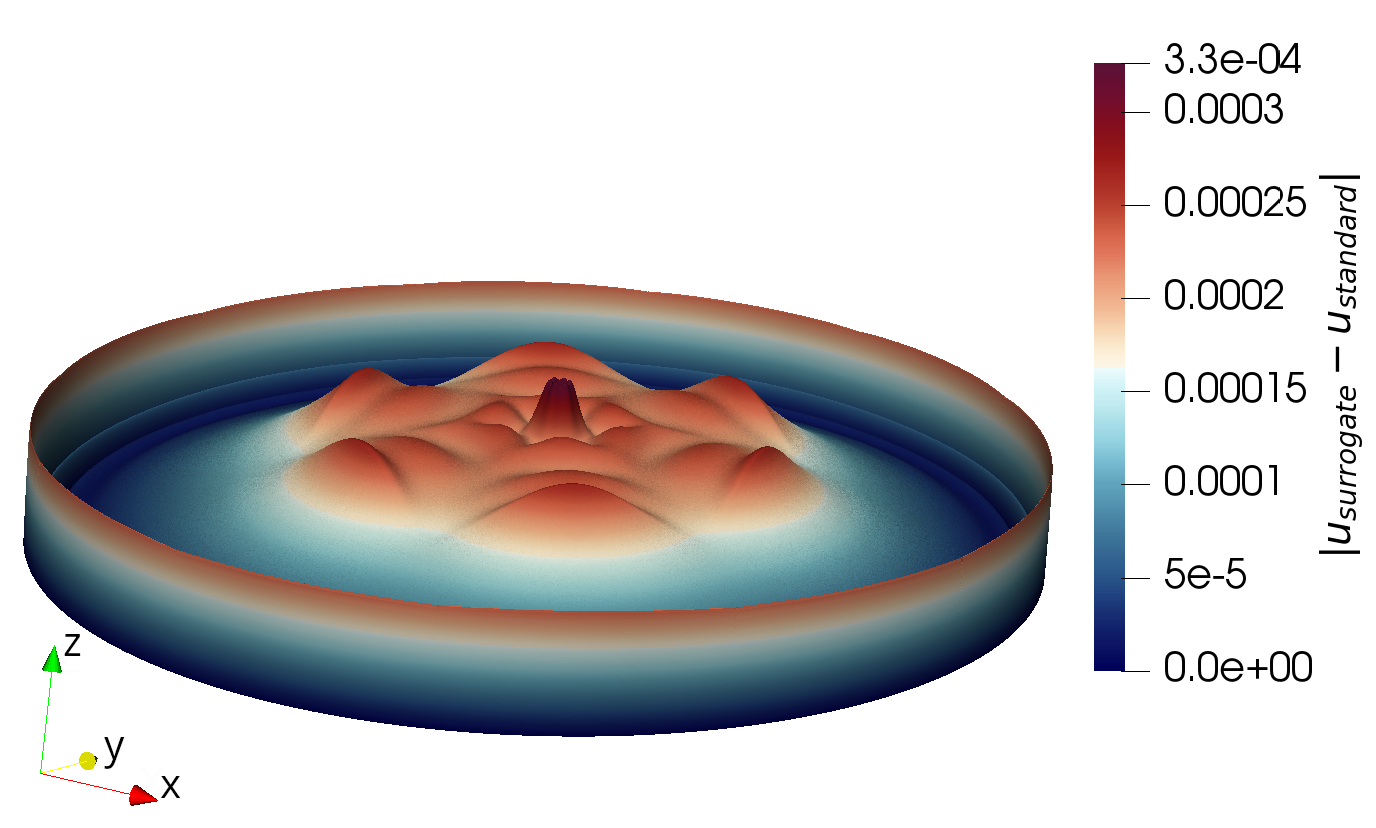}
\end{minipage}
\caption{\label{fig:nonlinear_surfaceplots}Surface plot of non-stationary $p$-Laplacian surrogate solution for $t = T_{\mathrm{end}}$ with $p=3$ and $q=6$ (left). Absolute difference between the discrete standard and surrogate solution at the time $t = T_{\mathrm{end}}$ (right).}
\end{figure}

The unit-disk is discretized by the macro-mesh $\mesh_H$ featured on the right of \cref{fig:stencilweight_surfaceplot}.
Note that the vertices of the central macro-elements meet at the origin $x=(0,0)$; i.e., exactly where the singularity occurs in the stationary limit $u_\infty$.
The simulations are conducted on the mesh $\mcS^9(\mesh_H)$, which involves about \num{4.72e6} degrees of freedom.
The macro-elements adjacent to the boundary are mapped to the physical geometry by using the mapping described in \cite{zlamal1973curved}. Furthermore, the least-squares regressions are carried out on the mesh corresponding to $H_{\mathrm{LS}} = 4h$ and the polynomial degree of the approximated stencil functions is fixed to $q = 6$.
In this scenario, we consider the $p$-Laplacian operator with $p = 3$, fix the time step size $\Delta t = \num{1e-2}$, and solve until time $T_{\mathrm{end}} = 1$.
\Cref{fig:nonlinear_pthree} illustrates the standard and surrogate solutions plotted over the line $[0,1] \times \{0\}$ for different times $t$. In the left of \cref{fig:nonlinear_surfaceplots}, the surface plot of the surrogate solution is depicted. Since the difference of the solutions is very small, we added in the right of \cref{fig:nonlinear_surfaceplots} a surface plot of the absolute difference of the surrogate and standard solution at the final time $t = T_{\mathrm{end}}$.
The surrogate approach required only $5.4$\% of the time required by the standard approach.
That is, a speed-up by more than a factor of $18$.

\appendix
\section{Proofs} \label{app:appendix}

\begin{proof}[Proof of \Cref{prop:SpectralConvergence}]
By the min-max theorem \cite{courant1953methods}, the $k$-th eigenvalue of $\sfM$ is
\begin{align}
	\lambda_k(\sfM)
	&=
	\min_{W\subset \R^N}
	\Bigg\{
		\max_{\|\sfx\|_2 = 1}
		\Big\{
			\sfx^\top \sfM\sfx : \sfx \in W
		\Big\}
		:
		\mathrm{dim}\,W=k
	\Bigg\}
	\\
	&=
	\max_{W\subset \R^N}
	\Bigg\{
		\min_{\|\sfx\|_2 = 1}
		\Big\{
			\sfx^\top \sfM\sfx : \sfx \in W
		\Big\}
		:
		\mathrm{dim}\,W=N-k+1
	\Bigg\}
	\,.
\end{align}

Define $\sfD = \sfM - \sfN$.
We first show that $\lambda_1(\sfD) \leq  \lambda_k(\sfM) - \lambda_k(\sfN) \leq \lambda_N(\sfD)$.
Indeed,
\begin{align}
	\lambda_k(\sfM)
												&\leq
	\min_{W\subset \R^N}
	\Bigg\{
		\max_{\|\sfx\|_2 = 1}
		\Big\{
			\sfx^\top \sfN\sfx : \sfx \in W
		\Big\}
		+
		\max_{\|\sfx\|_2 = 1}
		\Big\{
			\sfx^\top \sfD\sfx : \sfx \in W
		\Big\}
		:
		\mathrm{dim}\,W=k
	\Bigg\}
	\\
	&\leq
	\min_{W\subset \R^N}
	\Bigg\{
		\max_{\|\sfx\|_2 = 1}
		\Big\{
			\sfx^\top \sfN\sfx : \sfx \in W
		\Big\}
		:
		\mathrm{dim}\,W=k
	\Bigg\}
	+
	\max_{\|\sfx\|_2 = 1}
	\Big\{
		\sfx^\top \sfD\sfx : \sfx \in \R^N
	\Big\}
	\\
	&=
	\lambda_k(\sfN) + \lambda_N(\sfD)
\end{align}
and, likewise,
\begin{align}
	\lambda_k(\sfM)
												&\geq
	\max_{W\subset \R^N}
	\Bigg\{
		\min_{\|\sfx\|_2 = 1}
		\Big\{
			\sfx^\top \sfN\sfx : \sfx \in W
		\Big\}
												+
		\min_{\|\sfx\|_2 = 1}
		\Big\{
			\sfx^\top \sfD\sfx : \sfx \in W
		\Big\}
		:
		\mathrm{dim}\,W=N-k+1
	\Bigg\}
	\\
	&\geq
	\max_{W\subset \R^N}
	\Bigg\{
		\min_{\|\sfx\|_2 = 1}
		\Big\{
			\sfx^\top \sfN\sfx : \sfx \in W
		\Big\}
		:
		\mathrm{dim}\,W=N-k+1
	\Bigg\}
						+
	\min_{\|\sfx\|_2 = 1}
	\Big\{
		\sfx^\top \sfD\sfx : \sfx \in \R^N
	\Big\}
	\\
	&=
	\lambda_k(\sfN) + \lambda_1(\sfD)
		\,.
\end{align}
This immediately leads us to the inequality $|\lambda_k(\sfM) - \lambda_k(\sfN)| \leq \max\{|\lambda_1(\sfD)|,|\lambda_N(\sfD)|\}$.
Now, for at least one $i$, $|\lambda_N(\sfD)| - |\sfD_{ii}| \leq |\lambda_N(\sfD) - \sfD_{ii}| \leq \sum_{j\neq i} |\sfD_{ij}|$, by the Gershgorin circle theorem.
Therefore, $|\lambda_N(\sfD)| \leq \sum_{j} |\sfD_{ij}| \leq \|\sfD\|_{\infty}$.
Similarly, $|\lambda_1(\sfD)| \leq \|\sfD\|_{\infty}$.
\end{proof}

\section*{Acknowledgments}
This project has received funding from the European Union's Horizon 2020 research and innovation programme under grant agreement No 800898.
This work was also partly supported by the German Research Foundation through
the Priority Programme 1648 "Software for Exascale Computing" (SPPEXA) and by grant WO671/11-1.
The authors gratefully acknowledge the Gauss Centre for Supercomputing e.V. (GCS, \href{www.gauss-centre.eu}{www.gauss-centre.eu}) for funding this project by providing computing time on the GCS supercomputer SuperMUC at Leibniz Supercomputing Centre (LRZ, \href{www.lrz.de}{www.lrz.de}).

\phantomsection\bibliographystyle{siam}
\bibliography{main}

\end{document}